\documentclass[preprint,11pt]{elsarticle}
\usepackage{amsfonts}
\usepackage{amsmath}
\usepackage{mathrsfs}
\usepackage[applemac]{inputenc}
\usepackage{times, ulem, amsfonts}
\usepackage{mathpazo, amsmath, amssymb, amsthm,color}

\def\N{{\mathbb N}}

\newcommand\diverg{\mathop{\mbox{\rm div}}}
\theoremstyle{plain}

\numberwithin{equation}{section}
\newtheorem{theorem}{Theorem}[section]
\newtheorem{proposition}[theorem]{Proposition}
\newtheorem{Lemma}[theorem]{Lemma}
\newtheorem{corollary}[theorem]{Corollary}
\newtheorem{definition}[theorem]{Definition}

\theoremstyle{remark}
\newtheorem{remark}[theorem]{Remark}

\newcommand{\R}{{\mathbb R}}

\newcommand\Z{{\mathbb{Z}}}

\begin{document}
\bibliographystyle{plainmma}
\baselineskip=24pt

\begin{center}
{\Large\bf Global well-posedness for the 3D incompressible inhomogeneous Navier-Stokes equations and MHD equations }\\[2ex]
\footnote{$^{**}$
Email Address: pingxiaozhai@163.com (XP ZHAI); \
 mcsyzy@mail.sysu.edu.cn (Z. YIN).}
Xiaoping $\mbox{Zhai}^{1}$, Zhaoyang $\mbox{Yin}^{1,2}$\\[2ex]
$^1\mbox{Department}$ of Mathematics, Sun Yat-sen University,
\\[0.5ex]
 Guangzhou, Guangdong 510275, China
 \\[0.5ex]
$^2\mbox{Faculty}$ of Information Technology,\\ Macau University of Science and Technology, Macau, China
\end{center}

\bigskip
\centerline{\bf Abstract}
The present paper is dedicated to the global well-posedness   for the  3D inhomogeneous incompressible Navier-Stokes equations,  in critical Besov spaces  without smallness assumption on the variation of the density. We
aim at extending the work by Abidi, Gui and Zhang (Arch. Ration. Mech. Anal. 204 (1):189--230, 2012,  and J. Math. Pures Appl. 100 (1):166--203, 2013) to a more lower regularity index about the initial velocity. The key to that improvement is a new a priori estimate for an elliptic equation with nonconstant coefficients in Besov spaces which have the same degree as $L^2$ in $\R^3$. Finally, we also generalize our well-posedness result to the inhomogeneous incompressible MHD equations.



\section{Introduction and the main results}
The first part of this paper is devoted to studying  the Cauchy problem of the 3D  incompressible inhomogeneous Navier-Stokes equations in critical Besov spaces
\begin{eqnarray}\label{nsmoxing}
\left\{\begin{aligned}
&\partial_t\rho+u\cdot\nabla \rho =0,\\
&\rho(\partial_t u+u\cdot\nabla u)-\diverg\big(2\mu(\rho)\mathcal{M}(u)\big)+\nabla\Pi=0,\\
&\diverg u =0,\\
&(\rho,u)|_{t=0}=(\rho_0,u_0),
\end{aligned}\right.
\end{eqnarray}
where $\rho$, $u=(u_{1},u_{2},u_{3})$, stand for the density, velocity, $\mathcal{M}(u)=\frac{1}{2}(\partial_{i}u_{j}+\partial_{j}u_{i})$, $\Pi$ is a scalar pressure function, the viscosity coefficient $\mu(\rho)$ is smooth, positive on $[0,\infty)$. This system
describes a fluid which is obtained by mixing two miscible fluids that are incompressible and that have different
densities. It may also describe a fluid containing a melted substance. One may check  \cite{lions} for the detailed derivation.

A lot of recent works have been dedicated to the mathematical study of the above system. In the case of smooth data with no vacuum, Lady$\check{z}$enskaja and
Solonnikov   first addressed
in \cite{ladyzenskaja} the question of unique solvability of \eqref{nsmoxing}, similar results were obtained by Danchin \cite{danchin2004} in $\R^n$
with initial data in the almost critical Sobolev spaces.
Global weak solutions with finite energy were constructed by Simon in \cite{simon} (see also
the book by Lions \cite{lions} for the variable viscosity case). Yet the regularity and uniqueness of such weak solutions are big open problems.
 Recently, Danchin and Mucha \cite{danchin2013} proved by using a Lagrangian approach that the system \eqref{nsmoxing} has a unique local solution with initial data $(\rho_0, u_0)\in L^{\infty}(\R^n)\times H^2(\R^n)$ if initial vacuum does not occur.
  When the density $\rho$ is away from zero, we denote by $a =\frac{1}{\rho}-1 $ and $\mu(\rho)=1$ which allow us to work with the following system:
\begin{eqnarray}\label{moxing}
\left\{\begin{aligned}
&\partial_t a+u\cdot\nabla a =0,\\
&\partial_t u+u\cdot\nabla u-(1+a)\Delta u+(1+a)\nabla\Pi=0,\\
&\diverg u=0,\\
&(a,u)|_{t=0}=(a_0,u_0).
\end{aligned}\right.
\end{eqnarray}
Similar to the classical Navier-Stoeks equations, the above system also has a scaling.
It is easy to see that the transformations:
$$(a_\lambda, u_\lambda )(t,x)=(a(\lambda^2\cdot,\lambda\cdot),\lambda u(\lambda^2\cdot,\lambda\cdot))$$
have that property, provided that the pressure term $\Pi$ and the initial data  have been changed accordingly.
We can also verify that the  space $\dot{B}_{q,1}^{\frac{n}{q}}({\mathbb R} ^n) \times\dot{B}_{p,1}^{-1+\frac{n}{p}}({\mathbb R} ^n)$   is the critical space for the system.

 When the initial density is close enough to a positive constant, Danchin in \cite{danchin2003} proved that if initial data $a_0\in\dot{B}_{2,\infty}^{\frac{n}{2}}({\mathbb R} ^n)\cap L^\infty(\R^n)$, $u_0\in \dot{B}_{2,1}^{-1+\frac{n}{2}}({\mathbb R} ^n)$, then the system (\ref{moxing}) has a unique local-in-time solution.  This result was improved  by Abidi in \cite{abidi2007+1}, where he proved that if initial  data $a_0\in\dot{B}_{p,1}^{\frac{n}{p}}({\mathbb R} ^n)$, $u_0\in \dot{B}_{p,1}^{-1+\frac{n}{p}}({\mathbb R} ^n)$ is small enough $1<p<2n, $  then (\ref{moxing}) has a global solution, moreover, this solution is unique if $1<p\le  n.$  The results  in \cite{abidi2007+1}, \cite{danchin2003}, \cite{danchin2004}   also were  improved by Abidi and Paicu in \cite{abidi2007} to a more general case $i,e.$ if $a_0\in\dot{B}_{q,1}^{\frac{n}{q}}({\mathbb R} ^n)$, $u_0\in \dot{B}_{p,1}^{-1+\frac{n}{p}}({\mathbb R} ^n)$ for $p, q$ satisfying some technical assumptions.
Very recently, Danchin and Mucha \cite{danchin2012} improved the uniqueness result in
\cite{abidi2007+1} for $p\in(n,2n)$ through Lagrange approach.  Huang, Paicu and Zhang in \cite{huangjingchi} first proved the global existence of weak solutions to the system (\ref{moxing}). The regularity of the initial velocity in \cite{huangjingchi} is critical to the scaling of this system and is general enough to generate non-Lipschitz velocity fields. Furthermore, with additional regularity assumptions on the initial velocity or on the initial density, they also proved the uniqueness of such a solution.
 Paicu and Zhang in \cite{paicu2012} could also get the global well-posedness only under the assumption that horizontal components of the initial velocity are small exponentially small compared with the  third component of the initial velocity. Chemin, Paicu and Zhang in \cite{chemin2014} generalized the result in \cite{paicu2012} to the critical anisotropic  Besov spaces.

 When the initial density is not close enough to a positive constant, Abidi, Gui and Zhang in \cite{abidi2012} firstly proved the global well-posedness of (\ref{moxing})  in the energy space if the initial $\|u_0\|_{\dot{B}_{2,1}^{\frac{1}{2}}}$ is small enough. Lately, this  results was improved in \cite{abidi2013} to the critical  $L^p$ framework.
 More precisely,  Abidi, Gui and Zhang in \cite{abidi2013} proved the following theorem:
\begin{theorem}\label{}
 Let $q\in[1,2]$, $p\in[3,4]$ and $\frac1p+\frac1q>\frac56$, $\frac1q-\frac1p\le\frac13$,
$(a_0,u_0)\in{B}_{q,1}^{\frac{3}{q}}(\R^3)\times\dot{B}_{p,1}^{-1+\frac3p}(\R^3)$
with $\diverg u_0=0$ and $1+\inf_{x\in\R^3}a_0(x)\ge \kappa>0$.  Then \eqref{moxing} has a  unique local solution $(a,u,\nabla\Pi)$ on $[0,T]$ such that
\begin{align}
a\in& C([0,T];{B}_{q,1}^{\frac{3}{q}}(\R^3))\cap \widetilde{L}_T^\infty(\dot{B}_{q,1}^{\frac{3}{q}}(\R^3)),
\ \nabla\Pi\in L_T^1(\dot{B}_{p,1}^{-1+\frac3p}(\R^3)),\nonumber\\
u\in& C([0,T];\dot{B}_{p,1}^{-1+\frac3p}(\R^3))\cap \widetilde{L}_T^\infty(\dot{B}_{p,1}^{-1+\frac3p}(\R^3))
\cap L_T^1(\dot{B}_{p,1}^{1+\frac3p}(\R^3)).
\end{align}
Moreover,  if there exists a small constant $c$ depending $\|a_0\|_{{B}_{q,1}^{\frac{3}{q}}}$
such that
$$\|u_0\|_{\dot{B}_{p,1}^{-1+\frac3p}}\le c,$$
then \eqref{moxing} has a global solution $(a,u,\nabla\Pi)$ such that for any $t>0$\rm{:}
\begin{align}
&\|a\|_{\widetilde{L}_t^\infty({B}_{q,1}^{\frac{3}{q}})}+\|u\|_{\widetilde{L}_t^\infty(\dot{B}_{p,1}^{-1+\frac3p})}+\|u\|_{L^1_t(\dot{B}_{p,1}^{1+\frac3p})}
+\|\nabla\Pi\|_{L^1_t(\dot{B}_{p,1}^{1+\frac3p})}\nonumber\\
\lesssim&C(\|a_0\|_{B_{p,1}^{\frac{3}{q}}}+\|u_0\|_{\dot{B}_{p,1}^{-1+\frac{3}{p}}})\exp\left\{C\exp\big(Ct^\frac12\big)\right\},
\end{align}
for some time independent of constant $C$.
\end{theorem}
The main purpose of this paper is to improve the  well-posedness results from \cite{abidi2012, abidi2013}. In particular, we want to improve the index $p$ in \cite{abidi2013} to an ideal range, $i.e.$ $1<p<6$. The main difficulty in \cite{abidi2013} is to deal with the pressure function. Their main observations are that  with $p, q$ satisfying the restrictions $p\in[3,4], q\in[1,2]$, $\frac1q+\frac1p>\frac56$ and $\frac1q-\frac1p\le\frac13,$
they can use the $L^2$ estimate of $\nabla\Pi$. One can sketch more details  about the estimate $\|\nabla\Pi\|_{L^2}$ in \cite{abidi2012, abidi2013}.
In the present paper, we will instead  using $\|\nabla\Pi\|_{L^2}$ by $\|\nabla\Pi\|_{\dot{B}_{p,2}^{-\frac32+\frac3p}}$. This is  reasonable for the two spaces have the same degree in $\R^3$.
By using  a new commutator  of integral type, we can get the solution mapping $\mathcal{H}_a:F\mapsto\nabla\Pi$ to the 3D elliptic equation
$\mathrm{div}\big((1+a)\nabla\Pi\big)=\mathrm{div} F$ is bounded on $\dot{B}_{p,2}^{\frac3p-\frac32}(\mathbb{R}^3)$. More precisely, we prove, if  $1< q\le p$ with $p\in(\frac{1+\sqrt{17}}{4},\frac{5+\sqrt{17}}{2})$ and $\frac1p+\frac1q\geq\frac12$, $a\in\dot{B}_{q,1}^{\frac{3}{q}}(\R^3)$ with $$1+a\geq\kappa>0,$$
then
$$\|\nabla\Pi\|_{\dot{B}_{p,2}^{\frac3p-\frac32}}\leq C\big(1+\|a\|_{\dot{B}_{q,1}^{\frac3q}}\big)\|F\|_{\dot{B}_{p,1}^{\frac3p-\frac32}}.$$
However, it's a pity that we still cannot fill the gap $p$ in $[\frac{5+\sqrt{17}}{2},6)$.

The first main result of the present paper is stated in the following theorem:
\begin{theorem}\label{zhuyaodingli1}
 Let $1< q\le p$ with $p\in(1,\frac{5+\sqrt{17}}{2})$ and $\frac12-\frac1p\le\frac1q\le\frac1p+\frac13,$ $(a_0,u_0)\in\dot{B}_{q,1}^{\frac{3}{q}}(\R^3)\times\dot{B}_{p,1}^{-1+\frac3p}(\R^3)$
with $\diverg u_0=0$ and $1+\inf_{x\in\R^3}a_0(x)\ge \kappa>0$.  Then \eqref{moxing} has a  local solution $(a,u,\nabla\Pi)$ on $[0,T]$ such that
\begin{align}
a\in& C([0,T];{B}_{q,1}^{\frac{3}{q}}(\R^3))\cap \widetilde{L}_T^\infty(\dot{B}_{q,1}^{\frac{3}{q}}(\R^3)),
\ \nabla\Pi\in L_T^1(\dot{B}_{p,1}^{-1+\frac3p}(\R^3)),\nonumber\\
u\in& C([0,T];\dot{B}_{p,1}^{-1+\frac3p}(\R^3))\cap \widetilde{L}_T^\infty(\dot{B}_{p,1}^{-1+\frac3p}(\R^3))
\cap L_T^1(\dot{B}_{p,1}^{1+\frac3p}(\R^3)).
\end{align}
Especially, if $a_0\in{B}_{q,1}^{\frac{3}{q}}(\R^3)$, $p\in[3,4], q\in[1,2]$, this solution is unique.
Moreover,  if  there exists a small constant $c$ depending $\|a_0\|_{{B}_{q,1}^{\frac{3}{q}}}$
such that
$$\|u_0\|_{\dot{B}_{p,1}^{-1+\frac3p}}\le c,$$
then \eqref{moxing} has a global solution $(a,u,\nabla\Pi)$ such that for any $t>0$\rm{:}
\begin{align}
&\|a\|_{\widetilde{L}_t^\infty({B}_{q,1}^{\frac{3}{q}})}+\|u\|_{\widetilde{L}_t^\infty(\dot{B}_{p,1}^{-1+\frac3p})}+\|u\|_{L^1_t(\dot{B}_{p,1}^{1+\frac3p})}
+\|\nabla\Pi\|_{L^1_t(\dot{B}_{p,1}^{1+\frac3p})}\nonumber\\
\lesssim&C(\|a_0\|_{B_{p,1}^{\frac{3}{q}}}+\|u_0\|_{\dot{B}_{p,1}^{-1+\frac{3}{p}}})\exp\left\{C\exp\big(Ct^\frac12\big)\right\},
\end{align}
for some time independent of constant $C$.
\end{theorem}

\begin{remark}
  The  existence of our theorem  only requires that the initial density belongs to the homogeneous Besov space $\dot{B}_{q,1}^{\frac{3}{q}}(\R^3)$  and also extends the $p,q$ in \cite{abidi2013} to be a more larger range. The  uniqueness of our theorem   has removed the technical assumption $\frac1p+\frac1q>\frac56$  when $p\in[3,4], q\in[1,2]$ in \cite{abidi2013}. The key to this  improvement
 is that we donot need to make the $L^2$ estimate for the pressure.
  Thus, we will not use the $L^2$ estimate for the nonlinear terms in the  momentum equation.
\end{remark}
\begin{remark}
Our main result also allows the initial velocity field and initial density function to be highly oscillatory just as in  \cite{abidi2012,abidi2013}.
\end{remark}

We also generalize the above result to the 3D  incompressible inhomogeneous MHD system \cite{davidson} with variable
electrical conductivity which has the following form:
\begin{eqnarray}\label{mhddmoxing}
\left\{\begin{aligned}
&\partial_t\rho+u\cdot\nabla \rho =0,\\
&\rho(\partial_t u+u\cdot\nabla u)-\Delta u+\nabla\Pi=B\cdot\nabla B,\\
&\partial_t B+u\cdot\nabla B-\diverg\left(\frac{\nabla B}{\sigma(\rho)}\right)=B\cdot\nabla u,\\
&\diverg u =\diverg B=0,\\
&(\rho,u,B)|_{t=0}=(\rho_0,u_0,B_0),
\end{aligned}\right.
\end{eqnarray}
where $\rho$ is the density and $u$ is the velocity field, $B$
is the magnetic field, $\mathcal{M}(u)=\frac12(\nabla u+\nabla u^T)$ is the symmetrical part of the gradient, $\Pi(x,t)$ is the scalar pressure,
 $\sigma(\rho)>0$ is the electrical conductivity of the field. Moreover, we suppose that $\sigma(\rho)$ is a $C^{\infty}$ function and that
$
0<\underline{\sigma}\le\frac{1}{\sigma(\rho)}
\le\bar{\sigma}<\infty.$

Similar to \eqref{moxing},  when the density $\rho$ is away from zero, we can also use the transform
$a =\frac{1}{\rho}-1 $  to change \eqref{mhddmoxing} into the following system:
\begin{eqnarray}\label{mhdmoxing}
\left\{\begin{aligned}
&\partial_t a+u\cdot\nabla a =0,\\
&\partial_t u+u\cdot\nabla u-(1+a)\Delta u+(1+a)\nabla\Pi=(1+a)B\cdot\nabla B,\\
&\partial_t B+u\cdot\nabla B-\diverg\left(\widetilde{\sigma}(a)\nabla B\right)=B\cdot\nabla u,\\
&\diverg u =\diverg B=0,\\
&(a,u,B)|_{t=0}=(a_0,u_0,B_0),
\end{aligned}\right.
\end{eqnarray}
where  $\widetilde{\sigma}(a)=\frac{1}{\sigma(\frac{1}{1+a})}$ is a  smooth function.

Compared with the Navier-Stokes equations,
 the dynamic motion of the fluid and the magnetic field interact on each other and both the hydrodynamic and electrodynamic effects in the motion are strongly coupled, the problems of MHD system are considerably more complicated. Even through, in the past several years, there are also many mathematical results related to the
incompressible MHD system (see \cite{cao2010}, \cite{desjardins}, \cite{duvaut}, \cite{hecheng3}, \cite{linfanghua}, \cite{miaochangxing}, \cite{sermange}). For the homogeneous viscous incompressible MHD system (i.e. $\rho(t, x)$= constant), Duvaut and Lions \cite{duvaut} established the local existence and uniqueness of
solution in the classical Sobolev space $H^s(\mathbb{R}^n), s \ge n,$ they also proved the
global existence of solutions to this system with small initial data. Sermange and Temam \cite{sermange}
proved the global unique solution in $\R^2$. With mixed partial dissipation and additional magnetic diffusion in the two-dimensional MHD system, Cao and Wu \cite{caochongsheng}  proved that such a system is globally well-posed for any data in $H^2(\mathbb{R}^2)$. In a recent remarkable paper  Lin, Xu and Zhang \cite{linfanghua2013}   proved the global existence of smooth solution of the 2-D MHD system  around the trivial solution
$(x_2,0)$ (see \cite{linfanghua} for 3-D case). In \cite{renxiaoxia}, Ren, Wu, Xiang and Zhang  got the global existence and the decay estimates of small smooth solution for the 2-D MHD equations without
magnetic diffusion.
When the fluid is nonhomogeneous, Gerbeau and Le Bris \cite{gerbeau} (see also Desjardins and Le Bris \cite{desjardins}) studied the global existence of weak solutions of finite energy in the whole space or in the torus.  Abidi and Paicu \cite{abidi2008} established the global existence of strong solutions with small initial data in the critical Besov spaces. Moreover, they allowed variable viscosity and conductivity coefficients
but required an essential assumption that there is no vacuum (more precisely, the initial data are
closed to a constant state). Zhai, Li and Yan in \cite{zhaixiaoping} also considered the global well-posedness for \eqref{moxing} in critical Besov spaces.  By using the  Gagliardo-Nirenberg inequality, they obtained the global existence  for this system without any small conditions imposed on the third components of the initial velocity field and magnetic field, which can be regarded as an improvement of \cite{abidi2008}.
 Chen, Tan, and Wang \cite{chenqing} extended the local existence in presence of vacuum by using the Galerkin method, energy method and the domain expansion technique. Lately , with initial data satisfying some compatibility conditions, by using a critical Sobolev inequality of logarithmic type, Huang and Wang \cite{huangxiangdi} got the global strong solution to the 2-D nonhomogeneous incompressible MHD system. Recently,  Gui \cite{guiguilong} studied the Cauchy problem of the 2-D magnetohydrodynamic system with inhomogeneous density and electrical conductivity. He showed that this system with a constant viscosity is globally well-posed for a generic family of the variations of the initial data and an inhomogeneous electrical conductivity. Moreover, He established that the system is globally well-posed  if the electrical conductivity is homogeneous.

The second main result of this paper is stated in the following theorem:
\begin{theorem}\label{zhuyaodingli}
 Let $1< q\le p$ with $p\in(1,\frac{5+\sqrt{17}}{2})$ and $\frac12-\frac1p\le\frac1q\le\frac1p+\frac13,$ $(a_0,u_0,B_0)\in\dot{B}_{q,1}^{\frac{3}{q}}(\R^3)\times\dot{B}_{p,1}^{-1+\frac3p}(\R^3)\times\dot{B}_{p,1}^{-1+\frac3p}(\R^3)$
with $\diverg u_0=\diverg B_0=0$ and $1+\inf_{x\in\R^3}a_0(x)\ge \kappa>0$. Assume that $\widetilde{\sigma}(a)$
is a smooth, positive function on $[0,\infty)$. Then \eqref{mhdmoxing} has a  local solution $(a,u,B,\nabla\Pi)$ on $[0,T]$ such that
\begin{align}
a\in& C([0,T];\dot{{B}}_{q,1}^{\frac{3}{q}}(\R^3))\cap \widetilde{L}_T^\infty(\dot{B}_{q,1}^{\frac{3}{q}}(\R^3)),
\ \nabla\Pi\in L_T^1(\dot{B}_{p,1}^{-1+\frac3p}(\R^3)),\nonumber\\
u\in& C([0,T];\dot{B}_{p,1}^{-1+\frac3p}(\R^3))\cap \widetilde{L}_T^\infty(\dot{B}_{p,1}^{-1+\frac3p}(\R^3))
\cap L_T^1(\dot{B}_{p,1}^{1+\frac3p}(\R^3)),\nonumber\\
B\in& C([0,T];\dot{B}_{p,1}^{-1+\frac3p}(\R^3))\cap \widetilde{L}_T^\infty(\dot{B}_{p,1}^{-1+\frac3p}(\R^3))
\cap L_T^1(\dot{B}_{p,1}^{1+\frac3p}(\R^3)).
\end{align}
Especially, if $a_0\in{B}_{q,1}^{\frac{3}{q}}(\R^3)$, $p\in[3,4], q\in[1,2]$, this solution is unique.
Moreover, when  $\widetilde{\sigma}(a)$ is a  positive constant, and there exists a small constant $c$ depending $\|a_0\|_{{B}_{q,1}^{\frac{3}{q}}}$
such that
$$\|u_0\|_{\dot{B}_{p,1}^{-1+\frac3p}}+\|B_0\|_{\dot{B}_{p,1}^{-1+\frac3p}}\le c,$$
then \eqref{mhdmoxing} has a global solution $(a,u,B,\nabla\Pi)$ such that for any $t>0$\rm{:}
\begin{align}
&\|a\|_{\widetilde{L}^\infty({B}_{q,1}^{\frac{3}{q}})}+\|(u,B)\|_{\widetilde{L}^\infty(\dot{B}_{p,1}^{-1+\frac3p})}+\|(u,B)\|_{L^1(\dot{B}_{p,1}^{1+\frac3p})}
+\|\nabla\Pi\|_{L^1(\dot{B}_{p,1}^{1+\frac3p})}\nonumber\\
\lesssim&C(\|a_0\|_{B_{p,1}^{\frac{3}{q}}}+\|u_0\|_{\dot{B}_{p,1}^{-1+\frac{3}{p}}}+\|B_0\|_{\dot{B}_{p,1}^{-1+\frac{3}{p}}})\exp\left\{C\exp\big(Ct^\frac12\big)\right\},
\end{align}
for some time independent of constant $C$.
\end{theorem}
\begin{remark}
When the magnetic $B=0,$  Theorem \ref{zhuyaodingli}  coincides with Theorem \ref{zhuyaodingli1}. Thus, in the following, we only take the a  priori estimate for
the system \eqref{mhdmoxing}, the local existence, uniqueness, and the global solution under the small initial data of Theorem \ref{zhuyaodingli1} will be proved together with Theorem \ref{zhuyaodingli}.
\end{remark}
\begin{remark}
By the embedding relation, in what follows we only concern the case $p\in[3,\frac{5+\sqrt{17}}{2})$, the case where $1<p<3$ being easier.
\end{remark}

The paper is organized as follows.
In the second section, we shall first collect some basic facts on Littlewood-Paley
theory  and various product laws and commutator's estimates in Besov spaces;
then present
the estimates to the free transport equation and heat equation.
In Section 3, we use a new commutator  of integral type to give a new elliptic estimates with the nonconstant coefficients.
In Section 4, we apply the
  elliptic estimates with the nonconstant coefficients  to get the linear estimates of the inhomogeneous Navier-Stokes-type equations and MHD equations respectively. With these estimates in hand, we shall proved the local well-posedness part of Theorem \ref{zhuyaodingli} in Section 5. In Section 6, we present the proof of the global existence part of Theorem \ref{zhuyaodingli}.

Let us complete this section by describing the notations which will be used in the sequel.
\noindent{\rm\bf Notations: }
For two operators $A$ and $B$, we denote $[A,B]=AB-BA$, the commutator between $A$ and $B$.
The letter $C$ stands for a generic constant whose meaning is clear from the context.
We  write $a\lesssim b$ instead of $a\leq Cb$.
Given a Banach space $X$, we shall denote by$\langle a,b\rangle$ the $L^2(\R^2)$ inner product of $a$ and $b$, and $\|(a,b)\|_{X}=\|a\|_{X}+\|b\|_{X}$.

For $X$ a Banach space and $I$ an interval of $\mathbb{R}$, we denote by $C(I;X)$ the set of continuous functions on $I$ with values in $X$,
and by $C_{b}(I;X)$ the subset of bounded functions of $C(I;X)$.
For $q\in [1,+\infty]$, $L^q(I;X)$ stands for the set of measurable functions on $I$ with values in $X$, such that $t\mapsto\|f(t)\|_{X}$ belongs to $L^q(I)$.
For short, we  write $L_T^q(X)$ instead of $L^q((0,T);X)$. We always let $(d_j)_{j\in\mathbb{Z}}$ be a
generic element of $l^1(\mathbb{Z})$  so that $\sum_{j\in\mathbb{Z}}d_j=1$, and $(c_{j,r})_{j\in\Z}$ to be a generic element of $l^r(\Z)$ so that $c_{j,r}\ge0$ and $\sum_{j\in\Z}c_{j,r}^r=1$.

\section{Preliminaries}
Let  $(\chi,\phi)$  be two  smooth radial functions,
$0\leq  (\chi,\phi) \leq 1,$
such that $\chi$  is supported in the ball
$\mathcal{B}=\{\xi \in \R^{3}, |\xi|\leq \frac{4}{3}\}$ and $\varphi$  is
supported in the ring
 $\mathcal{C}=\{\xi \in \R^{3}, \frac34\leq |\xi|\leq \frac{8}{3}\}$. Moreover, there hold
$$
 \forall\,\xi\in\R^3 \setminus \{0\},\quad \quad\sum_{j \in \mathbb{Z}} \varphi(2^{-j}\xi)=1,
$$
$$
 \forall\,\xi\in\R^3,\quad \quad\chi(\xi)+\sum_{j \ge0} \varphi(2^{-j}\xi)=1.
$$
Let $h=\mathcal {F}^{-1}\varphi$ and $\widetilde{h}=\mathcal {F}^{-1}\chi$, the inhomogeneous dyadic blocks ${\Delta}_j$ are defined as follows:
\begin{align*}
if \quad&j\le -2, \quad {\Delta}_j f=0,\\
if \quad &j=-1, \quad{\Delta}_j f={\Delta}_{-1} f=\int_{\R^3}\widetilde{h}(y)f(x-y)dy,\\
if \quad &j\ge0, \quad{\Delta}_j f=2^{3j}\int_{\R^3}h(2^{j}y)f(x-y)dy.
\end{align*}
 The inhomogeneous low-frequency cut-off operator $S_j$ is defined by
 $$S_j f=\sum_{j'\le j-1}{\Delta}_{j'}f.$$
For $j \in Z,$ the homogeneous dyadic blocks $\dot{\Delta}_j $ and the homogeneous low-frequency
cut-off operator
$\dot{S}_j$ are defined as follows:
$$\dot{\Delta}_jf=\varphi(2^{-j}D)f=2^{3j}\int_{\R^3}h(2^{j}y)f(x-y)dy,$$
$$\dot{S}_jf=\chi(2^{-j}D)f=2^{3j}\int_{\R^3}\widetilde{h}(2^{j}y)f(x-y)dy.$$
Denote by $\mathscr{S}_h^{'}(\R^3)$ the space of tempered distributions $f$ such that
$$
\lim_{j\rightarrow -\infty}\dot{S}_j f=0\ \ in\ \ \mathscr{S}^{'}(\R^3).
$$
Then we have the formal decomposition
$$
f=\sum_{ j\in \mathbb{Z} }\dot{ \Delta }_{j}f,\ \ \  \forall f\in\mathscr{S}_h^{'}(\mathbb{R}^3).
$$

 Now we recall the definition of homogeneous Besov spaces.
 \begin{definition}
 Let $(p,r) \in [1,+\infty]^2, s \in \R$ and $u\in \mathcal{S}'_h(\R^3)$, which means that $u\in \mathcal{S}'(\R^3)$ and $\lim_{j\to -\infty}\|\dot{S}_ju\|_{L^\infty}=0$ (see Definition 1.26 of \cite{bcd}); we set
$$\|u\|_{\dot{B}_{p,r}^s}\triangleq(2^{js}\|\dot{\Delta}_ju\|_{L^p})_{l^r}.$$

\end{definition}
$\bullet$ For $s<\frac3p$ (or $s=\frac3p$ if $r=1$), we define $\dot{B}_{p,r}^s(\R^3)\triangleq\{u\in \mathcal{S}'_h(\R^3)\mid\|u\|_{\dot{B}_{p,r}^s}<\infty\}$.

$\bullet$ If $k\in \N$ and $\frac3p+k\le s<\frac3p+k+1$ (or $s=\frac3p+k+1$ if $r=1$), then $\dot{B}_{p,r}^s(\R^3)$ is defined as the subset of distributions $u\in \mathcal{S}'_h(\R^3)$ such that $\partial^\beta u\in \dot{B}_{p,r}^{s-k}(\R^3)$.

We also define the inhomogeneous Besov space ${B}_{p,r}^s(\R^3)$ as the space of those
distributions $u \in \mathcal{S}'_h(\R^3)$ such that
$$\|u\|_{{B}_{p,r}^s}\triangleq(2^{js}\|\dot{\Delta}_ju\|_{L^p})_{l^r}<\infty.$$

We are going to define the space of Chemin-Lerner (see \cite{bcd}) in which we will work, which is a
refinement of the space ${L^\lambda_{T}(\dot{B}_{p,r}^s(\R^3))}$.
\begin{definition}(see \cite{bcd})
Let $s\le\frac 3p$ (respectively $s \in \R$), $(r,\lambda,p) \in [1,+\infty]^3$ and $T \in (0,+\infty]$. We define
${\widetilde{L}^\lambda_{T}(\dot{B}_{p,r}^s(\R^3))}$ as the completion of $C([0,T ];\mathcal{S}(\R^3))$ by the norm
$$
\|f\|_{\widetilde{L}^\lambda_{T}(\dot{B}_{p,r}^s)}=\left\{\sum_{q\in \Z}2^{rqs}
\left(\int_0^T\|\dot{\Delta}_qf(t)\|_{L^p}^\lambda dt\right)^{\frac r\lambda}\right\}^{\frac 1r}<\infty,
$$
with the usual change if $r =\infty$. For short, we just denote this space by ${\widetilde{L}^\lambda_{T}(\dot{B}_{p,r}^s)}$.
\end{definition}
\begin{remark}
It is easy to observe that for $0<s_1<s_2,$ $\theta\in[0,1]$, $p,r,\lambda,\lambda_1,\lambda_2\in[1,+\infty]$, we have the following interpolation inequality in the Chemin-Lerner space (see \cite{bcd}):
\begin{align*}
\|u\|_{\widetilde{L}^\lambda_{T}(\dot{B}_{p,r}^s)}\le\|u\|^\theta_{\widetilde{L}^{\lambda_1}_{T}(\dot{B}_{p,r}^{s_1})}
\|u\|^{(1-\theta)}_{\widetilde{L}^{\lambda_2}_{T}(\dot{B}_{p,r}^{s_2})}
\end{align*}
with $\frac{1}{\lambda}=\frac{\theta}{\lambda_1}+\frac{1-\theta}{\lambda_2}$ and $s=\theta s_1+(1-\theta)s_2$.
\end{remark}
Let us emphasize that, according to the Minkowski inequality, we have
\begin{align*}
\|u\|_{\widetilde{L}^\lambda_{T}(\dot{B}_{p,r}^s)}\le\|u\|_{L^\lambda_{T}(\dot{B}_{p,r}^s)}\hspace{0.5cm} \mathrm{if }\hspace{0.2cm}  \lambda\le r,\hspace{0.5cm}
\|u\|_{\widetilde{L}^\lambda_{T}(\dot{B}_{p,r}^s)}\ge\|u\|_{L^\lambda_{T}(\dot{B}_{p,r}^s)},\hspace{0.5cm} \mathrm{if }\hspace{0.2cm}  \lambda\ge r.
\end{align*}
The following Bernstein's lemma will be repeatedly used throughout this paper.
\begin{Lemma}\label{bernstein}(see \cite{bcd})
Let $\mathcal{B}$ be a ball and $\mathcal{C}$ a ring of $\R^3$. A constant $C$ exists so that for any positive real number $\lambda$, any
non-negative integer k, any smooth homogeneous function $\sigma$ of degree m, and any couple of real numbers $(a, b)$ with
$1\le a \le b$, there hold
\begin{align*}
&&\mathrm{Supp} \,\hat{u}\subset\lambda \mathcal{B}\Rightarrow\sup_{|\alpha|=k}
\|\partial^{\alpha}u\|_{L^b}\le C^{k+1}\lambda^{k+3(1/a-1/b)}\|u\|_{L^a},\\
&&\mathrm{Supp} \,\hat{u}\subset\lambda \mathcal{C}\Rightarrow C^{-k-1}\lambda^k\|u\|_{L^a}\le\sup_{|\alpha|=k}
\|\partial^{\alpha}u\|_{L^a}
\le C^{k+1}\lambda^{k}\|u\|_{L^a},\\
&&\mathrm{Supp} \,\hat{u}\subset\lambda \mathcal{C}\Rightarrow \|\sigma(D)u\|_{L^b}\le C_{\sigma,m}\lambda^{m+3(1/a-1/b)}\|u\|_{L^a}.
\end{align*}
\end{Lemma}
In the sequel, we shall frequently use Bony's decomposition from \cite{bony} in the homogeneous context:
\begin{align}\label{bony}
uv=\dot{T}_uv+\dot{T}_vu+\dot{R}(u,v)=\dot{T}_uv+\mathcal{\dot{R}}(u,v),
\end{align}
where
$$\dot{T}_uv\triangleq\sum_{j\in Z}\dot{S}_{j-1}u\dot{\Delta}_jv, \hspace{0.5cm}\dot{R}(u,v)\triangleq\sum_{j\in Z}
\dot{\Delta}_ju\widetilde{\dot{\Delta}}_jv,$$
and
$$ \widetilde{\dot{\Delta}}_jv\triangleq\sum_{|j-j'|\le1}\dot{\Delta}_{j'}v,\hspace{0.5cm} \mathcal{\dot{R}}(u,v)\triangleq\sum_{j\in Z}\dot{S}_{j+2}v\dot{\Delta}_ju.$$

As an application of the above basic facts on the Littlewood-Paley theory, we present the following
different product laws in Besov spaces.
\begin{Lemma}\label{daishu}
Let $1\leq p, q\leq \infty$, $s_1\leq \frac{ 3}{q}$, $s_2\leq 3\min\{\frac1p,\frac1q\}$ and $s_1+s_2>3\max\{0,\frac1p +\frac1q -1\}$. For $\forall (a,b)\in\dot{B}_{q,1}^{s_1}({\mathbb R} ^3)\times\dot{B}_{p,1}^{s_2}({\mathbb R} ^3)$, we have
\begin{align}\label{product inequality}
\|ab\|_{\dot{B}_{p,1}^{s_1+s_2 -\frac{3}{q}}}\lesssim \|a\|_{\dot{B}_{q,1}^{s_1}}\|b\|_{\dot{B}_{p,1}^{s_2}}.
\end{align}
\end{Lemma}
\begin{proof} This lemma is proved in \cite{paicu2012} in the case when $q\leq p$. We shall only prove \eqref{product inequality} for the case $q>p$.
Applying Bony's decomposition, we have
$$
ab=\dot{T}_{a}b+\dot{T}_{b}a+\dot{R}(a,b).
$$
Then applying Lemma \ref{bernstein}, we get for $s_1\leq\frac{3}{q}$
\begin{align*}
\|\dot{\Delta}_j(\dot{T}_{a}b)\|_{L^p}\lesssim\sum_{|j'-j|\leq 4}\|\dot{S}_{j'-1}a\|_{L^\infty}\|\dot{\Delta}_{j'}b\|_{L^p}
\lesssim d_j2^{-j(s_1+s_2-\frac{3}{q})}\|a\|_{\dot{B}_{q,1}^{s_1}}\|b\|_{\dot{B}_{p,1}^{s_2}},
\end{align*}
and for $s_2\leq\frac{3}{q}$
\begin{align*}
\|\dot{\Delta}_j(\dot{T}_{b}a)\|_{L^p}\lesssim\sum_{|j'-j|\leq 4}\|\dot{\Delta}_{j'}a\|_{L^q}\|\dot{S}_{j'-1}b\|_{L^{\frac{pq}{q-p}}}
\lesssim d_j2^{-j(s_1+s_2-\frac{3}{q})}\|a\|_{\dot{B}_{q,1}^{s_1}}\|b\|_{\dot{B}_{p,1}^{s_2}}.
\end{align*}
If $\frac1p+\frac1q\geq1=\frac1p+\frac{1}{p'}$, we infer
\begin{align*}
\|\dot{\Delta}_j\big(\dot{R}(a,b)\big)\|_{L^p}
\lesssim&2^{3j(1-\frac1p)}\sum_{j'\geq j-3}\|\dot{\Delta}_{j'}a\|_{L^{p'}}\|\widetilde{\dot{\Delta}}_{j'}b\|_{L^{p}}\\
\lesssim&2^{3j(1-\frac1p)}\|a\|_{\dot{B}_{q,1}^{s_1}}\|b\|_{\dot{B}_{p,1}^{s_2}}\sum_{j'\geq j-3}
d_{j'}2^{-j'\big(s_1+s_2-3(\frac1p+\frac1q-1)\big)}\\
\lesssim&d_j2^{-j(s_1+s_2-\frac{3}{q})}\|a\|_{\dot{B}_{q,1}^{s_1}}\|b\|_{\dot{B}_{p,1}^{s_2}},
\end{align*}
for $s_1+s_2>3(\frac1p+\frac1q-1)$. Finally, in the case when $\frac1p+\frac1q\stackrel{\mathrm{def}}{=}\frac1r<1$, noticing that $s_1+s_2>0$, one has
\begin{align*}
\|\dot{\Delta}_j\big(\dot{R}(a,b)\big)\|_{L^p}
\lesssim&2^{3j(\frac1r-\frac1p)}\sum_{j'\geq j-3}\|\dot{\Delta}_{j'}a\|_{L^{q}}\|\widetilde{\dot{\Delta}}_{j'}b\|_{L^{p}}\\
\lesssim&d_j2^{-j(s_1+s_2-\frac{3}{q})}\|a\|_{\dot{B}_{q,1}^{s_1}}\|b\|_{\dot{B}_{p,1}^{s_2}}.
\end{align*}
This completes the proof of the lemma.
\end{proof}

\begin{remark}\label{product law1}
By the above lemma, we can easily verify the following product laws which will be used frequently in this paper.

\noindent(1) Let $1< q\le p$ and $\frac1p+\frac1q\ge\frac13,$ then there hold
\begin{align}\label{}
\|ab\|_{\dot{B}_{p,1}^{\frac{3}{p}}}\lesssim\|a\|_{\dot{B}_{q,1}^{\frac{3}{q}}}\|b\|_{\dot{B}_{p,1}^{\frac{3}{p}}},\quad
\|ab\|_{\dot{B}_{p,1}^{-1+\frac3p}}\lesssim\|a\|_{\dot{B}_{q,1}^{\frac{3}{q}}}\|b\|_{\dot{B}_{p,1}^{-1+\frac3p}}.
\end{align}
(2)  Let $1< q\le p$ and $\frac1p+\frac1q\ge\frac12,$ then there hold
\begin{align}\label{}
\|ab\|_{\dot{B}_{p,1}^{-\frac32+\frac3p}}\lesssim\|a\|_{\dot{B}_{q,1}^{\frac{3}{q}}}\|b\|_{\dot{B}_{p,1}^{-\frac32+\frac3p}},
\quad\|ab\|_{\dot{B}_{p,1}^{-\frac32+\frac3p}}\lesssim\|a\|_{\dot{B}_{q,1}^{-1+\frac{3}{q}}}\|b\|_{\dot{B}_{p,1}^{-\frac12+\frac{3}{p}}}.
\end{align}

\end{remark}

\begin{Lemma}\label{product law2}
 Let $1< q\le p$ and $\frac1p+\frac1q\ge\frac13,$ then there hold
\begin{align}\label{}
\|\mathbb{P}(a\nabla\Pi)\|_{\dot{B}_{p,1}^{\frac{3}{p}}}\lesssim\|\nabla a\|_{\dot{B}_{q,1}^{\frac{3}{q}}}\|\nabla\Pi\|_{\dot{B}_{p,1}^{-1+\frac3p}},
\end{align}
$$
\|\mathbb{P}(a\nabla\Pi)\|_{L_t^1(\dot{B}_{p,1}^{-1+\frac3p})}
\lesssim\|a\|_{L_t^\infty(\dot{B}_{q,1}^{\frac{3}{q}+\frac12})}\|\nabla\Pi\|_{L_t^1(\dot{B}_{p,2}^{-\frac32+\frac3p})}.
$$
\end{Lemma}
\begin{proof} We only treat the second inequality, since the proof of  the first inequality is similar.
For $\mathbb{P}(a\nabla\Pi)=\mathbb{P}(\nabla(a\Pi)-\nabla a \Pi)=\mathbb{P}(\nabla a\Pi)$, thus by Lemma \ref{daishu}, we have
\begin{align}\label{}
\|\mathbb{P}(a\nabla\Pi)\|_{L_t^1(\dot{B}_{p,1}^{-1+\frac3p})}
&\lesssim\|\mathbb{P}(\nabla a\Pi)\|_{L_t^1(\dot{B}_{p,1}^{-1+\frac3p})}\nonumber\\
&\lesssim\|\nabla a\|_{\widetilde{L}_t^\infty(\dot{B}_{q,1}^{\frac{3}{q}-\frac12})} \|\nabla\Pi\|_{L_t^1(\dot{B}_{p,2}^{-\frac32+\frac3p})},
\end{align}
from which the desired inequality follows.
\end{proof}

Let us also recall the following commutator estimate from \cite{bcd}.
\begin{Lemma} (see \cite{bcd})\label{jiaohuanzi}
Let $1\leq p,q\leq\infty$, $-3\min\{\frac1p,1-\frac{1}{q}\}<s\leq1+3\min\{\frac1p,\frac1q\}$, $\nabla a\in\dot{B}_{p,1}^{\frac{3}{p}}(\R^3)$
and $b\in\dot{B}_{q,1}^{s-1}(\R^3)$. Then there holds
$$
\|[\dot{\Delta}_j,a]b\|_{L^q}\lesssim d_j 2^{-js}\|\nabla a\|_{\dot{B}_{p,1}^{\frac{3}{p}}}\|b\|_{\dot{B}_{q,1}^{s-1}}.
$$
\end{Lemma}

\begin{Lemma}\label{jiaohuanzi2}
 Let $1<q\le p<6, a\in{\dot{B}_{q,1}^{\frac{3}{q}+\frac12}}(\R^3)$ and $\nabla\Pi\in{\dot{B}_{p,2}^{-\frac32+\frac3p}}(\R^3)$. Then there holds
\begin{align}\label{miaomiao}
\|[\dot{\Delta}_j,a]\nabla\Pi\|_{L^p}
\lesssim2^{-(\frac{3}{p}-1)j}d_j\|a\|_{\dot{B}_{q,1}^{\frac{3}{q}+\frac12}}\|\nabla\Pi\|_{\dot{B}_{p,2}^{-\frac32+\frac3p}}.
\end{align}
\end{Lemma}
\begin{proof}
Taking advantage of the Bony's decomposition \eqref{bony} , we rewrite the commutator as:
\begin{align}\label{fenjiefangji}
[\dot{\Delta}_j,a]\nabla\Pi=\dot{\Delta}_j(a\nabla\Pi)-a\dot{\Delta}_j\nabla\Pi
=[\dot{\Delta}_j,T_a]\nabla\Pi+\dot{\Delta}_j(\dot{T}'_{\nabla\Pi}a)-\dot{T}'_{\dot{\Delta}_j\nabla\Pi}a.
\end{align}
By the definition of Bony's decomposition again, we have
\begin{align*}
[\dot{\Delta}_j,\dot{T}_{a}]\nabla\Pi
=-2^{3j}\sum_{|j'-j|\leq4}\int_{\R^3}h(2^jy)\dot{\Delta}_{j'}\nabla\Pi(x-y)dy\int_0^1y\cdot\nabla\dot{S}_{j'-1}a(x-\tau y)d\tau,
\end{align*}
from which we can get by Lemma \ref{bernstein} and the H$\rm{\ddot{o}}$lder inequality  that
\begin{align}\label{pre1}
\|[\dot{\Delta}_j,\dot{T}_{a}]\nabla\Pi\|_{L^p}
\lesssim&
\sum_{|j'-j|\leq4}\|\nabla\dot{S}_{j'-1}a\|_{L^\infty}\|\dot{\Delta}_{j'}\nabla\Pi\|_{L^p}\nonumber\\
\lesssim&
\sum_{|j'-j|\leq4}(\sum_{k\le j'-2}\|\nabla\dot{\Delta}_{k}a\|_{L^q}2^{\frac{3k}{q}})\|\dot{\Delta}_{j'}\nabla\Pi\|_{L^p}\nonumber\\
\lesssim& c_j2^{-(\frac3p-1)j}\|a\|_{\dot{B}_{q,1}^{\frac{3}{q}+\frac12}}\|\nabla\Pi\|_{\dot{B}_{p,2}^{-\frac32+\frac3p}}.
\end{align}
When $1<q\le p\le2,$ we get from
$$
\dot{\Delta}_j\big(\dot{T}'_{\nabla\Pi}a\big)=\sum_{j'\geq j-3}\dot{\Delta}_j\big(\dot{\Delta}_{j'}a\dot{S}_{j'+2}\nabla\Pi\big)$$
that
\begin{align}\label{pre2}
\|\dot{\Delta}_j\big(\dot{T}'_{\nabla\Pi}a\big)\|_{L^p}\lesssim&2^{3j(1-\frac1p)}\sum_{j'\geq j-3}\|\dot{\Delta}_{j'}a\|_{L^{p}}\|\dot{S}_{j'+2}\nabla\Pi\|_{L^{\frac{p}{p-1}}}\nonumber\\
\lesssim& 2^{3j(1-\frac1p)}\sum_{j'\geq j-3}c_{j',2}2^{3(\frac1q-\frac1p)j'}\|\dot{\Delta}_{j'}a\|_{L^{q}}
(\sum_{k\le j'+1}\|\dot{\Delta}_{k}\nabla\Pi\|_{L^p}2^{(\frac{6}{p}-3)k})
\nonumber\\
\lesssim& 2^{3j(1-\frac1p)}\sum_{j'\geq j-3}c_{j',2}2^{-(\frac12+\frac3p)j'}\|a\|_{\dot{B}_{q,1}^{\frac{3}{q}+\frac12}}
(\sum_{k\le j'+1}c_{k,2}\|\nabla\Pi\|_{\dot{B}_{p,2}^{-\frac32+\frac3p}}2^{(\frac{3}{p}-\frac32)k})
\nonumber\\
\lesssim& d_j2^{-(\frac3p-1)j}\|a\|_{\dot{B}_{q,1}^{\frac{3}{q}+\frac12}}\|\nabla\Pi\|_{\dot{B}_{p,2}^{-\frac32+\frac3p}}.
\end{align}
Similarly, when $1<q\le p\in [2,6),$ we have
\begin{align}\label{pre3}
\|\dot{\Delta}_j\big(\dot{T}'_{\nabla\Pi}a\big)\|_{L^p}\lesssim&2^{\frac{3j}{p}}\sum_{j'\geq j-3}\|\dot{\Delta}_{j'}a\|_{L^{p}}\|\dot{S}_{j'+2}\nabla\Pi\|_{L^{p}}\nonumber\\
\lesssim& 2^{\frac{3j}{p}}\sum_{j'\geq j-3}c_{j',2}2^{3(\frac1q-\frac1p)j'}\|\dot{\Delta}_{j'}a\|_{L^{q}}
(\sum_{k\le j'+1}\|\dot{\Delta}_{k}\nabla\Pi\|_{L^p})
\nonumber\\
\lesssim& 2^{\frac{3j}{p}}\sum_{j'\geq j-3}c_{j',2}2^{-(\frac12+\frac3p)j'}\|a\|_{\dot{B}_{q,1}^{\frac{3}{q}+\frac12}}
(\sum_{k\le j'+1}c_{k,2}\|\nabla\Pi\|_{\dot{B}_{p,2}^{-\frac32+\frac3p}}2^{(\frac{3}{2}-\frac3p)k})
\nonumber\\
\lesssim&2^{\frac{3j}{p}}\sum_{j'\geq j-3}d_{j'}2^{(1-\frac6p)j'}\|a\|_{\dot{B}_{q,1}^{\frac{3}{q}+\frac12}}
\|\nabla\Pi\|_{\dot{B}_{p,2}^{-\frac32+\frac3p}}
\nonumber\\
\lesssim& d_j2^{-(\frac3p-1)j}\|a\|_{\dot{B}_{q,1}^{\frac{3}{q}+\frac12}}\|\nabla\Pi\|_{\dot{B}_{p,2}^{-\frac32+\frac3p}}.
\end{align}
For the last term on the right hand side of (\ref{fenjiefangji}), we write
\begin{align*}
\dot{T}'_{\dot{\Delta}_j\nabla\Pi}a=\sum_{j'\geq j-3}\dot{\Delta}_{j'}a\dot{S}_{j'+2}\dot{\Delta}_{j}\nabla\Pi.
\end{align*}
Thus
\begin{align}\label{pre4}
\|\dot{T}'_{\dot{\Delta}_j\nabla\Pi}a\|_{L^p}\lesssim&\sum_{j'\geq j-3}\|\dot{\Delta}_{j'}a\|_{L^{\infty}}\|\dot{S}_{j'+2}\dot{\Delta}_{j}\nabla\Pi\|_{L^p}\nonumber\\
\lesssim&\sum_{j'\geq j-3}2^{-\frac3qj'}\|\dot{\Delta}_{j'}a\|_{L^{q}}\|\dot{\Delta}_{j}\nabla\Pi\|_{L^p}\nonumber\\
\lesssim&\sum_{j'\geq j-3}c_{j',2}2^{-\frac12j'}\|a\|_{\dot{B}_{q,1}^{\frac{3}{q}+\frac12}}2^{-(\frac{3}{p}-\frac32)j}c_{j,2}\|\nabla\Pi\|_{\dot{B}_{p,2}^{-\frac32+\frac3p}}\nonumber\\
\lesssim& d_j2^{-(\frac3p-1)j}\|a\|_{\dot{B}_{q,1}^{\frac{3}{q}+\frac12}}\|\nabla\Pi\|_{\dot{B}_{p,2}^{-\frac32+\frac3p}}.
\end{align}
Combining with the above estimates \eqref{fenjiefangji}--\eqref{pre4}, we can complete the proof  of \eqref{miaomiao}.
\end{proof}

\begin{proposition}  (see \cite{abidi2012})  \label{shuyun}
Let $1\le q\le  p\le 6$ with $\frac1q-\frac1p\le\frac13,$ and $m\in\Z, a_0\in \dot{B}_{q,1}^{\frac{3}{q}}$, $\nabla u\in L_T^1({{\dot{B}}_{p,1}^{\frac{3}{p}}})$ with $div u=0$, and $a\in C([0, T ]; \dot{B}_{q,1}^{\frac{3}{q}})$ such that $(a,u)$ solves
\begin{equation*}%\label{3.1}
   \left\{
 \begin{array}{ll}
 \partial_t a +  u \cdot\nabla a =0, \\
a(x,0)= a_0.
 \end{array} \right.
  \end{equation*}
Then there hold for $\forall\, t\le T$
\begin{align}\label{3.2}
\|a\|_{\widetilde{L}^\infty_t(\dot{B}_{q,1}^{\frac{3}{q}})}\le\|a_0\|_{\dot{B}_{q,1}^{\frac{3}{q}}}e^{CU(t)},
\end{align}
\begin{align}\label{3.3}
\|a-\dot{S}_ma\|_{\widetilde{L}^\infty_t(\dot{B}_{q,1}^{\frac{3}{q}})}\le\sum_{q\ge m} 2^{3q/2}\|{\Delta}_qa_0\|_{L^2}+\|a_0\|_{\dot{B}_{q,1}^{\frac{3}{q}}}(e^{CU(t)}-1),
\end{align}
with $U(t)=\|\nabla u\|_{L^1_t(\dot{\dot{B}}_{p,1}^{\frac{3}{p}})}$.

Similar inequality holds for the inhomogeneous Besov norm ${B}_{q,1}^{\frac{3}{q}}(\R^3)$.
\end{proposition}
%%%%%%%%%%%%%%%%%%%%%%%%%%%%%%%%%%%%%%%%%%%%%%%%%%%%%%%%%%%%%%%%%%%%%%%%%%%%%%%%%%%
%%%%%%%%%%%%%%%%%%%%%%%%%%%%%%%%%%%%%%%%%%%%%%%%%%%%%%%%%5
In order to prove the uniqueness of the main theorems, we need the following proposition in \cite{abidi2013}, we omit the details here for its proof.
\begin{proposition} (see \cite{abidi2013}) \label{weiyixingmingti}
Let $\alpha\in (0, \frac14),$ $ p \in [3, 4],$ $u_0 \in B_{2,1}^{-\frac12}(\R^3)$ and $v$ be a divergence free vector field satisfying
$\nabla v\in {L_T^1({B_{p,1}^{\frac3p}})}.$
 And let $f \in {L_t^1({B_{2,1}^{-\frac12}})}$ and $ a\in L^\infty_T(H^2)$
 with $1 + a \ge  \underline{b} > 0$. We assume that $u \in C([0,T ];{B_{2,1}^{-\frac12}}) \cap{L_T^1({B_{2,1}^{\frac32}})}$ and $\nabla \Pi\in L^1_T(H^{-\frac12-\alpha})$
solve
\begin{eqnarray*}%\label{Model5}
\left\{\begin{aligned}
&\partial_t u+v\cdot\nabla u-(1+a)(\Delta u-\nabla\Pi)=f,\\
&\mathrm{div}u=0,\\
&u|_{t=0}=u_0.
\end{aligned}\right.
\end{eqnarray*}
Then for all $t\le T$, there holds:
\begin{align*}%\label{}
\|u\|_{\widetilde{L}_t^{\infty}({B}_{2,1}^{-\frac12})}+\|u\|_{{L}_t^1({B}_{2,1}^{\frac32})}\lesssim&\Bigg\{
\|u_0\|_{{B_{2,1}^{-\frac12}}}+\int_0^t\|u\|_{{B_{2,1}^{-\frac12}}}\|\nabla v\|_{{B_{p,1}^{\frac3p}}}d\tau+\|a\|_{L_t^\infty(H^{\frac32+\alpha})}\|\nabla \Pi\|_{L^1_t(H^{-\frac12-\alpha})}\nonumber\\
&+\|a\|_{L_t^\infty(H^{2})}\|\nabla u\|_{L^1_t(L^2)}+\|u\|_{L^1_t(L^2)}+\|f\|_{{L_t^1({B_{2,1}^{-\frac12}})}}
\Bigg\}.
\end{align*}

\end{proposition}
%%%%%%%%%%%%%%%%%%%%%%%%%%%%%%%%%%%%%%%%%%%%%%%%%%%%%%%%%%%%%%%%%%%%%%
%%%%%%%%%%%%%%%%%%%%%%%%%%%%%%%%%%%%%%%%%%%%%%%%%%%%%%%%%%%%%%%%%%%%%

\section{\bf Elliptic estimates with variable coefficients}
This section is devoted to the proof of new estimates for the elliptic equation with variable coefficients.
\begin{proposition}\label{tuoyuanguji}
Assume $1< q\le p$ with $p\in(\frac{1+\sqrt{17}}{4},\frac{5+\sqrt{17}}{2})$ and $\frac1p+\frac1q\geq\frac12$, $a\in\dot{B}_{q,1}^{\frac{3}{q}}(\R^3)$ with $$1+a\geq\kappa>0.$$
Let $F=(F_1,F_2,F_3)\in\dot{B}_{p,2}^{-\frac32+\frac3p}(\R^3)$ and $\nabla\Pi\in\dot{B}_{p,2}^{-\frac32+\frac3p}(\R^3)$ solve
\begin{align}\label{elliptic equation}
\diverg\big((1+a)\nabla\Pi\big)=\diverg F.
\end{align}
Then, we have
\begin{align}\label{elliptic estimate1}
\|\nabla\Pi\|_{\dot{B}_{p,2}^{-\frac32+\frac3p}}
\lesssim\big(1+\|a\|_{\dot{B}_{q,1}^{\frac{3}{q}}}\big)\|\mathbb{Q}F\|_{\dot{B}_{p,2}^{-\frac32+\frac3p}}.
\end{align}

\end{proposition}
\begin{proof} Thanks to $1+a\geq\kappa>0$ and $\diverg F=\diverg\mathbb{Q}F$, similar to the proof of Lemma 2 in \cite{danchin2010}, we readily deduce from \eqref{elliptic equation} that
\begin{align}\label{elliptic estimateL2}
\kappa\|\nabla\Pi\|_{L^2}\leq\|\mathbb{Q}F\|_{L^2}.
\end{align}
Applying $\dot{\Delta}_j$ to \eqref{elliptic equation} gives
\begin{align}\label{}
\diverg\big((1+a)\dot{\Delta}_j\nabla\Pi\big)=&\diverg\dot{\Delta}_jF-\diverg\big([\dot{\Delta}_j,a]\nabla\Pi\big).
\end{align}
We next multiply the above equation by $-|\dot{\Delta}_j\Pi|^{p-2}\dot{\Delta}_j\Pi$ and integrate over $\R^3$.
Then applying Lemma 8 in Appendix B of \cite{danchin2010} implies for some constants $c$ and $C$
\begin{align}\label{Ee1}
c\kappa2^{2j}\|\dot{\Delta}_j\Pi\|_{L^p}^p\leq C2^j\|\dot{\Delta}_j\mathbb{Q}F\|_{L^p}\|\dot{\Delta}_j\Pi\|_{L^p}^{p-1}
+\int_{\R^3}\diverg\big([\dot{\Delta}_j,a]\nabla\Pi\big)\cdot|\dot{\Delta}_j\Pi|^{p-2}\dot{\Delta}_j\Pi dx.
\end{align}
In order to estimate the last term on the right hand side of (\ref{Ee1}),
we need the following commutator estimates of integral type.
The estimates of the following lemma have no restrict on the size of the relationship of $p,q$ which are somewhat more general than the necessary one in the present paper.
\begin{Lemma}\label{Lemma3.1}
Let $(p,q)\in(\frac{1+\sqrt{17}}{4},2]\times[1,\infty)$ with $\frac1p-\frac1q\leq\frac12$.
Then we have
\begin{align}\label{Ijpq}
I_j\stackrel{\rm{def}}{=}\int_{\R^3}\diverg\big([\dot{\Delta}_j,a]\nabla\Pi\big)\cdot|\dot{\Delta}_j\Pi|^{p-2}\dot{\Delta}_j\Pi dx
\lesssim d_j2^{j(\frac52-\frac{3}{p})}\|a\|_{\dot{B}_{q,1}^{\frac{3}{q}}}\|\nabla\Pi\|_{L^2}\|\dot{\Delta}_j\Pi\|_{L^p}^{p-1}.
\end{align}

\end{Lemma}
\begin{proof}  Noticing that we can not directly use integration by parts. For this, we first get by using Bony's decomposition
\begin{align}\label{}
I_j=&\int_{\R^3}\diverg\big([\dot{\Delta}_j,\dot{T}_{a}]\nabla\Pi\big)\cdot|\dot{\Delta}_j\Pi|^{p-2}\dot{\Delta}_j\Pi dx
+\int_{\R^3}\diverg\dot{\Delta}_j\big(\dot{T}'_{\nabla\Pi}a\big)\cdot|\dot{\Delta}_j\Pi|^{p-2}\dot{\Delta}_j\Pi dx\nonumber\\
&-\int_{\R^3}\diverg\big(\dot{T}'_{\dot{\Delta}_j\nabla\Pi}a\big)\cdot|\dot{\Delta}_j\Pi|^{p-2}\dot{\Delta}_j\Pi dx\nonumber\\
\stackrel{\rm{def}}{=}&I_j^1+I_j^2+I_j^3.
\end{align}
By the definition of Bony's decomposition, we have
\begin{align}\label{cmt}
[\dot{\Delta}_j,\dot{T}_{a}]\nabla\Pi
=-2^{3j}\sum_{|j'-j|\leq4}\int_{\R^3}h(2^jy)\dot{\Delta}_{j'}\nabla\Pi(x-y)dy\int_0^1y\cdot\nabla\dot{S}_{j'-1}a(x-\tau y)d\tau,
\end{align}
from which, we get by using the H$\rm{\ddot{o}}$lder inequality and Lemma \ref{bernstein} that
\begin{align}\label{}
\|[\dot{\Delta}_j,\dot{T}_{a}]\nabla\Pi\|_{L^p}
\lesssim2^{-j}\sum_{|j'-j|\leq4}\|\nabla\dot{S}_{j'-1}a\|_{L^{\frac{2p}{2-p}}}\|\dot{\Delta}_{j'}\nabla\Pi\|_{L^2}
\lesssim c_j2^{-j(\frac3p-\frac32)}\|a\|_{\dot{B}_{q,1}^{\frac{3}{q}}}\|\nabla\Pi\|_{L^2},
\end{align}
where we have used $$\|\nabla\dot{S}_{j'-1}a\|_{L^{\frac{2p}{2-p}}}\lesssim d_{j'}2^{j'(\frac52-\frac3p)}\|a\|_{\dot{B}_{q,1}^{\frac{3}{q}}}
\ \ \text{for}\ \ p>\frac65\ \ \text{and}\ \ \frac1p-\frac1q\leq\frac12.$$
Note that $[\dot{\Delta}_j,\dot{T}_{a}]\nabla\Pi$ is spectrally supported in an annulus of size $2^j$. Whence we infer
\begin{align}\label{Ij1}
I_j^1\lesssim d_j2^{j(\frac52-\frac{3}{p})}\|a\|_{\dot{B}_{q,1}^{\frac{3}{q}}}\|\nabla\Pi\|_{L^2}\|\dot{\Delta}_j\Pi\|_{L^p}^{p-1}.
\end{align}
Owing to the localization properties of the Littlewood-Paley decomposition, we have
\begin{align*}
\dot{\Delta}_j\big(\dot{T}'_{\nabla\Pi}a\big)=\sum_{j'\geq j-3}\dot{\Delta}_j\big(\dot{\Delta}_{j'}a\dot{S}_{j'+2}\nabla\Pi\big).
\end{align*}
If $q\geq2$, we denote $\frac1\gamma\stackrel{\rm{def}}{=}\frac12+\frac1q\geq\frac1p$ and apply Lemma \ref{bernstein} to obtain
\begin{align}\label{}
\|\dot{\Delta}_j\big(\dot{T}'_{\nabla\Pi}a\big)\|_{L^p}\lesssim2^{3j(\frac1\gamma-\frac1p)}\sum_{j'\geq j-3}\|\dot{\Delta}_{j'}a\|_{L^{q}}\|\dot{S}_{j'+2}\nabla\Pi\|_{L^2}
\lesssim d_j2^{j(\frac32-\frac{3}{p})}\|a\|_{\dot{B}_{q,1}^{\frac{3}{q}}}\|\nabla\Pi\|_{L^2}.
\end{align}
While if $q<2$, the embedding $\dot{B}_{q,1}^{\frac{3}{q}}(\R^3)\hookrightarrow\dot{B}_{2,1}^{\frac32}(\R^3)$ ensures that the above inequality still holds.
Thus we obtain
\begin{align}\label{Ij2}
I_j^2\lesssim d_j2^{j(\frac52-\frac{3}{p})}\|a\|_{\dot{B}_{q,1}^{\frac{3}{q}}}\|\nabla\Pi\|_{L^2}\|\dot{\Delta}_j\Pi\|_{L^p}^{p-1}.
\end{align}
For $I_j^3$,
due to the fact that
\begin{align*}
\sum_{j'\geq j}\dot{S}_{j'+2}\dot{\Delta}_{j}\nabla\Pi\dot{\Delta}_{j'}a
=&\sum_{j'\geq j}(\sum_{k\le j'+1}\dot{\Delta}_{k}\dot{\Delta}_{j}\nabla\Pi)\dot{\Delta}_{j'}a
\nonumber\\
=&\sum_{j'\geq j}((I-\sum_{k\ge j'+2}\dot{\Delta}_{k})\dot{\Delta}_{j}\nabla\Pi)\dot{\Delta}_{j'}a=\sum_{j'\geq j}\dot{\Delta}_{j}\nabla\Pi\dot{\Delta}_{j'}a,
\end{align*}
hence, we can write
\begin{align}\label{}
\dot{T}'_{\dot{\Delta}_j\nabla\Pi}a=\sum_{j'-j=-1,-2}\dot{\Delta}_{j'}a\dot{S}_{j'+2}\dot{\Delta}_{j}\nabla\Pi
+\sum_{j'\geq j}\dot{\Delta}_{j'}a\dot{\Delta}_{j}\nabla\Pi.
\end{align}

From which, we get by applying Lemma 8 in Appendix B of \cite{danchin2010} that
\begin{align*}
I_j^3=&-\sum_{j'-j=-1,-2}
\int_{\R^3}\diverg\big(\dot{\Delta}_{j'}a\dot{S}_{j'+2}\dot{\Delta}_{j}\nabla\Pi\big)\cdot|\dot{\Delta}_j\Pi|^{p-2}\dot{\Delta}_j\Pi dx\nonumber\\
&+(p-1)\sum_{j'\geq j}\int_{\R^3}\dot{\Delta}_{j'}a|\dot{\Delta}_{j}\nabla\Pi|^2\cdot|\dot{\Delta}_j\Pi|^{p-2}dx
\nonumber\\\stackrel{\rm{def}}{=}&I_j^{3,1}+I_j^{3,2}.
\end{align*}
Then it is easy to observe for $\frac1p-\frac1q\leq\frac12$ that
\begin{align}
I_j^{3,1}
\lesssim&2^{j}\sum_{j'-j=-1,-2}\|\dot{\Delta}_{j'}a\|_{L^{\frac{2p}{2-p}}}\|\dot{\Delta}_{j}\nabla\Pi\|_{L^2}\|\dot{\Delta}_j\Pi\|_{L^p}^{p-1}\nonumber\\
\lesssim&d_j2^{j(\frac52-\frac{3}{p})}\|a\|_{\dot{B}_{q,1}^{\frac{3}{q}}}\|\nabla\Pi\|_{L^2}\|\dot{\Delta}_j\Pi\|_{L^p}^{p-1}.
\end{align}
While the assumption $p\in(\frac{1+\sqrt{17}}{4},2]$ ensures that $\frac{1}{p-1}<\frac{2p}{2-p}$.
In the case when $\max\{p,\frac{1}{p-1}\}<q\leq\frac{2p}{2-p}$, we have $(p-2)q'+1>0$ so that we can use a similar approximate argument
as in the proof of Lemma A.5 in the appendix of \cite{danchincpde2001} to obtain
\begin{align*}
&\big\||\dot{\Delta}_{j}\nabla\Pi|^2\cdot|\dot{\Delta}_j\Pi|^{p-2}\big\|_{L^{q'}}^{q'}\nonumber\\
=&\int_{\R^3}|\dot{\Delta}_j\Pi|^{(p-2)q'}\dot{\Delta}_{j}\nabla\Pi\cdot\dot{\Delta}_{j}\nabla\Pi|\dot{\Delta}_{j}\nabla\Pi|^{2q'-2}dx\nonumber\\
=&-\frac{1}{(p-2)q'+1}\int_{\R^3}|\dot{\Delta}_j\Pi|^{(p-2)q'}\dot{\Delta}_{j}\Pi
\cdot\diverg\big(\dot{\Delta}_{j}\nabla\Pi|\dot{\Delta}_{j}\nabla\Pi|^{2q'-2}\big)dx.
\end{align*}
Denoting $\frac1r\stackrel{\rm{def}}{=}\frac1p-\frac1q\leq\frac12$ and using the H$\rm{\ddot{o}}$lder inequality and Lemma \ref{bernstein} gives
\begin{align*}
&\big\||\dot{\Delta}_{j}\nabla\Pi|^2\cdot|\dot{\Delta}_j\Pi|^{p-2}\big\|_{L^{q'}}^{q'}\nonumber\\
\lesssim&\big\||\dot{\Delta}_j\Pi|^{(p-2)q'+1}\big\|_{L^{\frac{p}{(p-2)q'+1}}}
\big\||\dot{\Delta}_{j}\nabla\Pi|^{q'-1}\big\|_{L^{\frac{p}{q'-1}}}\big\||\dot{\Delta}_{j}\nabla\Pi|^{q'-1}\big\|_{L^{\frac{r}{q'-1}}}
\big\|\nabla^2\dot{\Delta}_{j}\Pi\big\|_{L^{r}}\nonumber\\
\lesssim&2^{jq'(\frac52-\frac3p+\frac3q)}\big\|\dot{\Delta}_j\Pi\big\|_{L^{p}}^{(p-1)q'}\big\|\nabla\dot{\Delta}_{j}\Pi\big\|_{L^{2}}^{q'},
\end{align*}
which implies
\begin{align}\label{Ij32}
I_j^{3,2}\lesssim&\sum_{j'\geq j}\|\dot{\Delta}_{j'}a\|_{L^{q}}
\big\||\dot{\Delta}_{j}\nabla\Pi|^2\cdot|\dot{\Delta}_j\Pi|^{p-2}\big\|_{L^{q'}}\nonumber\\
\lesssim&d_j2^{j(\frac52-\frac{3}{p})}\|a\|_{\dot{B}_{q,1}^{\frac{3}{q}}}\|\nabla\Pi\|_{L^2}\big\|\dot{\Delta}_j\Pi\big\|_{L^{p}}^{p-1}.
\end{align}
Similarly, \eqref{Ij32} is valid for $q\leq\max\{p,\frac{1}{p-1}\}$ according to embedding.
Summing up the inequalities \eqref{Ij1}--\eqref{Ij32} results in \eqref{Ijpq}.
\end{proof}

\noindent Now, let us go back to the estimate of (\ref{Ee1}).

\rm{(\romannumeral1)} When $p\in(\frac{1+\sqrt{17}}{4},2]$, substituting \eqref{Ijpq} into \eqref{Ee1} leads to
\begin{align*}
\|\dot{\Delta}_j\nabla\Pi\|_{L^p}\lesssim\|\dot{\Delta}_j\mathbb{Q}F\|_{L^p}+d_j2^{j(\frac32-\frac{3}{p})}\|a\|_{\dot{B}_{q,1}^{\frac{3}{q}}}\|\nabla\Pi\|_{L^2},
\end{align*}
which along with \eqref{elliptic estimateL2} and the embedding $\dot{B}_{p,2}^{-\frac32+\frac3p}(\R^3)\hookrightarrow L^2(\R^3)$, $l^1\hookrightarrow l^2$ gives
\begin{align}\label{Ee2}
\|\nabla\Pi\|_{\dot{B}_{p,2}^{-\frac32+\frac3p}}
\lesssim\big(1+\|a\|_{\dot{B}_{q,1}^{\frac{3}{q}}}\big)\|\mathbb{Q}F\|_{\dot{B}_{p,2}^{-\frac32+\frac3p}}.
\end{align}

Next, we consider the case when $p\in[2,\frac{5+\sqrt{17}}{2})$ with $\frac1p+\frac1q\geq\frac12$.
In this case, motivated by \cite{abidi2012,abidi2013}, we shall use a duality argument:
\begin{align}\label{Ee3}
\|\nabla\Pi\|_{\dot{B}_{p,2}^{-\frac32+\frac3p}}=\sup_{\|g\|_{\dot{B}_{p',2}^{\frac{3}{2}-\frac3p}}=1}\langle\nabla\Pi,g\rangle
=\sup_{\|g\|_{\dot{B}_{p',2}^{\frac{3}{p'}-\frac32}}=1}-\langle\Pi,\diverg g\rangle,
\end{align}
where $\langle\cdot,\cdot\rangle$ denotes the duality bracket between $\mathscr{S}'(\R^3)$ and $\mathscr{S}(\R^3)$.
Noticing that $p'\in(\frac{1+\sqrt{17}}{4},2]$ and $\frac{1}{p'}-\frac1q\leq\frac12$,
then applying \eqref{Ee2} ensures that for any $g\in\dot{B}_{p',2}^{\frac{3}{p'}-\frac32}(\R^3)$,
there exists a unique solution $\nabla P_g\in\dot{B}_{p',2}^{\frac{3}{p'}-\frac32}(\R^3)$ to the elliptic equation
\begin{align*}
\diverg\big((1+a)\nabla P_g\big)=\diverg g,
\end{align*}
such that
\begin{align}\label{Ee4}
\|\nabla P_g\|_{\dot{B}_{p',2}^{\frac{3}{p'}-\frac32}}
\lesssim\big(1+\|a\|_{\dot{B}_{q,1}^{\frac{3}{q}}}\big)\|g\|_{\dot{B}_{p',2}^{\frac{3}{p'}-\frac32}}.
\end{align}
We proceed
\begin{align}\label{}
-\langle\Pi,\diverg g\rangle=&-\langle\Pi,\diverg\big((1+a)\nabla P_g\big)\rangle=-\langle\diverg\big((1+a)\nabla\Pi\big),P_g\rangle\nonumber\\
=&-\langle\diverg F,P_g\rangle=\langle\mathbb{Q}F,\nabla P_g\rangle
\leq\|\mathbb{Q}F\|_{\dot{B}_{p,2}^{-\frac32+\frac3p}}\|\nabla P_g\|_{\dot{B}_{p',2}^{\frac{3}{p'}-\frac32}},
\end{align}
which along with \eqref{Ee3} and \eqref{Ee4} implies \eqref{elliptic estimate1}.

This completes the proof of the proposition.
\end{proof}
%%%%%%%%%%%%%%%%%%%%%%%%%%%%%%%%%%%%%%%%%%%%%%%%%%%%%%%%%%%%%%%%%%%%%%%%%%%%%%%%%%%%%%%%%%%%%%%%%%%%%%%%%%%%%%%%
%%%%%%%%%%%%%%%%%%%%%%%%%%%%%%%%%%%%%%%%%%%%%%%%%%%%%%%%%%%%%%%%%%%%%%%%%%%%%%%%%%%%%%%%%%%%%%%%%%%%%%%%

%%%%%%%%%%%%%%%%%%%%%%%%%%%%%%%%%%%%%%%%%%%%%%%%%%%%%%%%%%%%%%%%%%%%%%%%%%%%%%%%%%%%%%%%%%%%%%%%%%%%%%%%%%%%%%%%%%%%%%%%%%%
\section{Linear estimates}
With the pressure estimates in hand, now, we are going to give the linear estimates for the inhomogeneous incompressible Navier-Stokes equations, more precisely, we can get the following propositon:
\begin{proposition}\label{suduchangxianxing}
Assume   $1< q\le p$ with $p\in[3,\frac{5+\sqrt{17}}{2})$ and $\frac1p+\frac1q\geq\frac12$,   $u_0\in\dot{B}_{p,1}^{-1+\frac3p}(\R^3)$ and $a\in L_T^\infty(\dot{B}_{q,1}^{\frac{3}{q}}(\R^3))$,
with $1+a\geq\kappa>0$.
Let $f\in L_T^1(\dot{B}_{p,1}^{-1+\frac3p}(\R^3))\cap L_T^1(\dot{B}_{p,1}^{-\frac32+\frac3p}(\R^3))$,
 $(u,\nabla\Pi)\in C([0,T];\dot{B}_{p,1}^{-1+\frac3p}(\R^3))\cap L_T^1(\dot{B}_{p,1}^{1+\frac3p}(\R^3))
\times L_T^1(\dot{B}_{p,1}^{-1+\frac3p}(\R^3))$ solve
\begin{eqnarray}\label{Model5}
\left\{\begin{aligned}
&\partial_t u-\diverg((1+a)\nabla u)+(1+a)\nabla\Pi=f,\\
&\mathrm{div}u=0,\\
&u|_{t=0}=u_0.
\end{aligned}\right.
\end{eqnarray}
Then there holds for $t\in[0,T]$
\begin{align}\label{suduguji}
&\|u\|_{\widetilde{L}_t^\infty(\dot{B}_{p,1}^{-1+\frac3p})}+\|u\|_{L_t^1(\dot{B}_{p,1}^{1+\frac3p})}
+\|\nabla\Pi\|_{L_t^1(\dot{B}_{p,1}^{-1+\frac3p})}\nonumber\\
\lesssim&\|u_0\|_{\dot{B}_{p,1}^{-1+\frac3p}}+\|f\|_{L_t^1(\dot{B}_{p,1}^{-1+\frac3p})}+2^{m}\|a\|_{\widetilde{L}_t^\infty(\dot{B}_{q,1}^{\frac{3}{q}})}\|u\|_{L_t^1(\dot{B}_{p,1}^{\frac{3}{p}})}
\nonumber\\
&+2^{\frac m2}\big(1+\|a\|_{\widetilde{L}_t^\infty(\dot{B}_{q,1}^{\frac{3}{q}})}\big)^2(\|f\|_{L_t^1(\dot{B}_{p,1}^{-\frac32+\frac3p})}
+\|a\|_{L_t^\infty(\dot{B}_{q,1}^{\frac{3}{q}})}\|u\|_{L_t^1(\dot{B}_{p,1}^{\frac12+\frac{3}{p}})}),
\end{align}
provided that
\begin{align}\label{A122}
\big(1+\|a\|_{L_T^\infty(\dot{B}_{q,1}^{\frac{3}{q}})}\big)^2\|a-\dot{S}_m a\|_{L_T^\infty(\dot{B}_{q,1}^{\frac{3}{q}})}
\leq c_0
\end{align}
for some sufficiently small positive constant $c_0$ and some integer $m\in \mathbb{Z}$.
\end{proposition}
\begin{proof}
We first use the decomposition $\mathrm{Id}=\dot{S}_m+(\mathrm{Id}-\dot{S}_m)$ to turn the $u$ equation of \eqref{Model5} into
\begin{align}\label{A123}
\partial_t u-\mathrm{div}\big((1+\dot{S}_m a)\nabla u\big)+(1+\dot{S}_m a)\nabla\Pi=f+\dot{E}_m,
\end{align}
with $\dot{E}_m\stackrel{\mathrm{def}}{=}\mathrm{div}\big((a-\dot{S}_m a)\nabla u\big)-(a-\dot{S}_m a)\nabla\Pi$. Then we infer from $1+a\geq\kappa>0$ and \eqref{A122} that
\begin{align}\label{A124}
1+\dot{S}_m a\geq\frac12\kappa.
\end{align}

\noindent{\bf Step 1.} The estimate of $\|u\|_{\widetilde{L}_t^\infty(\dot{B}_{p,1}^{-1+\frac3p})}+\|u\|_{L_t^1(\dot{B}_{p,1}^{1+\frac3p})}.$

Applying $\dot{\Delta}_j\mathbb{P}$ to \eqref{A123} , we arrive at
\begin{align}\label{}
\partial_t\dot{\Delta}_j u-\mathrm{div}\big((1+\dot{S}_m a)\dot{\Delta}_j\nabla u\big)
=&\dot{\Delta}_j\mathbb{P}(f+\dot{E}_m-\dot{S}_m a\nabla\Pi)\nonumber\\
&+\dot{\Delta}_j\mathbb{Q}\big(-\nabla \dot{S}_m a\cdot\nabla u-\dot{S}_m a\Delta u\big)
+\mathrm{div}([\dot{\Delta}_j,\dot{S}_m a]\nabla u),
\end{align}
where we have used the fact that:
$$\mathbb{P}\nabla\Pi=0,\quad \mathrm{div}((1+\dot{S}_m a)\nabla u)=\partial_i((1+\dot{S}_m a)\partial_i u_j).$$
Applying Lemma 8 in the appendix of \cite{danchin2010} and using the H$\mathrm{\ddot{o}}$lder inequality, we get for some positive constant $c$ that
\begin{align}\label{}
&\frac{d}{dt}\|\dot{\Delta}_j u\|_{L^p}+c2^{2j}\|\dot{\Delta}_j u\|_{L^p}\nonumber\\
\lesssim&\|\dot{\Delta}_j f\|_{L^p}+\|\dot{\Delta}_j \dot{E}_m\|_{L^p}+\|\dot{\Delta}_j\mathbb{P}(\dot{S}_m a\nabla\Pi)\|_{L^p}
\nonumber\\
&+\|\dot{\Delta}_j(\nabla \dot{S}_m a\cdot\nabla u)\|_{L^p}
+\|\dot{\Delta}_j\mathbb{Q}(\dot{S}_m a\Delta u)\|_{L^p}
+2^j\|[\dot{\Delta}_j,\dot{S}_m a]\nabla u\|_{L^p}.
\end{align}
After time integration, multiplying $2^{(\frac{3}{p}-1)j}$ and summing up over $j$, we infer
\begin{align}\label{A126}
&\|u\|_{\widetilde{L}_t^\infty(\dot{B}_{p,1}^{-1+\frac3p})}+\|u\|_{L_t^1(\dot{B}_{p,1}^{1+\frac3p})}\nonumber\\
\lesssim&\|u_0\|_{\dot{B}_{p,1}^{-1+\frac3p}}+\|f\|_{L_t^1(\dot{B}_{p,1}^{-1+\frac3p})}
+\|\dot{E}_m\|_{L_t^1(\dot{B}_{p,1}^{-1+\frac3p})}+\|\mathbb{P}(\dot{S}_m a\nabla\Pi)\|_{L_t^1(\dot{B}_{p,1}^{-1+\frac3p})}\nonumber\\
&
+\|\nabla \dot{S}_m a\cdot\nabla u\|_{L_t^1(\dot{B}_{p,1}^{-1+\frac3p})}+\|\mathbb{Q}(\dot{S}_m a\Delta u)\|_{L_t^1(\dot{B}_{p,1}^{-1+\frac3p})}
+\sum_{j\in\mathbb{Z}}2^{\frac{3j}{p}}\|[\dot{\Delta}_j,\dot{S}_m a]\nabla u\|_{L_t^1(L^p)}.
\end{align}
Applying Lemmas \ref{daishu} and \ref{product law1}, one has
\begin{align}\label{A126+1}
\|\nabla \dot{S}_m a\cdot\nabla u\|_{L_t^1(\dot{B}_{p,1}^{-1+\frac3p})}
\lesssim&2^{m}\|a\|_{\widetilde{L}_t^\infty(\dot{B}_{q,1}^{\frac{3}{q}})}\|u\|_{L_t^1(\dot{B}_{p,1}^{\frac{3}{p}})},
\end{align}
\begin{align}\label{A126+2}
\|\mathbb{P}(\dot{S}_m a\nabla\Pi)\|_{L_t^1(\dot{B}_{p,1}^{-1+\frac3p})}
\lesssim&2^{\frac m2}\|a\|_{\widetilde{L}_t^\infty(\dot{B}_{q,1}^{\frac{3}{q}})}\|\nabla\Pi\|_{L_t^1(\dot{B}_{p,2}^{-\frac32+\frac3p})},
\end{align}
\begin{align}\label{A126+3}
\|\dot{E}_m\|_{L_t^1(\dot{B}_{p,1}^{-1+\frac3p})}
\lesssim&\|a-\dot{S}_m a\|_{L_t^\infty(\dot{B}_{q,1}^{\frac{3}{q}})}\| u\|_{L_t^1(\dot{B}_{p,1}^{1+\frac{3}{p}})}+\|a-\dot{S}_m a\|_{\widetilde{L}_t^\infty(\dot{B}_{q,1}^{\frac{3}{q}})}\|\nabla\Pi\|_{L_t^1(\dot{B}_{p,1}^{-1+\frac3p})}.
\end{align}
Yet noticing that $\mathbb{Q}=-\nabla(-\Delta)^{-1}\diverg$ and $\diverg u=0$, we get by applying Bony's decomposition that
$$
\mathbb{Q}(\dot{S}_m a\Delta u)=-\nabla(-\Delta)^{-1}\big(\dot{T}_{\nabla \dot{S}_m a}\Delta u\big)
+\mathbb{Q}\big(\dot{T}_{\Delta u}\dot{S}_m a\big)+\mathbb{Q}\big(\dot{R}(\dot{S}_m a,\Delta u)\big).
$$
Then it is easy to get
\begin{align}\label{A126+4}
\|\dot{\Delta}_j\big(\dot{T}_{\nabla \dot{S}_m a}\Delta u\big)\|_{L^p}
\lesssim&\sum_{|j'-j|\leq4}\|\dot{S}_{j'-1}\nabla \dot{S}_m a\|_{L^\infty}\|\dot{\Delta}_{j'}\Delta u\|_{L^p}
\lesssim d_j2^{j(2-\frac3p)+m}\|a\|_{\dot{B}_{q,1}^{\frac{3}{q}}}\|u\|_{\dot{B}_{p,1}^{\frac{3}{p}}},
\end{align}
\begin{align}\label{A126+5}
\|\dot{\Delta}_j\big(\dot{T}_{\Delta u}\dot{S}_m a\big)\|_{L^p}
\lesssim&\sum_{|j'-j|\leq4}\|\dot{\Delta}_{j'}\dot{S}_m a\|_{L^p}\|\dot{S}_{j'-1}\Delta u\|_{L^\infty}\nonumber\\
\lesssim&\sum_{|j'-j|\leq4}2^{3j'(\frac1q-\frac1p)}\|\dot{\Delta}_{j'}\dot{S}_m a\|_{L^q}\|\dot{S}_{j'-1}\Delta u\|_{L^\infty}\nonumber\\
\lesssim &d_j2^{j(1-\frac3p)+m}\|a\|_{\dot{B}_{q,1}^{\frac{3}{q}}}\|u\|_{\dot{B}_{p,1}^{\frac{3}{p}}}.
\end{align}
Let $\frac 1r=\frac1p+\frac1q$, for the fact that $\frac1p+\frac1q\ge \frac13$ then
\begin{align}\label{A126+6}
\|\dot{\Delta}_j\dot{R}(\dot{S}_m a,\Delta u)\|_{L^p}
\lesssim&2^{3j(\frac{1}{r}-\frac1p)}\sum_{j'\geq j-3}\|\dot{\Delta}_{j'}\dot{S}_m a\|_{L^q}\|\widetilde{\dot{\Delta}}_{j'}\Delta u\|_{L^p}\nonumber\\
\lesssim&2^{\frac{3j}{q}}\sum_{j'\geq j-3}\|\dot{\Delta}_{j'}\dot{S}_m a\|_{L^q}\|\widetilde{\dot{\Delta}}_{j'}\Delta u\|_{L^p}\nonumber\\
\lesssim& d_j2^{j(1-\frac3p)+m}\|a\|_{\dot{B}_{q,1}^{\frac{3}{q}}}\|u\|_{\dot{B}_{p,1}^{\frac{3}{p}}}.
\end{align}

Whence we conclude that
\begin{align}\label{A126+7}
\|\mathbb{Q}(\dot{S}_m a\Delta u)\|_{L_t^1(\dot{B}_{p,1}^{-1+\frac3p})}
\lesssim&2^m\|a\|_{\widetilde{L}_t^\infty(\dot{B}_{q,1}^{\frac{3}{q}})}\|u\|_{L_t^1(\dot{B}_{p,1}^{\frac{3}{p}})}
.
\end{align}
While applying Lemma \ref{jiaohuanzi} leads to
\begin{align}\label{A126+8}
\sum_{j\in\mathbb{Z}}2^{\frac{3}{p}j}\|[\dot{\Delta}_j,\dot{S}_m a]\nabla u\|_{L_t^1(L^p)}
\lesssim2^{m}\|a\|_{\widetilde{L}_t^\infty(\dot{B}_{q,1}^{\frac{3}{q}})}\|u\|_{L_t^1(\dot{B}_{p,1}^{\frac{3}{p}})}.
\end{align}
Plugging the above estimates \eqref{A126+1}--\eqref{A126+3}, \eqref{A126+7}, \eqref{A126+8} into \eqref{A126} and using \eqref{A122} yield
\begin{align}\label{A129}
&\|u\|_{\widetilde{L}_t^\infty(\dot{B}_{p,1}^{-1+\frac3p})}+\|u\|_{L_t^1(\dot{B}_{p,1}^{1+\frac3p})}\nonumber\\
\lesssim&\|u_0\|_{\dot{B}_{p,1}^{-1+\frac3p}}+\|f\|_{L_t^1(\dot{B}_{p,1}^{-1+\frac3p})}
+\|a-\dot{S}_m a\|_{\widetilde{L}_t^\infty(\dot{B}_{q,1}^{\frac{3}{q}})}\|\nabla\Pi\|_{L_t^1(\dot{B}_{p,1}^{-1+\frac3p})}\nonumber\\
&+2^{\frac m2}\|a\|_{\widetilde{L}_t^\infty(\dot{B}_{q,1}^{\frac{3}{q}})}\|\nabla\Pi\|_{L_t^1(\dot{B}_{p,2}^{-\frac32+\frac3p})}
+2^{m}\|a\|_{\widetilde{L}_t^\infty(\dot{B}_{q,1}^{\frac{3}{q}})}\|u\|_{L_t^1(\dot{B}_{p,1}^{\frac{3}{p}})}.
\end{align}

\noindent{\bf Step 2.} The estimate of $\|\nabla\Pi\|_{L_t^1(\dot{B}_{p,1}^{-1+\frac3p})}.$

We first get by taking $\mathrm{div}$ to \eqref{A123} that
\begin{align*}
\mathrm{div}\big((1+\dot{S}_m a)\nabla\Pi\big)=\mathrm{div}\big(f+\dot{E}_m+\nabla \dot{S}_m a\cdot\nabla u+\dot{S}_m a\Delta u\big),
\end{align*}
which implies
\begin{align}\label{}
\mathrm{div}\big((1+\dot{S}_m a)\dot{\Delta}_j\nabla\Pi\big)
=&\mathrm{div}\dot{\Delta}_j\big(f+\dot{E}_m+\nabla \dot{S}_m a\cdot\nabla u+\dot{S}_m a\Delta u
\big)-\mathrm{div}([\dot{\Delta}_j,\dot{S}_m a]\nabla\Pi).
\end{align}
A similar argument as  in \eqref{A126} results in
\begin{align}\label{}
\|\nabla\Pi\|_{L_t^1(\dot{B}_{p,1}^{-1+\frac3p})}
\lesssim&\|f\|_{L_t^1(\dot{B}_{p,1}^{-1+\frac3p})}+\|\dot{E}_m\|_{L_t^1(\dot{B}_{p,1}^{-1+\frac3p})}
+\|\nabla \dot{S}_m a\cdot\nabla u\|_{L_t^1(\dot{B}_{p,1}^{-1+\frac3p})}\nonumber\\
&+\|\mathbb{Q}(\dot{S}_m a\Delta u)\|_{L_t^1(\dot{B}_{p,1}^{-1+\frac3p})}
+\sum_{j\in\mathbb{Z}}2^{(\frac{3}{p}-1)j}\|[\dot{\Delta}_j,\dot{S}_m a]\nabla\Pi\|_{L_t^1(L^p)}.
\end{align}
Again thanks to Lemma \ref{jiaohuanzi2}, one has
\begin{align*}
\sum_{j\in\mathbb{Z}}2^{(\frac{3}{p}-1)j}\|[\dot{\Delta}_j,\dot{S}_m a]\nabla\Pi\|_{L_t^1(L^p)}
\lesssim2^{\frac m2}\|a\|_{\widetilde{L}_t^\infty(\dot{B}_{q,1}^{\frac{3}{q}})}\|\nabla\Pi\|_{L_t^1(\dot{B}_{p,2}^{-\frac32+\frac3p})}.
\end{align*}
Whence we obtain
\begin{align}\label{}
\|\nabla\Pi\|_{L_t^1(\dot{B}_{p,1}^{-1+\frac3p})}\lesssim&\|f\|_{L_t^1(\dot{B}_{p,1}^{-1+\frac3p})}
+\|a-\dot{S}_m a\|_{L_t^\infty(\dot{B}_{q,1}^{\frac{3}{q}})}\|\nabla u\|_{L_t^1(\dot{B}_{p,1}^{\frac{3}{p}})}\nonumber\\
&+2^{m}\|a\|_{\widetilde{L}_t^\infty(\dot{B}_{q,1}^{\frac{3}{q}})}\|u\|_{L_t^1(\dot{B}_{p,1}^{\frac{3}{p}})}
+2^{\frac m2}\|a\|_{\widetilde{L}_t^\infty(\dot{B}_{q,1}^{\frac{3}{q}})}\|\nabla\Pi\|_{L_t^1(\dot{B}_{p,2}^{-\frac32+\frac3p})},
\end{align}
which along with \eqref{A122}, \eqref{A129} gives rise to
\begin{align}\label{nengliangguji}
&\|u\|_{\widetilde{L}_t^\infty(\dot{B}_{p,1}^{-1+\frac3p})}+\|u\|_{L_t^1(\dot{B}_{p,1}^{1+\frac3p})}
+\|\nabla\Pi\|_{L_t^1(\dot{B}_{p,1}^{-1+\frac3p})}\nonumber\\
\lesssim&\|u_0\|_{\dot{B}_{p,1}^{-1+\frac3p}}+\|f\|_{L_t^1(\dot{B}_{p,1}^{-1+\frac3p})}
+2^{m}\|a\|_{\widetilde{L}_t^\infty(\dot{B}_{q,1}^{\frac{3}{q}})}\|u\|_{L_t^1(\dot{B}_{p,1}^{\frac{3}{p}})}
+2^{\frac m2}\|a\|_{\widetilde{L}_t^\infty(\dot{B}_{q,1}^{\frac{3}{q}})}\|\nabla\Pi\|_{L_t^1(\dot{B}_{p,2}^{-\frac32+\frac3p})}.
\end{align}

\noindent{\bf Step 3.} The estimate of $\|\nabla\Pi\|_{L_t^1(\dot{B}_{p,2}^{-\frac32+\frac3p})}.$

Applying $\mathrm{div}$ to the first equation of \eqref{Model5} implies that
\begin{align*}
\mathrm{div}\big((1+a)\nabla\Pi\big)=\mathrm{div}(f+\nabla a\cdot\nabla u+a\Delta u).
\end{align*}
Whence we get by applying Proposition \ref{tuoyuanguji} to the above equation that
\begin{align}\label{}
\|\nabla\Pi\|_{L_t^1(\dot{B}_{p,2}^{-\frac32+\frac3p})}\lesssim&\big(1+\|a\|_{\widetilde{L}_t^\infty(\dot{B}_{q,1}^{\frac{3}{q}})}\big)
\big(\|f\|_{L_t^1(\dot{B}_{p,2}^{-\frac32+\frac3p})}+\|\nabla a\cdot\nabla u\|_{L_t^1(\dot{B}_{p,2}^{-\frac32+\frac3p})}+\|\mathbb{Q}(a\Delta u)\|_{L_t^1(\dot{B}_{p,2}^{-\frac32+\frac3p})}\big).
\end{align}
Yet applying Remark \ref{product law1} yields
\begin{align*}
&\|\nabla a\cdot\nabla u\|_{L_t^1(\dot{B}_{p,2}^{-\frac32+\frac3p})}+\|\mathbb{Q}(a\Delta u)\|_{L_t^1(\dot{B}_{p,2}^{-\frac32+\frac3p})}\nonumber\\
\lesssim&\|\nabla a\|_{L_t^\infty(\dot{B}_{q,1}^{-1+\frac{3}{q}})}\|\nabla u\|_{L_t^1(\dot{B}_{p,1}^{-\frac12+\frac{3}{p}})}
+\| a\|_{L_t^\infty(\dot{B}_{q,1}^{\frac{3}{q}})}\|\Delta u\|_{L_t^1(\dot{B}_{p,1}^{-\frac32+\frac3p})}\nonumber\\
\lesssim&\|a\|_{L_t^\infty(\dot{B}_{q,1}^{\frac{3}{q}})}\|u\|_{L_t^1(\dot{B}_{p,1}^{\frac12+\frac{3}{p}})}.
\end{align*}

%\begin{align}\label{}
%\|\nabla b\cdot\nabla u\|_{L_t^1(\dot{B}_{p,1}^{-\frac32+\frac3p})}
%\lesssim&\|\nabla(b-\dot{S}_m a)\cdot\nabla u\|_{L_t^1(\dot{B}_{p,1}^{-\frac32+\frac3p})}+\|\nabla\dot{S}_m a\cdot\nabla %u\|_{L_t^1(\dot{B}_{p,1}^{-\frac32+\frac3p})}\\
%\lesssim&\|b-\dot{S}_m a\|_{L_t^\infty(\dot{B}_{p,1}^{\frac{3}{p}})}t^{\frac14}\|u\|_{L_t^\infty(\dot{B}_{p,1}^{-1+\frac3p})}^{\frac14}
%\|u\|_{L_t^1(\dot{B}_{p,1}^{1+\frac3p})}^{\frac34}
%+2^{\frac m2}\|b\|_{\widetilde{L}_t^\infty(\dot{B}_{q,1}^{\frac{3}{q}})}\|u\|_{L_t^1(\dot{B}_{p,1}^{\frac{3}{p}})},
%\\
%\|\mathbb{Q}(b\Delta u)\|_{L_t^1(\dot{B}_{p,1}^{-\frac32+\frac3p})}
%\lesssim&\|b\|_{\widetilde{L}_t^\infty(\dot{B}_{q,1}^{\frac{3}{q}})}t^{\frac14}\|u\|_{L_t^\infty(\dot{B}_{p,1}^{-1+\frac3p})}^{\frac14}
%\|u\|_{L_t^1(\dot{B}_{p,1}^{1+\frac3p})}^{\frac34},
%\end{align}

This gives rise to
\begin{align}\label{}
\|\nabla\Pi\|_{L_t^1(\dot{B}_{p,2}^{-\frac32+\frac3p})}
\lesssim&\big(1+\|a\|_{\widetilde{L}_t^\infty(\dot{B}_{q,1}^{\frac{3}{q}})}\big)\big(\|f\|_{L_t^1(\dot{B}_{p,1}^{-\frac32+\frac3p})}+
\|a\|_{L_t^\infty(\dot{B}_{q,1}^{\frac{3}{q}})}\|u\|_{L_t^1(\dot{B}_{p,1}^{\frac12+\frac{3}{p}})}\big).
\end{align}
Substituting the above inequality into \eqref{nengliangguji} we have
\begin{align}\label{}
&\|u\|_{\widetilde{L}_t^\infty(\dot{B}_{p,1}^{-1+\frac3p})}+\|u\|_{L_t^1(\dot{B}_{p,1}^{1+\frac3p})}
+\|\nabla\Pi\|_{L_t^1(\dot{B}_{p,1}^{-1+\frac3p})}\nonumber\\
\lesssim&\|u_0\|_{\dot{B}_{p,1}^{-1+\frac3p}}+\|f\|_{L_t^1(\dot{B}_{p,1}^{-1+\frac3p})}+2^{m}\|a\|_{\widetilde{L}_t^\infty(\dot{B}_{q,1}^{\frac{3}{q}})}\|u\|_{L_t^1(\dot{B}_{p,1}^{\frac{3}{p}})}
\nonumber\\
&+2^{\frac m2}\big(1+\|a\|_{\widetilde{L}_t^\infty(\dot{B}_{q,1}^{\frac{3}{q}})}\big)^2(\|f\|_{L_t^1(\dot{B}_{p,1}^{-\frac32+\frac3p})}
+\|a\|_{L_t^\infty(\dot{B}_{q,1}^{\frac{3}{q}})}\|u\|_{L_t^1(\dot{B}_{p,1}^{\frac12+\frac{3}{p}})}).
\end{align}

This completes the proof of the proposition.
\end{proof}
%%%%%%%%%%%%%%%%%%%%%%%%%%%%%%%%%%%%%%%%%%%%%%%%%%%%%%%%%%%%%%%%%%%%%%%%%%%%%%%%%%%%%%%%%%%%%%%%%%%%%%%%%%%%%%%%%%%%

%%%%%%%%%%%%%%%%%%%%%%%%%%%%%%%%%%%%%%%%%%%%%%%%%%%%%%%%%%%%%%%%%%%%%%%%%%%%%%%%%%%%%%%%%%%%%%%%%%%%%%%%%%%%%%%%%%%%%%%%%%%
Next, we give the linear estimates for the following magnetic field equation:
\begin{eqnarray}\label{cichangxianxing}
\left\{\begin{aligned}
&\partial_t B-\mathrm{div}\big(\widetilde{\sigma}(a)\nabla B\big)+v\cdot\nabla B=g,\\
&\mathrm{div}B=0,\\
&B|_{t=0}=B_0.
\end{aligned}\right.
\end{eqnarray}
\begin{proposition}\label{monimingti}
Assume  $1<p<6$, $1\le q\le \infty$,$\frac1p-\frac1q\le\frac13$, $B_0\in\dot{B}_{p,1}^{-1+\frac3p}(\R^3)$, $\widetilde{\sigma}(a)$ be a smooth, positive function on $[0, \infty),$    and $a\in L_T^\infty(\dot{B}_{q,1}^{\frac{3}{q}}(\R^3))$,
with $1+a\geq\kappa>0$.
Let $g\in L_T^1(\dot{B}_{p,1}^{-1+\frac3p}(\R^3))$,
and $B\in C([0,T];\dot{B}_{p,1}^{-1+\frac3p}(\R^3))\cap L_T^1(\dot{B}_{p,1}^{1+\frac3p}(\R^3))$ solve \eqref{cichangxianxing}.

Then there holds for $t\in[0,T]$
\begin{align}\label{moniguji}
&\|B\|_{\widetilde{L}_t^\infty(\dot{B}_{p,1}^{-1+\frac3p})}+\|B\|_{L_t^1(\dot{B}_{p,1}^{1+\frac3p})}\nonumber\\
\lesssim&\|B_0\|_{\dot{B}_{p,1}^{-1+\frac3p}}+\|g\|_{L_t^1(\dot{B}_{p,1}^{-1+\frac3p})}
+2^{m}\|a\|_{\widetilde{L}_t^\infty(\dot{B}_{q,1}^{\frac{3}{q}})}\|B\|_{L_t^1(\dot{B}_{p,1}^{\frac{3}{p}})}+\int_0^t\| v\|_{\dot{B}_{p,1}^{\frac{3}{p}}}\|B\|_{\dot{B}_{p,1}^{\frac{3}{p}}}d\tau
,
\end{align}
provided that
\begin{align}\label{xiaopingdehao}
\big\|a-\dot{S}_m a\|_{L_T^\infty(\dot{B}_{q,1}^{\frac{3}{q}})}+\|\widetilde{\sigma}(a)-\widetilde{\sigma}(\dot{S}_m a)\|_{L_T^\infty(\dot{B}_{q,1}^{\frac{3}{q}})}
\leq c_0
\end{align}
for some sufficiently small positive constant $c_0$ and some integer $m\in \mathbb{Z}$.
\end{proposition}
\begin{proof}
We may rewrite the system \eqref{cichangxianxing}, after decomposing $\widetilde{\sigma}(a)=\widetilde{\sigma}(\dot{S}_m a)+\widetilde{\sigma}(a)-\widetilde{\sigma}(\dot{S}_m a)$, that
\begin{align}\label{Modeeel5}
\partial_t B-\mathrm{div}\big((\dot{S}_m a)\nabla B\big)+v\cdot\nabla B+\mathrm{div}((\widetilde{\sigma}(a)-\widetilde{\sigma}(\dot{S}_m a))\nabla B)=g.
\end{align}
Applying the operator $\dot{\Delta}_j$ to \eqref{Modeeel5}, using a standard commutator's process , we get
\begin{align*}
&\partial_t\dot{\Delta}_j B-\mathrm{div}\big((\widetilde{\sigma}(\dot{S}_m a))\dot{\Delta}_j\nabla B\big)\nonumber\\
&\quad=\dot{\Delta}_jg+\dot{\Delta}_j( v\cdot\nabla B)+\dot{\Delta}_j\mathrm{div}((\widetilde{\sigma}(a)-\widetilde{\sigma}(\dot{S}_m a))\nabla B)+\mathrm{div}([\dot{\Delta}_j,\widetilde{\sigma}(\dot{S}_m a)]\nabla B).
\end{align*}
Noticing  that $\widetilde{\sigma}(a)-\widetilde{\sigma}(\dot{S}_m a)$ is small enough in norm ${L_T^\infty(\dot{B}_{q,1}^{\frac{3}{q}})}$, it follows that $\widetilde{\sigma}(\dot{S}_m a)\ge \frac{\kappa}{2}$. Taking $L^2$ inner product
with $|\dot{\Delta}_jB|^{p-2}\dot{\Delta}_jB$ and applying Lemma 8 in the appendix of \cite{danchin2010}, we get
\begin{align}\label{yan}
&\|B\|_{\widetilde{L}_t^\infty(\dot{B}_{p,1}^{-1+\frac3p})}+\|B\|_{L_t^1(\dot{B}_{p,1}^{1+\frac3p})}\nonumber\\
\lesssim&\|B_0\|_{\dot{B}_{p,1}^{-1+\frac3p}}+\|g\|_{L_t^1(\dot{B}_{p,1}^{-1+\frac3p})}
+\|v\cdot\nabla B\|_{L_t^1(\dot{B}_{p,1}^{-1+\frac3p})}\nonumber\\
&
+\|\mathrm{div}((\widetilde{\sigma}(a)-\widetilde{\sigma}(\dot{S}_m a))\nabla B)\|_{L_t^1(\dot{B}_{p,1}^{-1+\frac3p})}
+\sum_{j\in\mathbb{Z}}2^{\frac{3j}{p}}\|[\dot{\Delta}_j,\widetilde{\sigma}(\dot{S}_m a)]\nabla u\|_{L_t^1(L^p)}.
\end{align}
By Lemma \ref{jiaohuanzi}, one has
\begin{align}\label{yebo1}
\|v\cdot\nabla B\|_{L_t^1(\dot{B}_{p,1}^{-1+\frac3p})}\lesssim\int_0^t\|v\|_{\dot{B}_{p,1}^{\frac{3}{p}}}\|B\|_{\dot{B}_{p,1}^{\frac{3}{p}}}d\tau,
\end{align}
\begin{align}\label{yebo2}
\sum_{j\in\mathbb{Z}}2^{\frac{3j}{p}}\|[\dot{\Delta}_j,\widetilde{\sigma}(\dot{S}_m a)]\nabla B\|_{L_t^1(L^p)}
\lesssim2^{m}\|a\|_{\widetilde{L}_t^\infty(\dot{B}_{q,1}^{\frac{3}{q}})}\|B\|_{L_t^1(\dot{B}_{p,1}^{\frac{3}{p}})}.
\end{align}
By Lemma \ref{daishu}, we have
\begin{align}\label{yebo3}
\|\mathrm{div}((\widetilde{\sigma}(a)-\widetilde{\sigma}(\dot{S}_m a))\nabla B)\|_{L_t^1(\dot{B}_{p,1}^{-1+\frac3p})}
\lesssim&\|\widetilde{\sigma}(a)-\widetilde{\sigma}(\dot{S}_m a)\|_{L_t^\infty(\dot{B}_{q,1}^{\frac{3}{q}})}\| B\|_{L_t^1(\dot{B}_{p,1}^{1+\frac{3}{p}})}.
\end{align}
Inserting the above estimates \eqref{yebo1}--\eqref{yebo3} into \eqref{yan}, one can finally get
\begin{align}\label{}
&\|B\|_{\widetilde{L}_t^\infty(\dot{B}_{p,1}^{-1+\frac3p})}+\|B\|_{L_t^1(\dot{B}_{p,1}^{1+\frac3p})}\nonumber\\
\lesssim&\|B_0\|_{\dot{B}_{p,1}^{-1+\frac3p}}+\|g\|_{L_t^1(\dot{B}_{p,1}^{-1+\frac3p})}
+2^{m}\|a\|_{\widetilde{L}_t^\infty(\dot{B}_{q,1}^{\frac{3}{q}})}\|B\|_{L_t^1(\dot{B}_{p,1}^{\frac{3}{p}})}+\int_0^t\| v\|_{\dot{B}_{p,1}^{\frac{3}{p}}}\|B\|_{\dot{B}_{p,1}^{\frac{3}{p}}}d\tau
.
\end{align}
This completes the proof of the proposition.
\end{proof}

\section{\bf Local well-posedness of Theorem \ref{zhuyaodingli}}

\subsection{Local existence}
\noindent{\bf Step 1.} Construction of smooth approximate solutions.

Firstly, there exists a sequence $\{(a_0^n,\widetilde{u}_0^n,\widetilde{B}_0^n)\}_{n\in\mathbb{N}}\subset\mathscr{S}(\R^3)$
such that $(a_0^n,\widetilde{u}_0^n,\widetilde{B}_0^n)$ converges to $(a_0,u_0,B_0)$
in $\dot{B}_{q,1}^{\frac3q}(\R^3)\times\big(\dot{B}_{p,1}^{-1+\frac3p}(\R^3)\big)^2$.
Define $u_0^n\stackrel{\mathrm{def}}{=}\mathbb{P}\widetilde{u}_0^n$, $B_0^n\stackrel{\mathrm{def}}{=}\mathbb{P}\widetilde{B}_0^n$, so that $\mathrm{div} u_0^n=\mathrm{div} B_0^n=0$.
Then $(u_0^n,B_0^n)$ belongs to $H^\infty(\R^3)\times H^\infty(\R^3)$ and converges to $(u_0,B_0)$ in $\dot{B}_{p,1}^{-1+\frac3p}(\R^3)\times \dot{B}_{p,1}^{-1+\frac3p}(\R^3)$.
Furthermore, we could assume that
\begin{align}\label{ABA1}
\|a_0^n\|_{L^\infty}\leq2\|a_0\|_{L^\infty},\quad\|a_0^n\|_{\dot{B}_{q,1}^{\frac3q}}\leq2\|a_0\|_{\dot{B}_{q,1}^{\frac3q}},\quad
\|u_0^n\|_{\dot{B}_{p,1}^{-1+\frac3p}}\leq&2\|u_0\|_{\dot{B}_{p,1}^{-1+\frac3p}},\quad
\|B_0^n\|_{\dot{B}_{p,1}^{-1+\frac3p}}\leq&2\|B_0\|_{\dot{B}_{p,1}^{-1+\frac3p}},
\end{align}
and
\begin{align}\label{ABA2}
1+a_0^n=1+a_0+(a_0^n-a_0)\geq\frac12\kappa.
\end{align}
Therefor, applying Theorem 1.1 of \cite{abidi2007} ensures that the MHD system \eqref{mhdmoxing} with the initial data
$(a_0^n,u_0^n,B_0^n)$ admits a unique local in time solution $(a^n,u^n,B^n,\nabla\Pi^n)$ belonging to
\begin{align}\label{ABA3}
C([0,T^n);H^{\alpha+1}(\R^3))&\times C([0,T^n);H^{\alpha}(\R^3))\cap\widetilde{L}_{loc}^1(0,T^n;H^{\alpha+2}(\R^3))\nonumber\\
&\times C([0,T^n);H^{\alpha}(\R^3))\cap\widetilde{L}_{loc}^1(0,T^n;H^{\alpha+2}(\R^3))\times\widetilde{L}_{loc}^1(0,T^n;H^{\alpha}(\R^3))
\end{align}
with $\alpha>\frac12$. Moreover, from the transport equation of \eqref{mhdmoxing}, we deduce that
\begin{align}\label{ABA4}
\|a^n(t)\|_{L^\infty}=\|a_0^n\|_{L^\infty}\leq2\|a_0\|_{L^\infty},\ \ \forall t\in[0,T^n),
\end{align}
\begin{align}\label{ABA5}
1+\inf_{(t,x)\in[0,T^n)\times\R^3}a^n(t,x)=1+\inf_{y\in\R^3}a_0^n(y)\geq\frac12\kappa.
\end{align}
\noindent{\bf Step 2.} Uniform estimates to the approximate solutions.

Next let us turn to the uniform estimates for the approximate solutions consequences, that is, we shall prove that there exists a positive time $T<\inf_{n\in\mathbb{N}}T^n$
such that $(a^n,u^n,B^n,\nabla\Pi^n)$ is uniformly bounded in the space
$$
E_T\stackrel{\mathrm{def}}{=}\widetilde{L}_T^\infty(\dot{B}_{q,1}^{\frac3q})
\times\widetilde{L}_T^\infty(\dot{B}_{p,1}^{-1+\frac3p})\cap L_T^1(\dot{B}_{p,1}^{1+\frac3p})\times\widetilde{L}_T^\infty(\dot{B}_{p,1}^{-1+\frac3p})\cap L_T^1(\dot{B}_{p,1}^{1+\frac3p})
\times L_T^1(\dot{B}_{p,1}^{-1+\frac3p}).
$$

For this, let
$\big(u_F(t),u_F^n(t)\big)
\stackrel{\mathrm{def}}{=}\big(e^{ t\Delta}u_0,e^{ t\Delta}u_0^n\big)$,
$\big(B_F(t),B_F^n(t)\big)
\stackrel{\mathrm{def}}{=}\big(e^{\sigma t\Delta}B_0,e^{\sigma t\Delta}B_0^n\big)$
with $\sigma\stackrel{\mathrm{def}}{=}\widetilde{\sigma}(0)$. Then it is easy to observe that
\begin{align}\label{ABA6}
\|(u_F^n,B^n_F)\|_{\widetilde{L}^\infty(\R^+;\dot{B}_{p,1}^{-1+\frac3p})}
+\|( u_F^n,\sigma B^n_F)\|_{L^1(\R^+;\dot{B}_{p,1}^{1+\frac3p})}
\lesssim\|(u_0^n,B_0^n)\|_{\dot{B}_{p,1}^{-1+\frac3p}}\lesssim\|(u_0,B_0)\|_{\dot{B}_{p,1}^{-1+\frac3p}},
\end{align}
and
\begin{align}\label{ABA8}
\|u_F^n\|_{L_T^1(\dot{B}_{p,1}^{1+\frac3p})}
\lesssim&\|u_F\|_{L_T^1(\dot{B}_{p,1}^{1+\frac3p})}+\|u_F^n-u_F\|_{L_T^1(\dot{B}_{p,1}^{1+\frac3p})}\nonumber\\
\lesssim&\|u_F\|_{L_T^1(\dot{B}_{p,1}^{1+\frac3p})}+\|u_0^n-u_0\|_{\dot{B}_{p,1}^{-1+\frac3p}},
\end{align}
\begin{align}\label{ABA10}
\|B_F^n\|_{L_T^1(\dot{B}_{p,1}^{1+\frac3p})}
\lesssim&\|B_F\|_{L_T^1(\dot{B}_{p,1}^{1+\frac3p})}+\|B_F^n-B_F\|_{L_T^1(\dot{B}_{p,1}^{1+\frac3p})}\nonumber\\
\lesssim&\|B_F\|_{L_T^1(\dot{B}_{p,1}^{1+\frac3p})}+\|B_0^n-B_0\|_{\dot{B}_{p,1}^{-1+\frac3p}}.
\end{align}
Thus, for any $\varepsilon>0$, there exist a number $k=k(\varepsilon)\in\mathbb{N}$
and a positive time $T=T(\varepsilon,u_0)$ such that
\begin{align}\label{eq3}
\sup_{n\geq k}\|(u_F^n,B_F^n)\|_{L_{T}^1(\dot{B}_{p,1}^{1+\frac3p})}\leq\varepsilon.
\end{align}

Denote by $\bar{u}^n\stackrel{\mathrm{def}}{=}u^n-u_F^n$, $\bar{B}^n\stackrel{\mathrm{def}}{=}B^n-B_F^n$. Then the system for $(a^n,\bar{u}^n,\bar{B}^n,\nabla\Pi^n)$ reads
\begin{eqnarray}\label{bijinfangcheng}
\left\{\begin{aligned}
&\partial_t a^n+(u_F^n+\bar{u}^n)\cdot\nabla a^n =0,\\
&\partial_t\bar{u}^n-\diverg((1+a^n)\nabla \bar{u}^n) +(1+a^n)\nabla \Pi^n=F_n,\\
&\partial_t\bar{B}^n-\diverg\big(\widetilde{\sigma}(a^n)\nabla \bar{B}^n\big)+u_F^n \cdot \nabla \bar{B}^n=G_n,\\
&\mathrm{div}\bar{u}^n =\bar{B}^n=0,\\
&(a^n,\bar{u}^n,\bar{B}^n)|_{t=0}=(a_0^n,0,0),
\end{aligned}\right.
\end{eqnarray}
where
\begin{align*}
F_n=&-\nabla a^n\cdot\nabla \bar{u}^n-u_F^n\cdot \nabla \bar{u}^n-u_F^n\cdot \nabla u_F^n-\bar{u}^n\cdot \nabla u_F^n-\bar{u}^n\cdot \nabla \bar{u}^n
\\&+a^n\Delta u_F^n+(1+a^n)(B_F^n\cdot \nabla B_F^n+B_F^n\cdot \nabla \bar{B}^n+\bar{B}^n\cdot \nabla B_F^n+\bar{B}^n\cdot \nabla \bar{B}^n),
\end{align*}
\begin{align*}
G_n=&\bar{B}^n\cdot \nabla u_F^n-\bar{u}^n\cdot \nabla \bar{B}^n+\bar{B}^n\cdot \nabla \bar{u}^n-u_F^n\cdot \nabla B_F^n\\
&-\bar{u}^n\cdot \nabla B_F^n+B_F^n\cdot \nabla u_F^n+B_F^n\cdot \nabla \bar{u}^n+\diverg\big((\widetilde{\sigma}(a^n)-\widetilde{\sigma}(0))\nabla B_F^n\big).
\end{align*}
For notational simplicity, we denote by $A^n(t)\stackrel{\mathrm{def}}{=}\|a^n\|_{\widetilde{L}_t^\infty(\dot{B}_{q,1}^{\frac3q})}$ and
$$
Z^n(t)\stackrel{\mathrm{def}}{=}\|(\bar{u}^n,\bar{B}^n)\|_{\widetilde{L}_t^\infty(\dot{B}_{p,1}^{-1+\frac3p})}
+\|(\bar{u}^n,\bar{B}^n,\nabla\Pi^n)\|_{L_t^1(\dot{B}_{p,1}^{1+\frac3p})}+\|\nabla\Pi^n\|_{L_t^1(\dot{B}_{p,1}^{-1+\frac3p})}.
$$
By Propositions  \ref{suduchangxianxing}, \ref{monimingti}, we have
\begin{align}\label{ABA12}
Z^n(t)
\lesssim&\|(F_n,G_n)\|_{L_t^1(\dot{B}_{p,1}^{-1+\frac3p})}
+2^{\frac m2}\big(1+\|a^n\|_{\widetilde{L}_t^\infty(\dot{B}_{q,1}^{\frac{3}{q}})}\big)^2\|F_n\|_{L_t^1(\dot{B}_{p,1}^{-\frac32+\frac3p})}
\nonumber\\
&+2^{m}\|a^n\|_{\widetilde{L}_t^\infty(\dot{B}_{q,1}^{\frac{3}{q}})}\|(\bar{u}^n,\bar{B}^n)\|_{L_t^1(\dot{B}_{p,1}^{\frac{3}{p}})}+\int_0^t\| u_F^n\|_{\dot{B}_{p,1}^{\frac{3}{p}}}\|\bar{B}^n\|_{\dot{B}_{p,1}^{\frac{3}{p}}}d\tau\nonumber\\
&+2^{\frac m2}\big(1+\|a^n\|_{\widetilde{L}_t^\infty(\dot{B}_{q,1}^{\frac{3}{q}})}\big)^3
\|\bar{u}^n\|_{L_t^1(\dot{B}_{p,1}^{\frac12+\frac{3}{p}})}.
\end{align}
According to  $\mathrm{Id}=\dot{S}_m+(\mathrm{Id}-\dot{S}_m)$ and  Lemma \ref{daishu}, one has
\begin{align}\label{ABA13}
\|\nabla a^n\cdot\nabla \bar{u}^n\|_{L_t^1(\dot{B}_{p,1}^{-1+\frac3p})}
\lesssim&\|\nabla \dot{S}_ma^n\cdot\nabla \bar{u}^n\|_{L_t^1(\dot{B}_{p,1}^{-1+\frac3p})}+\|\nabla(a^n- \dot{S}_ma^n)\cdot\nabla \bar{u}^n\|_{L_t^1(\dot{B}_{p,1}^{-1+\frac3p})}\nonumber\\
\lesssim&2^{m}\|a^n\|_{\widetilde{L}_t^\infty(\dot{B}_{q,1}^{\frac{3}{q}})}\|\bar{u}^n\|_{L_t^1(\dot{B}_{p,1}^{\frac{3}{p}})}
+\|a^n-\dot{S}_ma^n\|_{\widetilde{L}_t^\infty(\dot{B}_{q,1}^{\frac{3}{q}})}\|\bar{u}^n\|_{L_t^1(\dot{B}_{p,1}^{1+\frac{3}{p}})}.
\end{align}
By Remark \ref{product law1}, we have
\begin{align}\label{ABA14}
&\|\nabla a^n\cdot\nabla \bar{u}^n\|_{L_t^1(\dot{B}_{p,2}^{-\frac32+\frac3p})}\nonumber\\
\lesssim&\|\nabla \dot{S}_ma^n\cdot\nabla \bar{u}^n\|_{L_t^1(\dot{B}_{p,2}^{-\frac32+\frac3p})}+\|\nabla(a^n- \dot{S}_ma^n)\cdot\nabla \bar{u}^n\|_{L_t^1(\dot{B}_{p,2}^{-\frac32+\frac3p})}\nonumber\\
\lesssim&\|\nabla \dot{S}_ma^n\|_{L_t^\infty(\dot{B}_{q,1}^{-1+\frac{3}{q}})}\|\nabla \bar{u}^n\|_{L_t^1(\dot{B}_{p,1}^{-\frac12+\frac{3}{p}})}
+\|\nabla(a^n- \dot{S}_ma^n)
\|_{L_t^\infty(\dot{B}_{q,1}^{-1+\frac{3}{q}})}\|\nabla \bar{u}^n\|_{L_t^1(\dot{B}_{p,1}^{-\frac12+\frac{3}{p}})}
\nonumber\\
\lesssim&t^{\frac14}\|a^n\|_{\widetilde{L}_t^\infty(\dot{B}_{q,1}^{\frac{3}{q}})}\|\bar{u}^n\|_{L_t^\infty(\dot{B}_{p,1}^{-1+\frac3p})}^{\frac14}
\|\bar{u}^n\|_{L_t^1(\dot{B}_{p,1}^{1+\frac3p})}^{\frac34}+\|a^n-\dot{S}_ma^n\|_{\widetilde{L}_t^\infty(\dot{B}_{q,1}^{\frac{3}{q}})}
\|\bar{u}^n\|_{L_t^1(\dot{B}_{p,1}^{\frac12+\frac{3}{p}})},
\end{align}
in which we have used the following interpolation inequality:
\begin{align*}%\label{ABA16}
\|\bar{u}^n\|_{L_t^1(\dot{B}_{p,1}^{\frac12+\frac{3}{p}})}\le t^{\frac14}\|\bar{u}^n\|_{L_t^\infty(\dot{B}_{p,1}^{-1+\frac3p})}^{\frac14}
\|\bar{u}^n\|_{L_t^1(\dot{B}_{p,1}^{1+\frac3p})}^{\frac34}.
\end{align*}
Similarly, we can get
\begin{align}\label{ABA17}
\|a^n\Delta u_F^n\|_{L_t^1(\dot{B}_{p,1}^{-1+\frac3p})}+\|a^n\Delta u_F^n\|_{L_t^1(\dot{B}_{p,2}^{-\frac32+\frac3p})}
\lesssim&\|a^n\|_{\widetilde{L}_t^\infty(\dot{B}_{q,1}^{\frac{3}{q}})}(\|u_F^n\|_{L_t^1(\dot{B}_{p,1}^{1+\frac3p})}+\|u_F^n\|_{L_t^1(\dot{B}_{p,1}^{\frac3p+\frac12})}).
\end{align}

Yet thanks to $\mathrm{div}u_F^n=\mathrm{div}\bar{u}^n=0$, we get by using product laws and interpolation inequality in Besov spaces that
\begin{align}\label{ABA18}
&\|\bar{u}^n\cdot\nabla \bar{u}^n+u_F^n\cdot\nabla u_F^n+\bar{u}^n\cdot\nabla u_F^n+u_F^n\cdot\nabla \bar{u}^n\|_{L_t^1(\dot{B}_{p,1}^{-1+\frac3p})}\nonumber\\
\lesssim&\|\bar{u}^n\otimes \bar{u}^n+u_F^n\otimes u_F^n+\bar{u}^n\otimes u_F^n+u_F^n\otimes\bar{u}^n\|_{L_t^1(\dot{B}_{p,1}^{\frac3p})}\nonumber\\
\lesssim&\int_0^t\|\bar{u}^n\|_{\dot{B}_{p,1}^{\frac3p}}\|\bar{u}^n\|_{\dot{B}_{p,1}^{\frac3p}}
+\|u_F^n\|_{\dot{B}_{p,1}^{\frac3p}}\|u_F^n\|_{\dot{B}_{p,1}^{\frac3p}}
+\|\bar{u}^n\|_{\dot{B}_{p,1}^{\frac3p}}\|u_F^n\|_{\dot{B}_{p,1}^{\frac3p}}
d\tau\nonumber\\
\lesssim&\int_0^t
\|\bar{u}^n\|_{\dot{B}_{p,1}^{-1+\frac3p}}\|\bar{u}^n\|_{\dot{B}_{p,1}^{1+\frac3p}}+\|u_F^n\|_{\dot{B}_{p,1}^{-1+\frac3p}}\|u_F^n\|_{\dot{B}_{p,1}^{1+\frac3p}}
d\tau\nonumber\\
\lesssim&(Z^n(t))^2+\|u_0\|_{\dot{B}_{p,1}^{-1+\frac3p}}\|u_F^n\|_{L_t^1(\dot{B}_{p,1}^{1+\frac3p})}.
\end{align}
Similarly,
\begin{align}\label{ABA21}
\|\diverg\big((\widetilde{\sigma}(a^n)-\widetilde{\sigma}(0))\nabla B_F^n\big)\|_{L_t^1(\dot{B}_{p,1}^{-1+\frac3p})}
\lesssim\|a^n\|_{\widetilde{L}_t^\infty(\dot{B}_{q,1}^{\frac{3}{q}})}\|B_F^n\|_{L_t^1(\dot{B}_{p,1}^{1+\frac3p})},
\end{align}
\begin{align}\label{ABA19}
&\|(1+a^n)(\bar{B}^n\cdot\nabla \bar{B}^n+B_F^n\cdot\nabla B_F^n+\bar{B}^n\cdot\nabla B_F^n+B_F^n\cdot\nabla \bar{B}^n)\|_{L_t^1(\dot{B}_{p,1}^{-1+\frac3p})}\nonumber\\
\lesssim&\big(1+\|a^n\|_{\widetilde{L}_t^\infty(\dot{B}_{q,1}^{\frac{3}{q}})}\big)
((Z^n(t))^2+\|B_0\|_{\dot{B}_{p,1}^{-1+\frac3p}}\|u_F^n\|_{L_t^1(\dot{B}_{p,1}^{1+\frac3p})}),
\end{align}
\begin{align}\label{ABA22}
&\|\bar{B}^n\cdot \nabla u_F^n-\bar{u}^n\cdot \nabla \bar{B}^n+\bar{B}^n\cdot \nabla \bar{u}^n-u_F^n\cdot \nabla B_F^n-\bar{u}^n\cdot \nabla B_F^n+B_F^n\cdot \nabla u_F^n+B_F^n\cdot \nabla \bar{u}^n\|_{L_t^1(\dot{B}_{p,1}^{-1+\frac3p})}\nonumber\\
\lesssim&(Z^n(t))^2+\|(u_0,B_0)\|_{\dot{B}_{p,1}^{-1+\frac3p}}\|(u_F^n,B_F^n)\|_{L_t^1(\dot{B}_{p,1}^{1+\frac3p})}.
\end{align}
Thus
\begin{align}\label{ABA23}
\|G_n\|_{L_t^1(\dot{B}_{p,1}^{-1+\frac3p})}
\lesssim(Z^n(t))^2+\|(u_0,B_0)\|_{\dot{B}_{p,1}^{-1+\frac3p}}\|(u_F^n,B_F^n)\|_{L_t^1(\dot{B}_{p,1}^{1+\frac3p})}+\|a^n\|_{\widetilde{L}_t^\infty(\dot{B}_{q,1}^{\frac{3}{q}})}\|B_F^n\|_{L_t^1(\dot{B}_{p,1}^{1+\frac3p})}.
\end{align}
By Lemma \ref{daishu} and interpolation inequality in Chemin-Lerner spaces, we have
\begin{align}\label{ABA24}
&\|\bar{u}^n\cdot\nabla \bar{u}^n+u_F^n\cdot\nabla u_F^n+\bar{u}^n\cdot\nabla u_F^n+u_F^n\cdot\nabla \bar{u}^n\|_{L_t^1(\dot{B}_{p,1}^{-\frac32+\frac3p})}\nonumber\\
\lesssim&\int_0^t\|\bar{u}^n\otimes\bar{u}^n\|_{\dot{B}_{p,1}^{-\frac12+\frac3p}}
+\|u_F^n\otimes u_F^n\|_{\dot{B}_{p,1}^{-\frac12+\frac3p}}+\|\bar{u}^n\otimes u_F^n\|_{\dot{B}_{p,1}^{-\frac12+\frac3p}}
+\|u_F^n\otimes \bar{u}^n\|_{\dot{B}_{p,1}^{-\frac12+\frac3p}}
d\tau\nonumber\\
\lesssim&\int_0^t\|\bar{u}^n\|_{\dot{B}_{p,1}^{-\frac12+\frac3p}}\|\bar{u}^n\|_{\dot{B}_{p,1}^{\frac3p}}
+\|u_F^n\|_{\dot{B}_{p,1}^{-\frac12+\frac3p}}\|u_F^n\|_{\dot{B}_{p,1}^{\frac3p}}+\|\bar{u}^n\|_{\dot{B}_{p,1}^{-\frac12+\frac3p}}\|u_F^n\|_{\dot{B}_{p,1}^{\frac3p}}
d\tau\nonumber\\
\lesssim&\int_0^t\|\bar{u}^n\|_{\dot{B}_{p,1}^{-1+\frac3p}}^{\frac32}\|\bar{u}^n\|_{\dot{B}_{p,1}^{1+\frac3p}}^{\frac12}
+\|u_F^n\|_{\dot{B}_{p,1}^{-1+\frac3p}}^{\frac32}\|u_F^n\|_{\dot{B}_{p,1}^{1+\frac3p}}^{\frac12}d\tau
+\int_0^t\|u_F^n\|_{\dot{B}_{p,1}^{-1+\frac3p}}\|u_F^n\|_{\dot{B}_{p,1}^{1+\frac3p}}
+\|\bar{u}^n\|_{\dot{B}_{p,1}^{-1+\frac3p}}\|\bar{u}^n\|_{\dot{B}_{p,1}^{1+\frac3p}}d\tau\nonumber\\
\lesssim& t^{\frac12}((Z^n(t))^2+\|u_0\|_{\dot{B}_{p,1}^{-1+\frac3p}})+\|u_0\|_{\dot{B}_{p,1}^{-1+\frac3p}}
\|u_F^n\|_{L_t^1(\dot{B}_{p,1}^{1+\frac3p})}+(Z^n(t))^2.
\end{align}

Similarly,
\begin{align}\label{ABA25}
&\|(1+a^n)(\bar{B}^n\cdot\nabla \bar{B}^n+B_F^n\cdot\nabla B_F^n+\bar{B}^n\cdot\nabla B_F^n+B_F^n\cdot\nabla \bar{B}^n)\|_{L_t^1(\dot{B}_{p,1}^{-\frac32+\frac3p})}\nonumber\\
\lesssim& \big(1+\|a^n\|_{\widetilde{L}_t^\infty(\dot{B}_{q,1}^{\frac{3}{q}})}\big)\left(t^{\frac12}((Z^n(t))^2+\|B_0\|_{\dot{B}_{p,1}^{-1+\frac3p}})+\|B_0\|_{\dot{B}_{p,1}^{-1+\frac3p}}
\|B_F^n\|_{L_t^1(\dot{B}_{p,1}^{1+\frac3p})}+(Z^n(t))^2\right).
\end{align}

As a consequence, we obtain
\begin{align}\label{ABA26}
\|(F_n,G_n)\|_{L_t^1(\dot{B}_{p,1}^{-1+\frac3p})}\lesssim&
2^{m}\|a^n\|_{\widetilde{L}_t^\infty(\dot{B}_{q,1}^{\frac{3}{q}})}\|\bar{u}^n\|_{L_t^1(\dot{B}_{p,1}^{\frac{3}{p}})}
+\|a^n-\dot{S}_ma^n\|_{\widetilde{L}_t^\infty(\dot{B}_{q,1}^{\frac{3}{q}})}(\|\bar{u}^n\|_{L_t^1(\dot{B}_{p,1}^{1+\frac{3}{p}})}+\|\bar{u}^n\|_{L_t^1(\dot{B}_{p,1}^{\frac12+\frac{3}{p}})})\nonumber\\
&+\big(1+\|a^n\|_{\widetilde{L}_t^\infty(\dot{B}_{q,1}^{\frac{3}{q}})}\big)((Z ^n(t))^2+\|(u_0,B_0)\|_{\dot{B}_{p,1}^{-1+\frac3p}}\|u_F^n\|_{L_t^1(\dot{B}_{p,1}^{1+\frac3p})}),
\nonumber\\
\lesssim&\varepsilon Z^n(t)+\|a^n-\dot{S}_ma^n\|_{\widetilde{L}_t^\infty(\dot{B}_{q,1}^{\frac{3}{q}})}Z^n(t)\nonumber\\
&+(1+2^{2m}(A^n(t))^2)(Z^n(t))^2+(1+A^n(t))\|(u_0,B_0)\|_{\dot{B}_{p,1}^{-1+\frac3p}})\|(u_F^n,B_F^n)\|_{L_t^1(\dot{B}_{p,1}^{1+\frac3p})},
\end{align}
\begin{align}\label{ABA27}
\|F_n\|_{L_t^1(\dot{B}_{p,1}^{-\frac32+\frac3p})}\lesssim&\varepsilon Z^n(t)+\|a^n-\dot{S}_ma^n\|_{\widetilde{L}_t^\infty(\dot{B}_{q,1}^{\frac{3}{q}})}Z^n(t)+t^{\frac14}A^n(t)Z^n(t)
\nonumber\\
&+
A^n(t)\|u_F^n\|_{L_t^1(\dot{B}_{p,1}^{\frac3p+\frac12})}+
\big(1+A^n(t)\big)t^{\frac12}((Z^n(t))^2
+\|(u_0,B_0)\|_{\dot{B}_{p,1}^{-1+\frac3p}})
\nonumber\\
&
+\big(1+A^n(t)\big)t^{\frac12}(\|(u_0,B_0)\|_{\dot{B}_{p,1}^{-1+\frac3p}}
\|(u_F^n,B^n_F)\|_{L_t^1(\dot{B}_{p,1}^{1+\frac3p})}+(Z^n(t))^2).
\end{align}
Choosing $\varepsilon$ small enough, using the Young inequality and
\begin{align}\label{ABA28}
\big(1+\|a\|_{L_T^\infty(\dot{B}_{q,1}^{\frac{3}{q}})}\big)^2\|a-\dot{S}_m a\|_{L_T^\infty(\dot{B}_{q,1}^{\frac{3}{q}})}
\leq c_0,
\end{align}
we have
\begin{align}\label{ABA29}
Z^n(t)
\lesssim&(1+2^{2m}(A^n(t))^2)(Z^n(t))^2\nonumber\\
&+2^{\frac m2}\big(1+A^n(t)\big)^4\Bigg[t^{\frac14}A^n(t)Z^n(t)
+
A^n(t)\|u_F^n\|_{L_t^1(\dot{B}_{p,1}^{\frac3p+\frac12})}
\nonumber\\
&+
\big(1+A^n(t)\big)\Big(t^{\frac12}((Z^n(t))^2
+\|(u_0,B_0)\|_{\dot{B}_{p,1}^{-1+\frac3p}})+\|(u_0,B_0)\|_{\dot{B}_{p,1}^{-1+\frac3p}}
\|(u_F^n,B^n_F)\|_{L_t^1(\dot{B}_{p,1}^{1+\frac3p})}\Big)\Bigg]
.
\end{align}

On the other hand, applying Proposition \ref{shuyun} to the transport equation of \eqref{bijinfangcheng}, we have for $t\in[0,T^n)$
\begin{align}\label{ABA31}
A^n(t)\leq&\|a_0^n\|_{\dot{B}_{q,1}^{\frac3q}}(C\|u_F^n+\bar{u}^n\|_{L_t^1(\dot{B}_{p,1}^{1+\frac3p})})\nonumber\\
\leq&C\|a_0\|_{\dot{B}_{q,1}^{\frac3q}}\exp(C\big(Z^n(t)+\|u_0\|_{\dot{B}_{p,1}^{-1+\frac3p}}\big)),
\end{align}
and
\begin{align}\label{haode}
&\|a^n-\dot{S}_m a^n\|_{\widetilde{L}_t^\infty(\dot{B}_{q,1}^{\frac3q})}\nonumber\\
\leq&\sum_{j\geq m}2^{\frac{3j}{q}}\|\dot{\Delta}_j a_0^n\|_{L^q}+\|a_0^n\|_{\dot{B}_{q,1}^{\frac3q}}
(\exp\big(C\|u_F^n+\bar{u}^n\|_{L_t^1(\dot{B}_{p,1}^{1+\frac3p})}\big)-1)\nonumber\\
\leq&\sum_{j\geq m}2^{\frac{3j}{q}}\|\dot{\Delta}_j a_0\|_{L^q}+\|a_0-a_0^n\|_{\dot{B}_{q,1}^{\frac3q}}\nonumber\\
&+C\|a_0\|_{\dot{B}_{q,1}^{\frac3q}}(\exp\big(CZ^n(t)+C\|u_F^n\|_{L_t^1(\dot{B}_{p,1}^{1+\frac3p})}\big)-1).
\end{align}
However, for any function $\chi\in\mathcal{D}(\R)$ with $\chi(0)=0$,
the composite function $\chi(a^n)$ with initial data $\chi(a_0^n)$ also solves the renormalized transport equation
$$\partial_t\chi(a^n)+(u_F^n+\bar{u}^n)\cdot\nabla \chi(a^n)=0.$$
Whence a similar argument as in \eqref{haode} leads to
\begin{align}\label{ABA32}
&\|\chi(a^n)-\dot{S}_m \chi(a^n)\|_{\widetilde{L}_t^\infty(\dot{B}_{q,1}^{\frac3q})}\nonumber\\
\leq&\sum_{j\geq m}2^{\frac{3j}{q}}\|\dot{\Delta}_j \chi(a_0^n)\|_{L^q}+\|\chi(a_0^n)\|_{\dot{B}_{q,1}^{\frac3q}}
\left(\exp\big(CZ^n(t)+C\|u_F^n\|_{L_t^1(\dot{B}_{p,1}^{1+\frac3p})}\big)-1\right)\nonumber\\
\leq&\sum_{j\geq m}2^{\frac{3j}{q}}\|\dot{\Delta}_j \chi(a_0)\|_{L^q}
+C\big(1+\|a_0\|_{\dot{B}_{q,1}^{\frac3q}}\big)\|a_0^n-a_0\|_{\dot{B}_{q,1}^{\frac3q}}\nonumber\\
&+C\|a_0\|_{\dot{B}_{q,1}^{\frac3q}}\left(\exp\big(CZ^n(t)+C\|u_F^n\|_{L_t^1(\dot{B}_{p,1}^{1+\frac3p})}\big)-1\right),
\end{align}
where we used
\begin{align}\label{ABA33}
\|\chi(a_0^n)-\chi(a_0)\|_{\dot{B}_{q,1}^{\frac3q}}
=&\left\|(a_0^n-a_0)\int_0^1\chi'(\tau a_0^n+(1-\tau)a_0)d\tau\right\|_{\dot{B}_{q,1}^{\frac3q}}\nonumber\\
\leq&C\big(1+\|a_0\|_{\dot{B}_{q,1}^{\frac3q}}\big)\|a_0^n-a_0\|_{\dot{B}_{q,1}^{\frac3q}}.
\end{align}
As a consequence, we obtain for $t\in[0,T^n)$
\begin{align}\label{ABA34}
\|(\mathrm{Id}-\dot{S}_m)a^n\|_{\widetilde{L}_t^\infty(\dot{B}_{q,1}^{\frac3q})}
\leq&\sum_{j\geq m}2^{\frac{3j}{q}}(\|\dot{\Delta}_ja_0\|_{L^q})
+C\big(1+\|a_0\|_{\dot{B}_{q,1}^{\frac3q}}\big)\|a_0^n-a_0\|_{\dot{B}_{q,1}^{\frac3q}}\nonumber\\
&+C\|a_0\|_{\dot{B}_{q,1}^{\frac3q}}\left(\exp\big(CZ^n(t)+C\|u_F^n\|_{L_t^1(\dot{B}_{p,1}^{1+\frac3p})}\big)-1\right).
\end{align}

Next, for any $n\in\mathbb{N}$, we define
\begin{align}\label{ABA35}
T_*^n\stackrel{\mathrm{def}}{=}\sup\{t\in(0,T^n):\ Z^n(t)\leq 2\varepsilon_0\},
\end{align}
with $\varepsilon_0\in(0,\frac12)$ to be determined. We shall prove $\inf_{n\in\mathbb{N}}T_*^n>0$.

Firstly, we deduce from \eqref{ABA31} for $t\leq T_*^n$ that
\begin{align}\label{ABA36}
A^n(t)\leq C\|a_0\|_{\dot{B}_{q,1}^{\frac3q}}\exp\left(C\big(1+\|u_0\|_{\dot{B}_{p,1}^{-1+\frac3p}}\big)\right)
\stackrel{\mathrm{def}}{=}A_0.
\end{align}
Noticing that $a_0\in\dot{B}_{q,1}^{\frac3q}(\R^3)$,
there exist $m=m(c_0)\in\mathbb{Z}$ and $n_0=n_0(c_0)\in\mathbb{N}$ such that
\begin{align}\label{ABA37}
(1+A_0)^2&\left(\sum_{j\geq m}2^{\frac{3j}{q}}(\|\dot{\Delta}_ja_0\|_{L^q})
+\sup_{n\geq n_0}C\big(1+\|a_0\|_{\dot{B}_{q,1}^{\frac3q}}\big)\|a_0^n-a_0\|_{\dot{B}_{q,1}^{\frac3q}}\right)\leq\frac12 c_0.
\end{align}
Yet thanks to \eqref{eq3}, taking $\varepsilon_0$ and $T_0$ small enough and $n_1\geq n_0$ large enough ensures
\begin{align}\label{ABA38}
C(1+A_0)^2\|a_0\|_{\dot{B}_{q,1}^{\frac3q}}\left(\exp\big(2C\varepsilon_0+C\|u_F^n\|_{L_{T_0}^1(\dot{B}_{p,1}^{1+\frac3p})}\big)-1\right)
\leq\frac12c_0
\end{align}
for any $n\geq n_1$.
Combining \eqref{ABA36}--\eqref{ABA38} implies that \eqref{ABA35} with $T=\min(T_*^n,T_0)$ is fulfilled for any $n\geq n_1$.
Without loss of generality, we may assume $T_*^n\leq T_0$. Then for any $t\leq T_*^n$, we deduce from \eqref{ABA29} that
\begin{align}\label{ABA39}
Z^n(t)
\lesssim&(1+2^{2m}A_0^2)\varepsilon_0Z^n(t)\nonumber\\
&+2^{\frac m2}\big(1+A_0\big)^4\Bigg[t^{\frac14}A_0\varepsilon_0
+
A_0\|u_F^n\|_{L_t^1(\dot{B}_{p,1}^{\frac3p+\frac12})}
\nonumber\\
&+
\big(1+A_0\big)\Big(t^{\frac12}(\varepsilon_0^2
+\|(u_0,B_0)\|_{\dot{B}_{p,1}^{-1+\frac3p}})+\|(u_0,B_0)\|_{\dot{B}_{p,1}^{-1+\frac3p}}
\|(u_F^n,B^n_F)\|_{L_t^1(\dot{B}_{p,1}^{1+\frac3p})}\Big)\Bigg]
.
\end{align}

Finally, taking $\varepsilon_0$ and $T_1$ small enough and $n_2\geq n_1$ large enough ensures for any $n\geq n_2$
\begin{align}\label{ABA41}
(1+2^{2m}A_0^2)\varepsilon_0\leq\frac12,
\end{align}
\begin{align}\label{ABA42}
2^{\frac m2}\big(1+A_0\big)^4{T_1}^{\frac14}A_0+\big(1+A_0\big){T_1}^{\frac12}(\varepsilon_0^2
+\|(u_0,B_0)\|_{\dot{B}_{p,1}^{-1+\frac3p}})\leq\frac{\varepsilon_0}{4},
\end{align}
and
\begin{align}\label{ABA43}
2^{\frac m2}\big(1+A_0\big)^5\|(u_0,B_0)\|_{\dot{B}_{p,1}^{-1+\frac3p}}
\|(u_F^n,B^n_F)\|_{L_{T_1}^1(\dot{B}_{p,1}^{1+\frac3p})}\leq\frac{\varepsilon_0}{4},
\end{align}
which together with \eqref{ABA39} implies
$$Z^n(t)\leq\varepsilon_0,\ \ \forall t\leq\min(T_*^n,T_1),\ n\geq n_2.$$
However, by the definition of $T_*^n$, we eventually infer $T_*^n\geq T_1$ and $\sup_{n\geq n_2}Z^n(T_1)\leq \varepsilon_0$,
which along with \eqref{ABA6} and \eqref{ABA36} ensures that $(a^n,u^n,B^n,\nabla\Pi^n)$ is uniformly bounded in $E_{T_1}$.

\noindent{\bf Step 3.} Convergence.

 The convergence of $(u_F^n,  B_F^n)$ to $(u_F, B_F)$ readily stems from the definition of Besov spaces.
As for the convergence of $(a^n,\bar{u}^n,\bar{B}^n)$, it relies upon Ascoli's theorem compactness properties of the consequence, which are obtained by considering the time derivative of the solution, we omit the details here.
%%%%%%%%%%%%%%%%%%%%%%%%%%%%%%%%%%%%%%%%%%%%%%%%%%%%%%%%%%%%%%%%%%%%%%%%%%%%%%%%%%%%%%%%%%%%%%%%%%%%%%%%%%%%%%%%%%%%%%%%%%%%%

\noindent{\bf Step 4.} Uniqueness.

The goal of this section is to prove the uniqueness of our main theorem. We will follow the method used in \cite{abidi2013} to complete our proof.
Before giving the details, we also need the following lemma which can be proved similarly as Lemma 4.1 in  \cite{abidi2013}.
\begin{Lemma} (see \cite{abidi2013}) \label{QA1}
Let $p\in[3,4], q\in[1,2]$, $\frac1q-\frac1p\le\frac13,$  and $$ (a^i,u^i,B^i,\nabla \Pi^i
)\in C_b(0,T;B_{q,1}^{\frac3q})\times( C_b(0,T;\dot{B}_{p,1}^{-1+\frac3p})\cap L^1_T(\dot{B}_{p,1}^{1+\frac3p}))^2\times  L^1_T(\dot{B}_{p,1}^{-1+\frac3p}),$$
for $i=1,2$, be two solutions of the system \eqref{mhdmoxing} and satisfy \eqref{A122} for some $m$. Denote
$$ (\delta a,\delta u,\delta B,\nabla\delta \Pi)\triangleq(a^2-a^1,u^2-u^1,B^2-B^1,\nabla \Pi^2-\nabla \Pi^1).$$
Then there holds
\begin{align}\label{QA2}
(\delta a, \delta u,\delta B,\nabla\delta \Pi) \in & C_b([0, T ]; B_{2,1}^{\frac32}\times C_b([0, T ]; B_{2,1}^{-1/2}\cap L^1([0,T];B_{2,1}^{\frac32})\times C_b([0, T ]; B_{2,1}^{-1/2}) \nonumber\\
&\cap L^1([0,T];B_{2,1}^{\frac32})\times L^1([0,T];B_{2,1}^{-1/2}).
\end{align}
\end{Lemma}

Now, let us begin to prove our uniqueness, it's easy to get $(\delta a, \delta u,\delta B,\nabla\delta \Pi)$ solves
\begin{eqnarray}\label{Model2}
\left\{\begin{aligned}
&\partial_t \delta a+u^2\cdot\nabla \delta a =-\delta u\cdot\nabla a^1,\\
&\partial_t \delta u+u^2\cdot\nabla \delta u-(1+S_ma^2)(\Delta\delta u-\nabla\delta\Pi)=F_1(a^i,u^i,B^i,\nabla\Pi^i),\\
&\partial_t \delta B+u^2\cdot\nabla \delta B-\mathrm{div}\big(\widetilde{\sigma}(a^2)\nabla \delta B)=F_2(a^i,u^i,B^i),\\
&\diverg u =\diverg B =0,\\
&(\delta a,\delta u,\delta B)|_{t=0}=(0,0,0),
\end{aligned}\right.
\end{eqnarray}
where
\begin{align*}%\label{QA3}
F_1(a^i,u^i,B^i,\nabla\Pi^i)=&(a^2-S_ma^2)(\Delta\delta u-\nabla\delta\Pi)
-\delta u\cdot\nabla  u^1+\delta a(\Delta u^1-\nabla\Pi^1)\nonumber\\
&+(1+a^2)(\delta B\cdot\nabla B^2+ B^1\cdot\nabla \delta B)+\delta a( B^1\cdot\nabla B^1),
\end{align*}
\begin{align*}%\label{QA4+1}
F_2(a^i,u^i,B^i)=-\delta u\cdot\nabla  B^1+\delta B\cdot\nabla  u^2+B^1\cdot\nabla \delta u+\mathrm{div}\big((\widetilde{\sigma}(a^2)-\widetilde{\sigma}(a^1))\nabla  B^1.
\end{align*}
Using a similar method as  in Proposition \ref{monimingti}, we can get from the third equation of \eqref{Model2} that
\begin{align}\label{QA4}
&\|\delta B\|_{\widetilde{L}_t^{\infty}({B}_{2,1}^{-\frac12})}+\|\delta B\|_{{L}_t^1({B}_{2,1}^{\frac32})}\nonumber\\
\lesssim &\|[\Delta_j, u^2\cdot\nabla]\delta B\|_{L^1({B}_{2,1}^{-\frac12}))}+\|F_2\|_{{L_t^1({B_{2,1}^{-\frac12}})}}\nonumber\\
&+\sum_{j\in\mathbb{Z}}2^{\frac{j}{2}}\|[{\Delta}_j,\widetilde{\sigma}({S}_m a^2)]\nabla \delta B\|_{L_t^1(L^2)}+\|\mathrm{div}((\widetilde{\sigma}(a^2)-\widetilde{\sigma}({S}_m a^2))\nabla \delta B)\|_{L_t^1({B}_{2,1}^{-\frac{1}{2}})}.
\end{align}
By Lemmas \ref{bernstein}, \ref{daishu}, \ref{jiaohuanzi},  one has
\begin{align}\label{QA5}
\|[\Delta_j, u^2\cdot\nabla]\delta B\|_{L^1({B}_{2,1}^{-\frac12}))}\lesssim\int_0^t \| \nabla u^2\|_{{\dot{B}_{p,1}^{\frac3p}}}\|\delta B\|_{{{B}_{2,1}^{-\frac12}}}
d\tau,
\end{align}
\begin{align}\label{QA6}
\|\mathrm{div}((\widetilde{\sigma}(a^2)-\widetilde{\sigma}({S}_m a^2))\nabla \delta B)\|_{L_t^1({B}_{2,1}^{-\frac{1}{2}})}\lesssim&\|(\widetilde{\sigma}(a^2)-\widetilde{\sigma}({S}_m a^2))\nabla \delta B\|_{L_t^1({B}_{2,1}^{\frac{1}{2}})}\nonumber\\
\le&C_{\widetilde{\sigma}}\|a^2-{S}_m a^2\|_{\widetilde{L}_t^{\infty}({B}_{2,1}^{\frac32})}\| \delta B\|_{L_t^1({B}_{2,1}^{\frac{1}{2}})},
\end{align}
\begin{align}\label{QA7}
\sum_{j\in\mathbb{Z}}2^{\frac{j}{2}}\|[{\Delta}_j,\widetilde{\sigma}({S}_m a^2)]\nabla \delta B\|_{L_t^1(L^2)}\lesssim&\int_0^t\|\nabla\widetilde{\sigma}({S}_m a^2)\|_{{B}_{2,1}^{\frac32}}\| \delta B\|_{{B}_{2,1}^{1}}d\tau\nonumber\\
\lesssim&\varepsilon\| \delta B\|_{L_t^1({B}_{2,1}^{\frac32})}+2^{2m}\int_0^t \| a^2\|_{{B}_{2,1}^{\frac32}}^2\| \delta B\|_{{B}_{2,1}^{\frac{1}{2}}}d\tau,
\end{align}
\begin{align}\label{QA8}
&\|\delta u\cdot\nabla  B^1+\delta B\cdot\nabla  u^2+B^1\cdot\nabla \delta u\|_{L_t^1({B}_{2,1}^{-\frac{1}{2}})}\nonumber\\
\lesssim&\int_0^t (\|\delta u\|_{{{B}_{2,1}^{\frac12}}}+\|\delta B\|_{{{B}_{2,1}^{\frac12}}})
(\| B^1\|_{{\dot{B}_{p,1}^{\frac3p}}}+\| u^2\|_{{\dot{B}_{p,1}^{\frac3p}}})
d\tau\nonumber\\
\lesssim&\int_0^t (\|\delta u\|^{\frac12}_{{{B}_{2,1}^{-\frac12}}}\|\delta u\|^{\frac12}_{{{B}_{2,1}^{\frac32}}}+\|\delta B\|^{\frac12}_{{{B}_{2,1}^{-\frac12}}}\|\delta B\|^{\frac12}_{{{B}_{2,1}^{\frac32}}})
(\| B^1\|_{{\dot{B}_{p,1}^{\frac3p}}}+\| u^2\|_{{\dot{B}_{p,1}^{\frac3p}}})
d\tau\nonumber\\
\le& \varepsilon(\| \delta u\|_{{L_t^1({B_{2,1}^{\frac32}})}}+\| \delta B\|_{{L_t^1({B_{2,1}^{\frac32}})}})\nonumber\\
&+C_\varepsilon\int_0^t (\|\delta u\|_{{{B}_{2,1}^{-\frac12}}}+\|\delta B\|_{{{B}_{2,1}^{-\frac12}}}) (\| B^1\|^2_{{\dot{B}_{p,1}^{\frac3p}}}+\| u^2\|^2_{{\dot{B}_{p,1}^{\frac3p}}})  d\tau,
\end{align}
\begin{align}\label{QA9}
\|\mathrm{div}(\widetilde{\sigma}(a^2)-\widetilde{\sigma}(a^1))\nabla  B^1\|_{L_t^1({B}_{2,1}^{-\frac{1}{2}})}\lesssim&\|(\widetilde{\sigma}(a^2)-\widetilde{\sigma}(a^1))\nabla  B^1\|_{L_t^1({B}_{2,1}^{\frac{1}{2}})} \nonumber\\
&\le C_{\widetilde{\sigma}}\|(a^1,a^2)\|_{\widetilde{L}_t^{\infty}({B}_{2,1}^{\frac32})}
\int_0^t \| \delta a\|_{{\dot{B}_{2,1}^{\frac12}}}\| B^1\|_{{\dot{B}_{p,1}^{1+\frac3p}}}
d\tau.
\end{align}
Thus, inserting  the above estimates \eqref{QA5}--\eqref{QA9} into \eqref{QA4} and taking $\varepsilon$ small enough, we have
\begin{align}\label{QA11}
&\|\delta B\|_{\widetilde{L}_t^{\infty}({B}_{2,1}^{-\frac12})}+\|\delta B\|_{{L}_t^1({B}_{2,1}^{\frac32})}\nonumber\\
\lesssim &\int_0^t \| \nabla u^2\|_{{\dot{B}_{p,1}^{\frac3p}}}\|\delta B\|_{{{B}_{2,1}^{-\frac12}}}
d\tau+C_{\widetilde{\sigma}}\|a^2-{S}_m a^2\|_{\widetilde{L}_t^{\infty}({B}_{2,1}^{\frac32})}\| \delta B\|_{L_t^1({B}_{2,1}^{\frac{1}{2}})}\nonumber\\
&
+2^{2m}\int_0^t \| a^2\|_{{B}_{2,1}^{\frac32}}^2\| \delta B\|_{{B}_{2,1}^{\frac{1}{2}}}d\tau
+\varepsilon(\| \delta u\|_{{L_t^1({B_{2,1}^{\frac32}})}}+\| \delta B\|_{{L_t^1({B_{2,1}^{\frac32}})}})\nonumber\\
&+C_\varepsilon\int_0^t (\|\delta u\|_{{{B}_{2,1}^{-\frac12}}}+\|\delta B\|_{{{B}_{2,1}^{-\frac12}}}) (\| B^1\|^2_{{\dot{B}_{p,1}^{\frac3p}}}+\| u^2\|^2_{{\dot{B}_{p,1}^{\frac3p}}})  d\tau\nonumber\\
&+C_{\widetilde{\sigma}}\|(a^1,a^2)\|_{\widetilde{L}_t^{\infty}({B}_{2,1}^{\frac32})}
\int_0^t \| \delta a\|_{{\dot{B}_{2,1}^{\frac12}}}\| B^1\|_{{\dot{B}_{p,1}^{1+\frac3p}}}
d\tau.
\end{align}
Applying Proposition \ref{weiyixingmingti} to the second equation of \eqref{Model2}, we can get
\begin{align}\label{QA12}
&\|\delta u\|_{\widetilde{L}_t^{\infty}({B}_{2,1}^{-\frac12})}+\|\delta u\|_{{L}_t^1({B}_{2,1}^{\frac32})}\nonumber\\
\lesssim&\Bigg\{
\int_0^t\|\delta u\|_{{B_{2,1}^{-\frac12}}}\|\nabla u^2\|_{{B_{p,1}^{-1+\frac3p}}}d\tau+\|S_ma^2\|_{L_t^\infty(H^{\frac32+\alpha})}\|\nabla \delta\Pi\|_{L^1_t(H^{-\frac12-\alpha})}\nonumber\\
&+\|S_ma^2\|_{L_t^\infty(H^{2})}\|\nabla \delta u\|_{L^1_t(L^2)}+\|\delta u\|_{L^1_t(L^2)}+\|F_1(a^i,u^i,B^i,\nabla\Pi^i)\|_{{L_t^1({B_{2,1}^{-\frac12}})}}
\Bigg\},
\end{align}
where
\begin{align}\label{QA13}
\|F_1\|_{{L_t^1({B_{2,1}^{-\frac12}})}}\lesssim& \|a^2-S_ma^2\|_{L^\infty({B_{2,1}^{\frac32}})}(\|\delta u\|_{{L_t^1({B_{2,1}^{-\frac12}})}}
+\|\nabla\delta \Pi\|_{{L_t^1({B_{2,1}^{-\frac12}})}})
\nonumber\\
&+\|\delta u\cdot\nabla  u^1\|_{{L_t^1({B_{2,1}^{-\frac12}})}}
+\|\delta a(\Delta u^1-\nabla\Pi^1)\|_{{L_t^1({B_{2,1}^{-\frac12}})}}
+\|(1+a^2)\delta B\cdot\nabla B^2\|_{{L_t^1({B_{2,1}^{-\frac12}})}}\nonumber\\
&+\|(1+a^2)(B^1\cdot\nabla \delta B)\|_{{L_t^1({B_{2,1}^{-\frac12}})}}+\|\delta a( B^1\cdot\nabla B^1)\|_{{L_t^1({B_{2,1}^{-\frac12}})}}.
\end{align}

Taking $divergence$ to the second equation of \eqref{Model2}, we have
\begin{align}\label{QA14}
\diverg((1+S_ma^2)\nabla\delta\Pi)=\diverg F_3,
\end{align}
with
$$F_3=-u^2\cdot\nabla \delta u+S_ma^2\Delta\delta u+F_1.$$
By Proposition 3.4 in \cite{abidi2013}, we have
\begin{align}\label{QA15}
\|\nabla\delta\Pi\|_{{L_t^1({B_{2,1}^{-\frac12}})}}\lesssim&\|F_3\|_{{L_t^1({B_{2,1}^{-\frac12}})}}+\|S_ma^2\|_{L_t^\infty(H^{\frac32+\alpha})}\|\nabla \delta\Pi\|_{L^1_t(H^{-\frac12-\alpha})}\nonumber\\
\lesssim&\|u^2\cdot\nabla \delta u\|_{{L_t^1({B_{2,1}^{-\frac12}})}}+  \|S_ma^2\|_{{L}_t^{\infty}({B}_{2,1}^{\frac32})}\|\Delta\delta u\|_{{L_t^1({B_{2,1}^{-\frac12}})}}\nonumber\\
&+\|F_1\|_{{L_t^1({B_{2,1}^{-\frac12}})}}+\|S_ma^2\|_{L_t^\infty(H^{\frac32+\alpha})}\|\nabla \delta\Pi\|_{L^1_t(H^{-\frac12-\alpha})}.
\end{align}
Takin the above estimate into \eqref{QA13},  we get from \eqref{QA12} that
\begin{align}\label{QA16}
&\|\delta u\|_{\widetilde{L}_t^{\infty}({B}_{2,1}^{-\frac12})}+\|\delta u\|_{{L}_t^1({B}_{2,1}^{\frac32})}\nonumber\\
\lesssim&
\int_0^t\|\delta u\|_{{B_{2,1}^{-\frac12}}}\|\nabla u^2\|_{{B_{p,1}^{-1+\frac3p}}}d\tau+\|S_ma^2\|_{L_t^\infty(H^{\frac32+\alpha})}\|\nabla \delta\Pi\|_{L^1_t(H^{-\frac12-\alpha})}\nonumber\\
&+\|S_ma^2\|_{L_t^\infty(H^{2})}\|\nabla \delta u\|_{L^1_t(L^2)}+\|\delta u\|_{L^1_t(L^2)}+\|u^2\cdot\nabla \delta u\|_{{L_t^1({B_{2,1}^{-\frac12}})}}
\nonumber\\
&+\|\delta u\cdot\nabla  u^1\|_{{L_t^1({B_{2,1}^{-\frac12}})}}
+\|\delta a(\Delta u^1-\nabla\Pi^1)\|_{{L_t^1({B_{2,1}^{-\frac12}})}}+\|\delta a( B^1\cdot\nabla B^1)\|_{{L_t^1({B_{2,1}^{-\frac12}})}}
\nonumber\\
&+\|(1+a^2)\delta B\cdot\nabla B^2\|_{{L_t^1({B_{2,1}^{-\frac12}})}}
+\|(1+a^2)(B^1\cdot\nabla \delta B)\|_{{L_t^1({B_{2,1}^{-\frac12}})}}.
\end{align}
Firstly, it follows from Proposition \ref{shuyun}  that
\begin{align}\label{buyaoshan}
&\|\delta a\|_{\widetilde{L}_t^{\infty}({B}_{2,1}^{\frac12})}\lesssim\exp\{C \|\nabla u^2\|_{{L_t^1({B_{p,1}^{\frac3p}})}}\}
\| a^1\|_{{L}_t^{\infty}({B}_{2,1}^{\frac32})}\|\delta u\|_{{L_t^1({B_{2,1}^{\frac32}})}}.
\end{align}
In what follows, we will estimate the terms  on the right hand side of \eqref{QA16}.
By Lemma \ref{daishu}, Young's inequality and \eqref{buyaoshan}, we have
\begin{align*}
&\|\delta a( B^1\cdot\nabla B^1)\|_{{L_t^1({B_{2,1}^{-\frac12}})}}+\|\delta a(\Delta u^1-\nabla\Pi^1)\|_{{L_t^1({B_{2,1}^{-\frac12}})}}\nonumber\\
\lesssim&\|\delta a(\Delta u^1-\nabla\Pi^1)\|_{{L_t^1({\dot{B}_{2,1}^{-\frac12}})}}+\|\delta a( B^1\cdot\nabla B^1)\|_{{L_t^1({\dot{B}_{2,1}^{-\frac12}})}}\nonumber\\
\lesssim&\int_0^t\|\delta a\|_{{\dot{B}_{2,1}^{\frac12}}}(\|\Delta u^1\|_{{\dot{B}_{p,1}^{-1+\frac3p}}}+\|\nabla\Pi^1\|_{{\dot{B}_{p,1}^{-1+\frac3p}}} +\|B^1\cdot\nabla B^1\|_{{\dot{B}_{p,1}^{-1+\frac3p}}} )d\tau\nonumber\\
\lesssim&\int_0^t\exp\{C \|\nabla u^2\|_{{L_\tau^1({B_{p,1}^{\frac3p}})}}\}
\| a^1\|_{{L}_\tau^{\infty}({B}_{2,1}^{\frac32})}\|\delta u\|_{{L_\tau^1({B_{2,1}^{\frac32}})}}
\end{align*}
\begin{align}\label{QA18}
&\quad\times(\|\Delta u^1\|_{{\dot{B}_{p,1}^{-1+\frac3p}}}+\|\nabla\Pi^1\|_{{\dot{B}_{p,1}^{-1+\frac3p}}} +\|B^1\|_{{\dot{B}_{p,1}^{-1+\frac3p}}}\|B^1\|_{{\dot{B}_{p,1}^{1+\frac3p}}})d\tau\nonumber\\
\lesssim&\int_0^t\|\delta u\|_{{L_\tau^1({B_{2,1}^{\frac32}})}}
(\|\Delta u^1\|_{{\dot{B}_{p,1}^{-1+\frac3p}}}+\|\nabla\Pi^1\|_{{\dot{B}_{p,1}^{-1+\frac3p}}} +\|B^1\|_{{\dot{B}_{p,1}^{-1+\frac3p}}}\|B^1\|_{{\dot{B}_{p,1}^{1+\frac3p}}})d\tau,
\end{align}
and
\begin{align}\label{QA19}
&\|u^2\cdot\nabla \delta u\|_{{L_t^1({B_{2,1}^{-\frac12}})}}+\|\delta u\cdot\nabla  u^1\|_{{L_t^1({B_{2,1}^{-\frac12}})}}+\|(1+a^2)\delta B\cdot\nabla B^2\|_{{L_t^1({B_{2,1}^{-\frac12}})}}\nonumber\\
&\quad+\|(1+a^2)(B^1\cdot\nabla \delta B)\|_{{L_t^1({B_{2,1}^{-\frac12}})}}\nonumber\\
\lesssim&\int_0^t \Big((\| u^1\|_{{\dot{B}_{p,1}^{\frac3p}}}+\| u^2\|_{{\dot{B}_{p,1}^{\frac3p}}})\|\delta u\|_{{{B}_{2,1}^{\frac12}}}
+(1+\| a^2\|_{{L}_t^{\infty}({B}_{2,1}^{\frac32})})(\| B^1\|_{{\dot{B}_{p,1}^{\frac3p}}}+\| B^2\|_{{\dot{B}_{p,1}^{\frac3p}}})\|\delta u\|_{{{B}_{2,1}^{\frac12}}}
\Big)d\tau\nonumber\\
\le& \varepsilon(\| \delta u\|_{{L_t^1({B_{2,1}^{\frac32}})}}+\| \delta B\|_{{L_t^1({B_{2,1}^{\frac32}})}})
+C_\varepsilon\int_0^t \|\delta u\|_{{{B}_{2,1}^{-\frac12}}} (\| u^1\|^2_{{\dot{B}_{p,1}^{\frac3p}}}+\| u^2\|^2_{{\dot{B}_{p,1}^{\frac3p}}})  d\tau\nonumber\\
&+C_\varepsilon\int_0^t \|\delta B\|_{{{B}_{2,1}^{-\frac12}}}(1+\| a^2\|_{{L}_t^{\infty}({B}_{2,1}^{\frac32})})^2(\| B^1\|_{{\dot{B}_{p,1}^{\frac3p}}}+\| B^2\|_{{\dot{B}_{p,1}^{\frac3p}}})d\tau.
\end{align}
Plugging the estimates \eqref{QA18}, \eqref{QA19} into \eqref{QA16}, we can get
\begin{align}\label{QA20}
&\|\delta u\|_{\widetilde{L}_t^{\infty}({B}_{2,1}^{-\frac12})}+\|\delta u\|_{{L}_t^1({B}_{2,1}^{\frac32})}\nonumber\\
\lesssim&\|S_ma^2\|_{L_t^\infty(H^{\frac32+\alpha})}\|\nabla \delta\Pi\|_{L^1_t(H^{-\frac12-\alpha})}
+(1+\|S_ma^2\|_{L_t^\infty(H^{2})})\| \delta u\|_{L^1_t(H^1)}\nonumber\\
+&\int_0^t\|\delta u\|_{{B_{2,1}^{-\frac12}}}(\| u^2\|_{{B_{p,1}^{\frac3p}}}+\| u^1\|^2_{{\dot{B}_{p,1}^{\frac3p}}}+\| u^2\|^2_{{\dot{B}_{p,1}^{\frac3p}}})d\tau\nonumber\\
&+\int_0^t \|\delta B\|_{{{B}_{2,1}^{-\frac12}}}(1+\| a^2\|_{{L}_t^{\infty}({B}_{2,1}^{\frac32})})^2(\| B^1\|_{{\dot{B}_{p,1}^{\frac3p}}}+\| B^2\|_{{\dot{B}_{p,1}^{\frac3p}}})d\tau\nonumber\\
&+\int_0^t\|\delta u\|_{{L_\tau^1({B_{2,1}^{\frac32}})}}
(\|\Delta u^1\|_{{\dot{B}_{p,1}^{-1+\frac3p}}}+\|\nabla\Pi^1\|_{{\dot{B}_{p,1}^{-1+\frac3p}}} +\|B^1\|_{{\dot{B}_{p,1}^{-1+\frac3p}}}\|B^1\|_{{\dot{B}_{p,1}^{1+\frac3p}}})d\tau.
\end{align}
In the following, we have to estimate $\|\nabla\delta\Pi\|_{H^{-\frac12-\alpha}}$.
We get from \eqref{Model2} that
\begin{align*}%\label{QA21}
\diverg((1+a^2)\nabla\delta\Pi)=\diverg F_4,
\end{align*}
with
\begin{align*}%\label{QA22}
F_4=&a^2\Delta\delta u-u^2\cdot\nabla \delta u
-\delta u\cdot\nabla  u^1+\delta a(\Delta u^1-\nabla\Pi^1)\nonumber\\
&+(1+a^2)(\delta B\cdot\nabla B^2+ B^1\cdot\nabla \delta B)+\delta a( B^1\cdot\nabla B^1).
\end{align*}
Hence, by Proposition 3.5 in \cite{abidi2013}, we have
\begin{align}\label{QA23}
\|\nabla\delta\Pi\|_{H^{-\frac12-\alpha}}\lesssim(1+2^{m(\frac12+\frac3q)}\|a^2\|_{{B_{2,1}^{\frac32}}})\|F_4\|_{H^{-\frac12-\alpha}}.
\end{align}
By Lemma  \ref{daishu}, and Young's inequality, we have
\begin{align}\label{QA24}
\|F_4\|_{H^{-\frac12-\alpha}}\lesssim&\|a^2\|_{{B_{2,1}^{\frac32}}}\|\Delta\delta u\|_{H^{-\frac12-\alpha}}+\|\delta u\|_{\dot{H}^{\frac32-\alpha}}(\| u^1\|_{{\dot{B}_{p,1}^{-1+\frac3p}}}+\| u^2\|_{{\dot{B}_{p,1}^{-1+\frac3p}}})\nonumber\\
&+\|\delta B\|_{\dot{H}^{\frac32-\alpha}} (1+\| a^2\|_{{L}_t^{\infty}({B}_{2,1}^{\frac32})})(\| B^1\|_{{\dot{B}_{p,1}^{\frac3p}}}+\| B^2\|_{{\dot{B}_{p,1}^{\frac3p}}})\nonumber\\
&+\|\delta a\|_{{B_{2,1}^{\frac12-\alpha}}}(\| u^1\|_{{\dot{B}_{p,1}^{\frac3p}}}
+\| \nabla\Pi^1\|_{{\dot{B}_{p,1}^{-1+\frac3p}}}+\| B^1\|_{{\dot{B}_{p,1}^{-1+\frac3p}}}\| B^1\|_{{\dot{B}_{p,1}^{1+\frac3p}}})\nonumber\\
\lesssim&\|\delta u\|_{{H}^{\frac32-\alpha}}(\|a^2\|_{{B_{2,1}^{\frac32}}}+\| u^1\|_{{\dot{B}_{p,1}^{-1+\frac3p}}}+\| u^2\|_{{\dot{B}_{p,1}^{-1+\frac3p}}})\nonumber\\
&+\|\delta B\|_{{H}^{\frac32-\alpha}} (1+\| a^2\|_{{L}_t^{\infty}({B}_{2,1}^{\frac32})})(\| B^1\|_{{\dot{B}_{p,1}^{\frac3p}}}+\| B^2\|_{{\dot{B}_{p,1}^{\frac3p}}})\nonumber\\
&+\|\delta a\|_{{B_{2,1}^{\frac12}}}(\| u^1\|_{{\dot{B}_{p,1}^{\frac3p}}}
+\| \nabla\Pi^1\|_{{\dot{B}_{p,1}^{-1+\frac3p}}}+\| B^1\|_{{\dot{B}_{p,1}^{-1+\frac3p}}}\| B^1\|_{{\dot{B}_{p,1}^{1+\frac3p}}}).
\end{align}
Thus, taking  the above estimate into \eqref{QA23}, we have
\begin{align}\label{QA25}
\|\nabla\delta\Pi\|_{L_t^1(H^{-\frac12-\alpha})}\lesssim&(1+2^{m(\frac12+\frac3q)}\|a^2\|_{L^\infty({B_{2,1}^{\frac32}})})\Big\{  \int_0^t \|\delta u\|_{{H}^{\frac32-\alpha}}(\|a^2\|_{{B_{2,1}^{\frac32}}}+\| u^1\|_{{\dot{B}_{p,1}^{-1+\frac3p}}}+\| u^2\|_{{\dot{B}_{p,1}^{-1+\frac3p}}})d\tau\nonumber\\
&+\int_0^t \|\delta B\|_{{H}^{\frac32-\alpha}} (1+\| a^2\|_{{L}_t^{\infty}({B}_{2,1}^{\frac32})})(\| B^1\|_{{\dot{B}_{p,1}^{\frac3p}}}+\| B^2\|_{{\dot{B}_{p,1}^{\frac3p}}})d\tau\nonumber\\
&+\int_0^t \|\delta a\|_{{B_{2,1}^{\frac12}}}(\| u^1\|_{{\dot{B}_{p,1}^{\frac3p}}}
+\| \nabla\Pi^1\|_{{\dot{B}_{p,1}^{-1+\frac3p}}}+\| B^1\|_{{\dot{B}_{p,1}^{-1+\frac3p}}}\| B^1\|_{{\dot{B}_{p,1}^{1+\frac3p}}})d\tau \Big\}.
\end{align}
By interpolation inequality, one has
\begin{align}\label{QA26}
&\|\delta u\|_{L_t^1({H}^{\frac32-\alpha})}\le \varepsilon\|\delta u\|_{{L}_t^1({B}_{2,1}^{\frac32})}+C_\varepsilon\|\delta u\|_{{L}_t^1({B}_{2,1}^{-\frac12})},\nonumber\\
&\|\delta B\|_{L_t^1({H}^{\frac32-\alpha})}\le \varepsilon\|\delta B\|_{{L}_t^1({B}_{2,1}^{\frac32})}+C_\varepsilon\|\delta B\|_{{L}_t^1({B}_{2,1}^{-\frac12})}.
\end{align}
Taking the above estimates into \eqref{QA25} and choosing $\varepsilon$ small enough, we can get from \eqref{QA20} that
\begin{align}\label{QA28}
&\|\delta u\|_{\widetilde{L}_t^{\infty}({B}_{2,1}^{-\frac12})}+\|\delta u\|_{{L}_t^1({B}_{2,1}^{\frac32})}\nonumber\\
\lesssim&\|S_ma^2\|_{L_t^\infty(H^{\frac32+\alpha})}\|\nabla \delta\Pi\|_{L^1_t(H^{-\frac12-\alpha})}
+(1+\|S_ma^2\|_{L_t^\infty(H^{2})})\| \delta u\|_{L^1_t(H^1)}\nonumber\\
+&\int_0^t\|\delta u\|_{{B_{2,1}^{-\frac12}}}(\| u^2\|_{{B_{p,1}^{\frac3p}}}+\| u^1\|^2_{{\dot{B}_{p,1}^{\frac3p}}}+\| u^2\|^2_{{\dot{B}_{p,1}^{\frac3p}}})d\tau\nonumber\\
&+\int_0^t \|\delta B\|_{{{B}_{2,1}^{-\frac12}}}(1+\| a^2\|_{{L}_t^{\infty}({B}_{2,1}^{\frac32})})^2(\| B^1\|^2_{{\dot{B}_{p,1}^{\frac3p}}}+\| B^2\|^2_{{\dot{B}_{p,1}^{\frac3p}}})d\tau\nonumber\\
&+\int_0^t\|\delta u\|_{{L_\tau^1({B_{2,1}^{\frac32}})}}
(\|\Delta u^1\|_{{\dot{B}_{p,1}^{-1+\frac3p}}}+\|\nabla\Pi^1\|_{{\dot{B}_{p,1}^{-1+\frac3p}}} +\|B^1\|_{{\dot{B}_{p,1}^{-1+\frac3p}}}\|B^1\|_{{\dot{B}_{p,1}^{1+\frac3p}}})d\tau.
\end{align}
Thanks to estimates \eqref{QA11}, \eqref{QA28}, one can finally get by choosing $\varepsilon$ small enough that
\begin{align}\label{QA29}
&\|\delta u\|_{\widetilde{L}_t^{\infty}({B}_{2,1}^{-\frac12})}+\|\delta u\|_{{L}_t^1({B}_{2,1}^{\frac32})}+\|\delta B\|_{\widetilde{L}_t^{\infty}({B}_{2,1}^{-\frac12})}+\|\delta B\|_{{L}_t^1({B}_{2,1}^{\frac32})}\nonumber\\
\lesssim&\int_0^t(\|\delta u\|_{{B_{2,1}^{-\frac12}}}+\|\delta B\|_{{B_{2,1}^{-\frac12}}}+\|\delta u\|_{{L_\tau^1({B_{2,1}^{\frac32}})}})w(\tau)d\tau,
\end{align}
where
\begin{align}\label{QA30}
w(\tau)=&\|  u^2\|_{{\dot{B}_{p,1}^{1+\frac3p}}}+2^{2m} \| a^2\|_{{B}_{2,1}^{\frac32}}^2+(1+\| a^2\|_{{L}_t^{\infty}({B}_{2,1}^{\frac32})})^2
(\| B^1\|^2_{{\dot{B}_{p,1}^{\frac3p}}}+\| B^2\|^2_{{\dot{B}_{p,1}^{\frac3p}}}+| u^2\|^2_{{\dot{B}_{p,1}^{\frac3p}}})\nonumber\\
&+C_{\widetilde{\sigma}}\|(a^1,a^2)\|_{\widetilde{L}_t^{\infty}({B}_{2,1}^{\frac32})}
\exp\{C \|\nabla u^2\|_{{L_\tau^1({B_{p,1}^{\frac3p}})}}\}
\| a^1\|_{{L}_\tau^{\infty}({B}_{2,1}^{\frac32})}\| B^1\|_{{\dot{B}_{p,1}^{1+\frac3p}}}\nonumber\\
&+(\| u^2\|_{{B_{p,1}^{\frac3p}}}+\| u^1\|^2_{{\dot{B}_{p,1}^{\frac3p}}}+\| u^2\|^2_{{\dot{B}_{p,1}^{\frac3p}}})\nonumber\\
&+(\|\Delta u^1\|_{{\dot{B}_{p,1}^{-1+\frac3p}}}+\|\nabla\Pi^1\|_{{\dot{B}_{p,1}^{-1+\frac3p}}} +\|B^1\|_{{\dot{B}_{p,1}^{-1+\frac3p}}}\|B^1\|_{{\dot{B}_{p,1}^{1+\frac3p}}}).
\end{align}
Applying Gronwall's inequality and using (\ref{buyaoshan}) implies
$\delta a=\delta u = \delta B= 0$
for all $t \in [0, T ].$ This concludes the proof to the uniqueness part of Theorem \ref{zhuyaodingli}.
\subsection{Higher regularity part of the solution}
Let $(a^n, u^n,B^n,\Pi^n)$  be the approximate solutions of (\ref{mhdmoxing})
constructed in Step 2 of Subsection 5.1. Then for $0 < \tau < t_0 < t\le T^\ast$, with $T^\ast$
being determined by (\ref{ABA35}), we deduce by a similar proof of (\ref{moniguji}), (\ref{suduguji}) that
\begin{align}\label{}
&\|(u^n,B^n)\|_{\widetilde{L}^\infty([\tau,t];\dot{B}_{p,1}^{\frac{3}{p}})}+\|(u^n,B^n)\|_{L^1([\tau,t];\dot{B}_{p,1}^{2+\frac{3}{p}})}
+\|\nabla \Pi^n\|_{L^1([\tau,t];\dot{B}_{p,1}^{\frac{3}{p}})} \nonumber\\
\lesssim&  \|a_0\|_{\dot{B}_{q,1}^{\frac{3}{q}}}(\|(u^n(\tau),B^n(\tau))\|_{\dot{B}_{p,1}^{{\frac{3}{p}}}}+
\|(u^n(\tau),B^n(\tau))\|_{\dot{B}_{p,1}^{-1+\frac{3}{p}}})
\nonumber\\
&\times\exp\{ \|(u^n,B^n)\|_{L^1([\tau,t];\dot{B}_{p,1}^{1+\frac{3}{p}})}\}.
\end{align}
Integrating the above inequality for $\tau$ over $[0, t_0]$, and then dividing the resulting
inequality by $t_0$ lead to
\begin{align}\label{bande}
&\|(u^n,B^n)\|_{\widetilde{L}^\infty([t_0,t];\dot{B}_{p,1}^{\frac{3}{p}})}+\|(u^n,B^n)\|_{L^1([t_0,t];\dot{B}_{p,1}^{2+\frac{3}{p}})}
+\|\nabla \Pi^n\|_{L^1([t_0,t];\dot{B}_{p,1}^{\frac{3}{p}})} \nonumber\\
\lesssim&  \|a_0\|_{\dot{B}_{q,1}^{\frac{3}{q}}}
\|(u_0,B_0)\|_{\dot{B}_{p,1}^{-1+\frac{3}{p}}}(1+1/{\sqrt{t_0}})
\exp\{C\|(u_0,B_0)\|_{\dot{B}_{p,1}^{-1+\frac{3}{p}}}\}.
\end{align}

\begin{center}
\section{Global Well-posedness of Theorem \ref{zhuyaodingli}}
\end{center}

In this section, we will give the proof of the global well-posedness of Theorem \ref{zhuyaodingli} under the assumption that
$\|u_0\|_{\dot{B}_{p,1}^{-1+\frac{3}{p}}}+\|B_0\|_{\dot{B}_{p,1}^{-1+\frac{3}{p}}} $
is sufficiently small.
By a similar proof of Theorem 4.1 in \cite{abidi2013}, we can get, for $a_0 \in{B}_{q,1}^{\frac{3}{q}}$,  if $u_0 \in \dot{B}_{p,1}^{-1+\frac{3}{p}}, B_0 \in \dot{B}_{p,1}^{-1+\frac{3}{p}}$,
sufficiently small, (\ref{mhdmoxing}) has a unique local
solution $(a, u, B)$ satisfying
 \begin{align*}
&a\in C_b([0, T^\ast ]; B_{q,1}^{\frac{3}{q}}(\R^3), \hspace{0.5cm} u\in C_b([0, T^\ast ]; \dot{B}_{p,1}^{-1+\frac{3}{p}}(\R^3)\cap L^1([0,T^\ast];\dot{B}_{p,1}^{1+\frac{3}{p}}(\R^3)),\nonumber\\
&B\in C_b([0, T^\ast ]; \dot{B}_{p,1}^{-1+\frac{3}{p}}(\R^3)\cap L^1([0,T^\ast];\dot{B}_{p,1}^{1+\frac{3}{p}}(\R^3)),
\end{align*}
for some $T^\ast>1$. In what follows, we will prove $T^\ast=\infty$.
Let $u=v+w, {B}={h}+b$ where $(v,{h})$ satisfies the following equations
\begin{equation}\label{mhd1}
\left\{\begin{aligned}
&v_t+v\cdot\nabla v-\Delta v-{h}\cdot\nabla {h}+\nabla P_v=0,\\
&{h}_t-\Delta {h}+v\cdot\nabla {h}-{h}\cdot\nabla v=0,\\
&\mathrm{div}v=0,\mathrm{div} {h}=0,\\
&v|_{t=t_1}=u(t_1),{h}|_{t=t_1}={B}(t_1).
\end{aligned}\right.
\end{equation}
Then $(\rho,w,b)$ solves the equations
\begin{equation}\label{mhd2}
\left\{\begin{aligned}
&\rho w_t+\rho(v+w)\cdot\nabla w-\Delta w+\nabla P_w={h}\cdot\nabla b+b\cdot\nabla {h}+b\cdot\nabla b\\
&\hspace{6.2cm}  +(1-\rho)(v_t+v\cdot\nabla v)-\rho w\cdot\nabla v,\\
&b_t-\Delta b+w\cdot\nabla {h}+w\cdot\nabla b+v\cdot\nabla b-{h}\cdot\nabla w-b\cdot\nabla w-b\cdot\nabla v=0,\\
&\mathrm{div}w=0,\mathrm{div} b=0,\\
&\rho|_{t=t_1}=\rho(t_1),w|_{t=t_1}=0,b|_{t=t_1}=0.
\end{aligned}\right.
\end{equation}

Note  that  $\|u(t_1)\|_{\dot{B}_{p,1}^{-1+\frac{3}{p}}\cap\dot{B}_{p,1}^{2+\frac{3}{p}}}+ \|{B}(t_1)\|_{\dot{B}_{p,1}^{-1+\frac{3}{p}}\cap\dot{B}_{p,1}^{2+\frac{3}{p}}} $ is very small, provided that
$\|u_0\|_{\dot{B}_{p,1}^{-1+\frac{3}{p}}}+\|B_0\|_{\dot{B}_{p,1}^{-1+\frac{3}{p}}}$
is sufficiently small. It follows from the classical theory of MHD
equations (see \cite{miaochangxing}) that (\ref{mhd1}) has a unique global solution $$(v,{h})\in C([t_1,+\infty];  \dot{B}_{p,1}^{-1+\frac{3}{p}} )\cap L^1([t_1,+\infty]; \dot{B}_{p,1}^{1+\frac{3}{p}} )\times ([t_1,+\infty];  \dot{B}_{p,1}^{-1+\frac{3}{p}} )\cap L^1([t_1,+\infty]; \dot{B}_{p,1}^{1+\frac{3}{p}} )$$ satisfying
\begin{align}\label{Q1}
&\|v\|_{\widetilde{L}^\infty([t_1,+\infty];\dot{B}_{p,1}^{-1+\frac{3}{p}})}+\|{h}\|_{\widetilde{L}^\infty([t_1,+\infty];\dot{B}_{p,1}^{-1+\frac{3}{p}})}
+\|v\|_{L^1([t_1,+\infty];\dot{B}_{p,1}^{1+\frac{3}{p}})}\nonumber\\
&~~~+\|{h}\|_{L^1([t_1,+\infty];\dot{B}_{p,1}^{1+\frac{3}{p}})}+
\|\nabla P_v\|_{L^1([t_1,+\infty];\dot{B}_{p,1}^{-1+\frac{3}{p}})}\nonumber\\
\lesssim&  (\|u(t_1)\|_{\dot{B}_{p,1}^{-1+\frac{3}{p}}}+\|{B}(t_1)\|_{\dot{B}_{p,1}^{-1+\frac{3}{p}}})
\end{align}
 and
\begin{align}\label{Q2}
&\|v_t\|_{L^1([t_1,+\infty];\dot{B}_{p,1}^{-1+\frac{3}{p}})}+\|{h}_t\|_{L^1([t_1,+\infty];\dot{B}_{p,1}^{-1+\frac{3}{p}})}\nonumber\\
\lesssim& (\|u(t_1)\|_{\dot{B}_{p,1}^{-1+\frac{3}{p}}}+\|{B}(t_1)\|_{\dot{B}_{p,1}^{-1+\frac{3}{p}}})
+\|\mathrm{div} (v\otimes v)\|_{L^1([t_1,+\infty];\dot{B}_{p,1}^{-1+\frac{3}{p}})}+\|\mathrm{div} ({h}\otimes {h})\|_{L^1([t_1,+\infty];\dot{B}_{p,1}^{-1+\frac{3}{p}})}\nonumber\\
&+\|\mathrm{div} (v\otimes {h})\|_{L^1([t_1,+\infty];\dot{B}_{p,1}^{-1+\frac{3}{p}})}
+\|\mathrm{div} ({h}\otimes v)\|_{L^1([t_1,+\infty];\dot{B}_{p,1}^{-1+\frac{3}{p}})}\nonumber\\
\lesssim& (\|u(t_1)\|_{\dot{B}_{p,1}^{-1+\frac{3}{p}}}+\|{B}(t_1)\|_{\dot{B}_{p,1}^{-1+\frac{3}{p}}}).
\end{align}
With $(v,{h})$ thus obtained, we denote $w=u-v,b={B}-{h}$. Then thanks to (\ref{mhd1}) and (\ref{mhd2}).
The proof of Theorem \ref{zhuyaodingli} reduces to proving the global well-posedness
of (\ref{mhd2}). For simplicity, in what follows, we just present the a priori
estimates for smooth enough solutions of (\ref{mhd2}) on $[0, T^\ast)$.
\vspace{0.5cm}
\subsection{The higher regularities of $(v,{h})$.}

\begin{proposition}\label{Q3}
Let $(v,{h},\nabla P_v)$ be the unique global solution of (\ref{mhd1}) which satisfies (\ref{Q1}) and
(\ref{Q2}). Then for $s_1\in [\frac{3}{p},2+\frac{3}{p}] $ and $s_2\in [-1+\frac{3}{p},\frac{3}{p}]$, there hold
\begin{align}\label{Q4}
&\|v\|_{\widetilde{L}^\infty([t_1,+\infty];\dot{B}_{p,1}^{s_1})}+\|{h}\|_{\widetilde{L}^\infty([t_1,+\infty];\dot{B}_{p,1}^{s_1})}
+\|v\|_{L^1([t_1,+\infty];\dot{B}_{p,1}^{{s_1+2}})}\nonumber\\
&\quad+\|{h}\|_{L^1([t_1,+\infty];\dot{B}_{p,1}^{s_1+2})}+
\|\nabla P_v\|_{L^1([t_1,+\infty];\dot{B}_{p,1}^{s_1})}\nonumber\\
\le& C (\|u_0\|_{\dot{B}_{p,1}^{-1+\frac{3}{p}}}+\|B_0\|_{\dot{B}_{p,1}^{-1+\frac{3}{p}}})
\end{align}
and
\begin{align}\label{Q5}
&\|v_t\|_{\widetilde{L}^\infty([t_1,+\infty];\dot{B}_{p,1}^{s_2})}+\|{h}_t\|_{\widetilde{L}^\infty([t_1,+\infty];\dot{B}_{p,1}^{s_2})}+
\| v_t\|_{L^1([t_1,+\infty];\dot{B}_{p,1}^{s_2+2})}\nonumber\\
&\quad+\| {h}_t\|_{L^1([t_1,+\infty];\dot{B}_{p,1}^{s_2+2})}
+\|\partial_t\nabla P_v\|_{L^1([t_1,+\infty];\dot{B}_{p,1}^{s_2})}
\nonumber\\
\le& C (\|u_0\|_{\dot{B}_{p,1}^{-1+\frac{3}{p}}}+\|B_0\|_{\dot{B}_{p,1}^{-1+\frac{3}{p}}}).
\end{align}
\end{proposition}
The proof of this proposition is rather standard, we omit the details here.
By Proposition \ref{Q3}, we can easily get the following corollary:
\begin{corollary}\label{Q6}
Under the assumptions of Proposition \ref{Q3}, one has
\begin{align}\label{Q7}
&\|\nabla v\|_{L^2([t_1,+\infty];L^\infty)}+\|\nabla {h}\|_{L^2([t_1,+\infty];L^\infty)}+\|v_t+v\cdot\nabla v\|_{L^2([t_1,+\infty];L^\infty)}\nonumber\\
\le& C (\|u_0\|_{\dot{B}_{p,1}^{-1+\frac{3}{p}}}+\|B_0\|_{\dot{B}_{p,1}^{-1+\frac{3}{p}}}).
\end{align}
\end{corollary}
\subsection{The $L^2$ estimate of  $(w,b)$.}
\begin{proposition}\label{Q8}
If the conditions of Theorem \ref{zhuyaodingli} are satisfied, there holds for $t_1 < t < T^\ast$
\begin{align}\label{Q8+1}
&\|w\|_{L^\infty([t_1,t];L^2)}+\|b\|_{L^\infty([t_1,t];L^2)}+ \|\nabla w\|_{L^2([t_1,t];L^2)}+ \|\nabla b\|_{L^2([t_1,t];L^2)}\nonumber\\
\le& C (\|u_0\|_{\dot{B}_{p,1}^{-1+\frac{3}{p}}}+\|B_0\|_{\dot{B}_{p,1}^{-1+\frac{3}{p}}}),
\end{align}
with $C$ being independent of  $t$.
\end{proposition}
\begin{proof} \quad Firstly, thanks to $1+\inf_{x\in\R^3}a_0(x)\ge \kappa>0$, we can get from the transport equation of (\ref{mhdmoxing})
that
\begin{align}\label{Q9}
(1+\|a_0\|_{{B}_{q,1}^{\frac3q}})^{-1}\le \rho(t,x)\le \underline{d}^{-1},
\end{align}
from which and $1-\rho = \rho a,$ we get by taking the $L^2$ inner product of the $w$ equation of (\ref{mhd2}) with $w$ and of the $b$ equation of (\ref{mhd2}) with $b$ that
\begin{align}\label{Q10}
&\frac12\frac{d}{dt}(\|\sqrt{\rho}w\|_{L^2}^2+\|b\|_{L^2})^2+\|\nabla w\|_{L^2}^2+\|\nabla b\|_{L^2}^2\nonumber\\
=&\int_{\R^3}((1-\rho)(v_t+v\cdot\nabla v)-\rho w\cdot\nabla v)\cdot wdx
+\int_{\R^3}{h}\cdot\nabla b\cdot w dx+\int_{\R^3}b\cdot\nabla {h}\cdot w dx\nonumber\\
&+\int_{\R^3}b\cdot\nabla b\cdot w dx
-\int_{\R^3}w\cdot\nabla {h}\cdot b dx-\int_{\R^3}w\cdot\nabla b\cdot b dx-\int_{\R^3}v\cdot\nabla b\cdot b dx
\nonumber\\
&+\int_{\R^3}{h}\cdot\nabla w\cdot b dx+\int_{\R^3}b\cdot\nabla w\cdot b dx+\int_{\R^3}b\cdot\nabla v\cdot b dx\nonumber\\
\triangleq&\sum_{j=1}^{10}I_j.
\end{align}
Integrating by parts, we can get $I_2+I_8=I_4+I_9=I_6=I_7=0$.
Using  the {H}$\mathrm{\ddot{o}}$lder inequality, we have
\begin{align}\label{Q11}
|I_1|=\bigg|\int_{\R^3}((1-\rho)(v_t+v\cdot\nabla v)-\rho w\cdot\nabla v)\cdot wdx\bigg|\le C \|\sqrt{\rho}w\|_{L^2}\|a\|_{L^2}\|v_t+v\cdot\nabla v\|_{L^\infty},
\end{align}
\begin{align}\label{Q12}
|I_3+I_5+I_{10}|&=\bigg|\int_{\R^3}b\cdot\nabla {h}\cdot w dx+\int_{\R^3}w\cdot\nabla {h}\cdot b dx+\int_{\R^3}b\cdot\nabla v\cdot b dx\bigg|\nonumber\\
&\le C (\|b\|_{L^2}\|\sqrt{\rho}w\|_{L^2}\|\nabla {h}\|_{L^\infty}+ \|b\|_{L^2}^2\|\nabla v\|_{L^\infty}).
\end{align}
Substituting the above estimates into (\ref{Q10}), we have
\begin{align}\label{Q13}
&\frac12\frac{d}{dt}(\|\sqrt{\rho}w\|_{L^2}^2+\|b\|_{L^2})^2+\|\nabla w\|_{L^2}^2+\|\nabla b\|_{L^2}^2\nonumber\\
\le& C(\|\sqrt{\rho}w\|_{L^2}\|a\|_{L^2}\|v_t+v\cdot\nabla v\|_{L^\infty})+C(\|\sqrt{\rho}w\|_{L^2}+\|b\|_{L^2})(\|\nabla v\|_{L^\infty}+\|\nabla {h}\|_{L^\infty}),
\end{align}
from which, we infer for $t \in (t_1, T^\ast)$ that
\begin{align}\label{Q14}
&\frac{d}{dt}(e^{-2\int_{t_1}^t(\|\nabla v\|_{L^\infty}+\|\nabla {h}\|_{L^\infty})d\tau}(\|\sqrt{\rho}w\|_{L^2}^2+\|b\|_{L^2}^2))\nonumber\\
\le& C \|a_0\|_{L^2}e^{-2\int_{t_1}^t(\|\nabla v\|_{L^\infty}+\|\nabla {h}\|_{L^\infty})d\tau}\|\sqrt{\rho}w\|_{L^2}  \|v_t+v\cdot\nabla v\| _{\dot{B}_{p,1}^{\frac3p} }.
\end{align}
This, along with  (\ref{Q4}), implies
\begin{align}\label{Q15}
&\|\sqrt{\rho}w\|_{L^\infty([t_1,t];L^2)}^2+\|b\|_{L^\infty([t_1,t];L^2)}^2\nonumber\\
&\le Ce^{\int_{t_1}^t(\|\nabla v\|_{L^\infty}+\|\nabla {h}\|_{L^\infty})d\tau} \|v_t+v\cdot\nabla v\| _{L^1([t_1,t];{\dot{B}_{p,1}^{\frac3p} })}\nonumber\\
&\le C(\|u(t_1)\|_{\dot{B}_{p,1}^{\frac3p}}+\|{B}(t_1)\|_{\dot{B}_{p,1}^{\frac3p}})\exp\{C(\|u(t_1)\|_{\dot{B}_{p,1}^{-1+\frac{3}{p}}}
+\|{B}(t_1)\|_{\dot{B}_{p,1}^{-1+\frac{3}{p}}}) \}\nonumber\\
&\le C(\|u_0\|_{\dot{B}_{p,1}^{-1+\frac{3}{p}}}+\|B_0\|_{\dot{B}_{p,1}^{-1+\frac{3}{p}}}).
\end{align}
Taking the above estimate into (\ref{Q13}) gives rise to
\begin{align}\label{Q16}
 \|\nabla w\|_{L^2([t_1,t];L^2)}+ \|\nabla b\|_{L^2([t_1,t];L^2)}\le C(\|u_0\|_{\dot{B}_{p,1}^{-1+\frac{3}{p}}}+\|B_0\|_{\dot{B}_{p,1}^{-1+\frac{3}{p}}}).
\end{align}
\end{proof}
\subsection{The ${H}^1$ estimate of  $(w,b)$.}\label{Q17}
\begin{proposition}\label{Q18}
Under the assumptions of Theorem \ref{zhuyaodingli}, there exist two positive constants
$e_1, e_2 $ so that  for $t_1 < t < T^\ast$
\begin{align}\label{Q19}
&\|\nabla w\|^2_{L^\infty([t_1,t];L^2)}+\|\nabla b\|^2_{L^\infty([t_1,t];L^2)}
\nonumber\\
&~~~+\int_{t_1}^t(e_1(\|\partial_tw\|_{L^2}^2+\|\partial_tb\|_{L^2}^2)+e_2(\|\nabla^2w\|_{L^2}^2+\|\nabla^2b\|_{L^2}^2)
+\|\nabla P_w\|_{L^2}^2)dt'\nonumber\\
\le& C(\|u_0\|_{\dot{B}_{p,1}^{-1+\frac{3}{p}}}^2+\|B_0\|_{\dot{B}_{p,1}^{-1+\frac{3}{p}}}^2)
\end{align}
with $C$ being independent of  $t$.
\end{proposition}
\begin{proof} We first get, by taking the $L^2$ inner product  of $(\ref{mhd2})_1$, $(\ref{mhd2})_2$ with
$\frac{1}{\rho}\Delta w,\Delta b$ respectively  and using  the {H}$\ddot{\mathrm{o}}$lder inequality, Young inequality, (\ref{Q9}) that,
\begin{align}\label{Q20}
&\frac12\frac{d}{dt}(\|\nabla w\|_{L^2}^2+\|\nabla b\|_{L^2}^2)+\|\frac{1}{\sqrt{\rho}}\Delta w\|_{L^2}^2 +\|\Delta b\|_{L^2}^2\nonumber\\
\le&C\|\frac{1}{\sqrt{\rho}}\Delta w\|_{L^2}\Big( \|\nabla P_w\|_{L^2}+ \|v\|_{L^\infty}\|\Delta w\|_{L^2}+ \|w\|_{L^3}\|\nabla w\|_{L^6}+\|a\|_{L^2}\|v_t+v\cdot\nabla v\|_{L^\infty}\nonumber\\
&+\| w\|_{L^2} \|\Delta v\|_{L^\infty} +\|{h}\|_{L^\infty}\|\Delta b\|_{L^2}+\|b\|_{L^2}\|\Delta {h}\|_{L^\infty}+\|b\|_{L^3}\|\Delta b\|_{L^6}\Big)\nonumber\\
&+C\|\Delta b\|_{L^2}\Big( \|b\|_{L^\infty}\|\Delta b\|_{L^2}+\|w\|_{L^2}\|\Delta {h}\|_{L^\infty} +\|w\|_{L^3}\|\Delta b\|_{L^6}\nonumber\\
&+ \|{h}\|_{L^\infty}\|\Delta w\|_{L^2}+\|\Delta v\|_{L^\infty}\| b\|_{L^2}
+\|b\|_{L^3}\|\Delta w\|_{L^6}\Big)\nonumber\\
\le&\frac{1}{16}\|\frac{1}{\sqrt{\rho}}\Delta w\|_{L^2}^2+ \frac{1}{16}\|\Delta b\|_{L^2}^2
+C\|\nabla P_w\|_{L^2}^2+C\|v_t+v\cdot\nabla v\|_{L^\infty}^2\nonumber\\
&+C(\|v\|_{L^\infty}^2+\|{h}\|_{L^\infty}^2)(\|\nabla w\|_{L^2}^2+\|\nabla b\|_{L^2}^2)+
C(\|\nabla v\|_{L^\infty}^2+\|\nabla {h}\|_{L^\infty}^2)(\| w\|_{L^2}^2+ \| b\|_{L^2}^2)\nonumber\\
&+C(\|w\|_{L^3}^2+\|b\|_{L^3}^2)(\|\nabla^2 w\|_{L^2}^2+\|\nabla^2 b\|_{L^2}^2).
\end{align}
Now, we give the estimate of pressure function $\nabla P_w$. Thanks to $\mathrm{div} w=0$, we obtain from the momentum equation in (\ref{mhd2}) that
\begin{align}\label{Q21}
\|\Delta w\|_{L^2}^2+\|\nabla P_w\|_{L^2}^2\le& 2\|\Delta w-\nabla P_w\|_{L^2}^2\nonumber\\
\le&   C \|\sqrt{\rho}w_t\|_{L^2}^2+C(\|v\|_{L^\infty}^2+\|{h}\|_{L^\infty}^2)(\|\nabla w\|_{L^2}^2+\|\nabla b\|_{L^2}^2)\nonumber\\
&+C(\|\nabla v\|_{L^\infty}^2+\|\nabla {h}\|_{L^\infty}^2)(\| w\|_{L^2}^2+ \| b\|_{L^2}^2)\nonumber\\
&+C(\|w\|_{L^3}^2+\|b\|_{L^3}^2)(\|\nabla^2 w\|_{L^2}^2+\|\nabla^2 b\|_{L^2}^2),
\end{align}
which along with (\ref{Q20}) leads to
\begin{align}\label{Q22}
&\frac{d}{dt}(\|\nabla w\|_{L^2}^2+\|\nabla b\|_{L^2}^2)+c_1(\|\nabla^2 w\|_{L^2}^2 +\|\nabla^2 b\|_{L^2}^2)\nonumber\\
\le&C \|\sqrt{\rho}w_t\|_{L^2}^2+C\|v_t+v\cdot\nabla v\cdot\nabla {h}\|_{L^\infty}^2+C(\|v\|_{L^\infty}^2+\|{h}\|_{L^\infty}^2)(\|\nabla w\|_{L^2}^2+\|\nabla b\|_{L^2}^2)\nonumber\\
&+
C(\|\nabla v\|_{L^\infty}^2+\|\nabla {h}\|_{L^\infty}^2)(\| w\|_{L^2}^2+ \| b\|_{L^2}^2)+C(\|w\|_{L^3}^2+\|b\|_{L^3}^2)(\|\nabla^2 w\|_{L^2}^2+\|\nabla^2 b\|_{L^2}^2).
\end{align}
Along the same line, we get by taking the $L^2$ inner-product of the equation of
$(\ref{mhd2})_1$, $(\ref{mhd2})_2$ with $w_t,b_t$ respectively that
\begin{align}\label{Q23}
&\frac{d}{dt}(\|\nabla w\|_{L^2}^2+\|\nabla b\|_{L^2}^2)+\|\sqrt{\rho}w_t\|_{L^2}^2 +\|b_t\|_{L^2}^2\nonumber\\
\le&C\|v_t+v\cdot\nabla v\|_{L^\infty}^2+C(\|v\|_{L^\infty}^2+\|{h}\|_{L^\infty}^2)(\|\nabla w\|_{L^2}^2+\|\nabla b\|_{L^2}^2)\nonumber\\
&+
C(\|\nabla v\|_{L^\infty}^2+\|\nabla {h}\|_{L^\infty}^2)(\| w\|_{L^2}^2+ \| b\|_{L^2}^2)+C(\|w\|_{L^3}^2+\|b\|_{L^3}^2)(\|\nabla^2 w\|_{L^2}^2+\|\nabla^2 b\|_{L^2}^2).
\end{align}
Combining (\ref{Q22}) with (\ref{Q23}), we deduce that there is a positive constant $c_2$ such
that
\begin{align}\label{Q24}
&\frac{d}{dt}(\|\nabla w\|_{L^2}^2+\|\nabla b\|_{L^2}^2)+c_2(\| w_t\|_{L^2}^2 +\|b_t\|_{L^2}^2)\nonumber\\
&\quad+(\frac{c_1}{2C_{11}}-C_{11}(\|w\|_{L^3}^2+\|b\|_{L^3}^2) )(\|\nabla^2 w\|_{L^2}^2 +\|\nabla^2 b\|_{L^2}^2) \nonumber\\
\le&C_{12}\Big(\|v_t+v\cdot\nabla v\| _{L^\infty}^2+(\|v\|_{L^\infty}^2+\|{h}\|_{L^\infty}^2)(\|\nabla w\|_{L^2}^2+\|\nabla b\|_{L^2}^2)\nonumber\\
&+
(\|\nabla v\|_{L^\infty}^2+\|\nabla {h}\|_{L^\infty}^2)(\| w\|_{L^2}^2+ \| b\|_{L^2}^2)\Big).
\end{align}
Denote
\begin{align}\label{Q25}
\tau^\ast\triangleq\sup\left\{ t\ge t_1,\quad\|w\|_{L^3}^2+\|b\|_{L^3}^2\le  \frac{c_1}{2C_{11}^2}\right\}.
\end{align}
We claim that $\tau^\ast=T^\ast$ provided that $ \|u_0\|_{\dot{B}_{p,1}^{-1+\frac{3}{p}}}+\|B_0\|_{\dot{B}_{p,1}^{-1+\frac{3}{p}}}$
is sufficiently small.

 Otherwise for
$t \in [t_1,\tau^\ast)$, we can get from (\ref{Q24}) that
\begin{align}\label{Q26}
&\frac{d}{dt}(\|\nabla w\|_{L^2}^2+\|\nabla b\|_{L^2}^2)+c_2(\|w_t\|_{L^2}^2 +\|b_t\|_{L^2}^2)+\frac{c_1}{4C_{11}}(\|\nabla^2 w\|_{L^2}^2 +\|\nabla^2 b\|_{L^2}^2) \nonumber\\
\le&C_{12}\Big(\|v_t+v\cdot\nabla v\| _{L^\infty}^2+(\|v\|_{L^\infty}^2+\|{h}\|_{L^\infty}^2)(\|\nabla w\|_{L^2}^2+\|\nabla b\|_{L^2}^2)\nonumber\\
&+
(\|\nabla v\|_{L^\infty}^2+\|\nabla {h}\|_{L^\infty}^2)(\| w\|_{L^2}^2+ \| b\|_{L^2}^2)\Big).
\end{align}
Applying Gronwall's inequality to (\ref{Q26}) and using (\ref{Q7}) give rise to
\begin{align}\label{Q27}
\|\nabla w\|_{L^2}^2+\|\nabla b\|_{L^2}^2\le& C_{12}\exp\{C_{12}\int_{t_1}^t(\|v\|_{L^\infty}^2+\|{h}\|_{L^\infty}^2)dt'\}\nonumber\\
&\times\int_{t_1}^t(\|v_t+v\cdot\nabla v\| _{L^\infty}^2+\|\nabla v\|_{L^\infty}^2+\|\nabla {h}\|_{L^\infty}^2)dt'\nonumber\\
\le& C_{13}(\|u_0\|_{\dot{B}_{p,1}^{-1+\frac{3}{p}}}^2+\|B_0\|_{\dot{B}_{p,1}^{-1+\frac{3}{p}}}^2).
\end{align}
However,  (\ref{Q8}) and (\ref{Q27}) tell us that
\begin{align}\label{Q28}
\| w\|_{L^3}^2+\|b\|_{L^3}^2&\le C(\|w\|_{L^2}\|\nabla w\|_{L^2}+\| b\|_{L^2}\|\nabla b\|_{L^2})\nonumber\\
&\le C_{14}(\|u_0\|_{\dot{B}_{p,1}^{-1+\frac{3}{p}}}^2+\|B_0\|_{\dot{B}_{p,1}^{-1+\frac{3}{p}}}^2)\le\frac{c_1}{4C_{11}^2}
\end{align}
for $t\in [t_1,\tau^\ast),$ provided that $\|u_0\|_{\dot{B}_{p,1}^{-1+\frac{3}{p}}}^2+\|B_0\|_{\dot{B}_{p,1}^{-1+\frac{3}{p}}}^2\le \frac{c_1}{4C_{11}^2C_{14}}$, which contradicts (\ref{Q25}). This, in turn, shows that
$\tau^\ast=T^\ast$. Then integrating (\ref{Q26}) and using (\ref{Q7}) lead to (\ref{Q19}). This completes the
proof of the proposition.
\end{proof}
%%%%%%%%%%%%%%%%%%%%%%%%%%%%%%%%%%%%%%%%%%%%%%%%%%%%%%%%%%%%%%%%%%%%%%%%%%%%%%%%%%%%%%%%%%%%%%%%%%%%%%%%%
%%%%%%%%%%%%%%%%%%%%%%%%%%%%%%%%%%%%%%%%%%%%%%%%%%%%%%%%%%%
\subsection{The $H^2$ estimate of  $(w,b)$.}\label{Q29}
\begin{proposition}\label{Q30}
Under the assumptions of Theorem \ref{zhuyaodingli}, there exists a time independent
constant C such that for $t_1 < t < T^\ast$
\begin{align}\label{Q31}
&\|\nabla^2 w\|_{L^\infty([t_1,t];L^2)}+\|\nabla^2 b\|_{L^\infty([t_1,t];L^2)}+\|\nabla w_t\|_{L^2([t_1,t];L^2)}+\|\nabla b_t\|_{L^2([t_1,t];L^2)}\nonumber\\
&\quad+\|\nabla ^2w\|_{L^2([t_1,t];L^6)}
+\|\nabla ^2b\|_{L^2([t_1,t];L^6)}\le C.
\end{align}
\end{proposition}
\begin{proof}
We get by first applying $\partial_t$ to the  first equation and second equation of (\ref{mhd2}) respectively and then taking the $L^2$
inner product of the resulting equation with $(w_t,b_t)$, that
\begin{align*}
&\frac12\frac{d}{dt}(\|\sqrt{\rho}w_t\|_{L^2}^2+\|b_t\|_{L^2}^2)+\|\nabla w_t\|_{L^2}^2+\|\nabla b_t\|_{L^2}^2\nonumber\\
=&\int_{\R^3}(1-\rho)w_t\cdot\partial_t(\Delta v-\nabla P_v-{h}\cdot\nabla {h})dx\nonumber\\
&- \int_{\R^3}\rho_tw_t\cdot(w_t+(v+w)\cdot w+w\cdot\nabla v+(\Delta v-\nabla P_v-{h}\cdot\nabla {h}))dx
\end{align*}
\begin{align}\label{Q32}
&- \int_{\R^3}\rho w_t\cdot((v+w)_t\cdot\nabla w+w_t\cdot\nabla v+w\cdot\nabla v_t)dx \nonumber\\                 &+\int_{\R^3}{h}_t\cdot\nabla b\cdot w_tdx+\int_{\R^3}b_t\cdot\nabla {h}\cdot w_tdx
+\int_{\R^3}b\cdot\nabla {h}_t\cdot w_tdx\nonumber\\
&+\int_{\R^3}b_t\cdot\nabla b\cdot w_tdx-\int_{\R^3}w_t\cdot\nabla {h}\cdot b_tdx-\int_{\R^3}w\cdot\nabla {h}_t\cdot b_tdx\nonumber\\
&+\int_{\R^3}{h}_t\cdot\nabla w\cdot b_tdx+\int_{\R^3}b_t\cdot\nabla v\cdot b_tdx+\int_{\R^3}b\cdot\nabla v_t\cdot b_tdx+\int_{\R^3}b_t\cdot\nabla w\cdot b_tdx\nonumber\\
\triangleq& \sum_{i=1}^{13}J_i.
\end{align}
$J_1, J_2, J_3$ can be estimated the same as in \cite{abidi2012},  that is,
\begin{align}\label{Q33}
|J_1|\le C \|a_0\|_{L^\infty} \|\sqrt{\rho}w_t\|_{L^2}\|\partial_t(\Delta v-\nabla P_v-{h}\cdot\nabla {h})\|_{L^2},
\end{align}
\begin{align}\label{Q34}
|J_2|\le&\frac14 \|\nabla w_t\|_{L^2}^2 +C \Big(\|v\|_{L^\infty}^4+ \|\nabla v\|_{L^6}^2+\|\Delta w\|_{L^2}^2
+\|v_t+v\cdot\nabla v\| _{L^2}^2\nonumber\\
&+\|\nabla^2 v\|_{L^6}^2
+\|\sqrt{\rho}w_t\|_{L^2}^2( \|v\|_{L^\infty}^2+\|\nabla w\|_{L^2}\|\nabla w^2 \|_{L^2})\nonumber\\
&+\|\sqrt{\rho}w_t\|_{L^2}+\|\nabla(\Delta v-\nabla P_v-{h}\cdot\nabla {h})\|_{L^4}\Big),
\end{align}
\begin{align}\label{Q35}
|J_3|\le&\frac14 \|\nabla w_t\|_{L^2}^2 +C \Big(
\|\sqrt{\rho}w_t\|_{L^2}^2( \|\nabla v\|_{L^\infty}+\|\nabla w\|_{L^2}\|\nabla w^2 \|_{L^2})\nonumber\\
&+\|\sqrt{\rho}w_t\|_{L^2}( \|v_t\|_{L^\infty}+ \|\nabla v_t\|_{L^4} )\Big).
\end{align}
Using the {H}$\ddot{\mathrm{o}}$lder inequality and Young's inequality implies
\begin{align}\label{Q36}
|J_4+J_{10}|\le C\|{h}_t\|_{L^\infty}( \|\nabla b\|_{L^2}\|\sqrt{\rho} w_t\|_{L^2}+\|\nabla w\|_{L^2}\| b_t\|_{L^2} )\le C\|{h}_t\|_{L^\infty}(\|\sqrt{\rho} w_t\|_{L^2}+ \| b_t\|_{L^2}),
\end{align}
\begin{align}\label{Q37}
|J_6+J_9+J_{12}|\le C(\| b\|_{L^4}\| \nabla {h}_t\|_{L^4}\|\sqrt{\rho} w_t\|_{L^2}+\| w\|_{L^4}\| \nabla {h}_t\|_{L^4}\|b_t\|_{L^2}+\| b\|_{L^4}\| \nabla v_t\|_{L^4}\|b_t\|_{L^2}),
\end{align}
\begin{align}\label{Q38}
|J_5+J_8+J_{11}|\le C(\|b_t\|_{L^2}\|\sqrt{\rho} w_t\|_{L^2}\|\nabla {h}\|_{L^\infty}+\|b_t\|_{L^2}^2\|\nabla v\|_{L^\infty} ),
\end{align}
\begin{align}\label{Q39}
|J_7+J_{13}|\le& C(\|b_t\|_{L^6}\|\sqrt{\rho} w_t\|_{L^2} \|\nabla b\|_{L^3}+\|b_t\|_{L^6}\|b_t\|_{L^2} \|\nabla w\|_{L^3} )\nonumber\\
\le&\frac{1}{16} \|\nabla b_t\|_{L^2} ^2+C\|\sqrt{\rho} w_t\|_{L^2}^2  \|\nabla b\|_{L^2}\|\nabla^2 b\|_{L^2}
+C\|b_t\|_{L^2}^2  \|\nabla w\|_{L^2}\|\nabla^2 w\|_{L^2}.
\end{align}
Taking above estimates (\ref{Q33})-(\ref{Q39}) into (\ref{Q32}), we have
\begin{align}\label{Q40}
&\frac{d}{dt}(\|\sqrt{\rho}w_t\|_{L^2}^2+\|b_t\|_{L^2}^2)+\|\nabla w_t\|_{L^2}^2+\|\nabla b_t\|_{L^2}^2\nonumber\\
\le& C((f_1(t)+f_3(t))(\|\sqrt{\rho}w_t\|_{L^2}^2+\|b_t\|_{L^2}^2)+f_2(t)+f_3(t)),
\end{align}
with
\begin{align*}%\label{Q41}
f_1(t)\triangleq&\|v\|_{L^\infty}^2+\|\nabla v\|_{L^\infty}+\|\nabla {h}\|_{L^\infty}+  \|\nabla w\|_{L^2}\|\nabla^2 w\|_{L^2} +\|\nabla b\|_{L^2}\|\nabla^2 b\|_{L^2},\nonumber\\
f_2(t)\triangleq&\|v\|_{L^\infty}^4+ \|\nabla v\|_{L^6}^2+\|\Delta w\|_{L^2}^2
+\|v_t+v\cdot\nabla v\| _{L^2}^2+\|\nabla^2 v\|_{L^6}^2,\nonumber\\
f_3(t)\triangleq& \|\nabla(\Delta v-\nabla P_v-{h}\cdot\nabla {h})\|_{L^4}+\|\partial_t(\Delta v-\nabla P_v-{h}\cdot\nabla {h})\|_{L^2}+\|v_t\|_{L^\infty}+ \|\nabla v_t\|_{L^4}\nonumber\\
&+\|{h}_t\|_{L^\infty}+ \|\nabla {h}_t\|_{L^4}.
\end{align*}
Gronwall's inequality helps us to get from (\ref{Q40}) for $t \in (t_1, T^\ast)$ that
\begin{align}\label{Q42}
&\|\sqrt{\rho}w_t\|_{L^2}^2+\|b_t\|_{L^2}^2+\int_{t_1}^t (\|\nabla w_t\|_{L^2}^2+\|\nabla b_t\|_{L^2}^2) dt'\nonumber\\
\le& C\exp\{ \int_{t_1}^t (f_1(t')+f_3(t')) dt'\}( \|\sqrt{\rho}w_t(t_1)\|_{L^2}^2+ \|b_t(t_1)\|_{L^2}^2+ \int_{t_1}^t (f_2(t')+f_3(t')) dt').
\end{align}
However, by embedding relations $$  \dot{B}_{p,1}^{\frac3p}(\R^3)\hookrightarrow L^\infty(\R^3),  \dot{B}_{p,1}^{\frac{3}{p}-\frac12}(\R^3)\hookrightarrow L^6(\R^3),\dot{B}_{p,1}^{\frac{3}{p}-\frac34}(\R^3)\hookrightarrow L^4(\R^3),$$
we deduce from
(\ref{Q4}), (\ref{Q5}), (\ref{Q8+1}) and (\ref{Q19}) that
\begin{align*}%\label{Q43}
 \int_{t_1}^t(f_1(t')+f_2(t')+f_3(t'))  dt'\le C,
\end{align*}
with $C$ being independent of t.
Whereas taking the $L^2$ inner product of the first and second
equation of $(\ref{mhd2})$ with $(w_t,b_t)$ at $t = t_1$  respectively and using the higher regularity of $(a,u,B,\nabla\Pi)$ give rise to
\begin{align*}%\label{Q44}
\|\sqrt{\rho}w_t(t_1)\|_{L^2}+ \|b_t(t_1)\|_{L^2}&\le C\|a_t(t_1)\|_{L^2}\|(v_t+v\cdot\nabla v)(t_1)\|_{L^\infty}\nonumber\\
&\le C \| (\Delta v-\nabla P_v-{h}\cdot\nabla {h})(t_1)\|_{\dot{B}_{p,1}^{\frac3p}} \nonumber\\
&\le C\|v(t_1)\|_{\dot{B}_{p,1}^{2+\frac3p}}+\|v(t_1)\|_{\dot{B}_{p,1}^{\frac3p}}\|v(t_1)\|_{\dot{B}_{p,1}^{1+\frac{3}{p}}}
+\|{h}(t_1)\|_{\dot{B}_{p,1}^{\frac3p}}\|{h}(t_1)\|_{\dot{B}_{p,1}^{1+\frac{3}{p}}} \nonumber\\
&\le C.
\end{align*}
As a consequence, we deduce from (\ref{Q42}) that
\begin{align}\label{Q45}
\sup_{t\in[t_1,T^\ast]} (\|\sqrt{\rho}w_t\|_{L^2}^2+\|b_t\|_{L^2}^2)+\int_{t_1}^{T^\ast} (\|\nabla w_t(t')\|_{L^2}^2+\|\nabla b_t(t')\|_{L^2}^2) dt'.
\end{align}
In the following, we give the estimates of the second space derivative estimate of $(w,b)$.

We first observe from the equation of (\ref{mhd2}) that
\begin{align}\label{Q46}
&\|\nabla^2 w\|_{L^2}+\|\nabla P_w\|_{L^2}+\|\nabla^2 b\|_{L^2}\nonumber\\
\lesssim& \|\sqrt{\rho}w_t\|_{L^2}+\|\rho w\cdot\nabla w\|_{L^2}+\|\rho v\cdot\nabla w\|_{L^2}+\|\rho w\cdot\nabla v\|_{L^2}+\|{h}\cdot\nabla b\|_{L^2}\nonumber\\
&+\|b\cdot\nabla {h}\|_{L^2}+\|b\cdot\nabla b\|_{L^2}
+\|(1-\rho)(v_t+v\cdot\nabla v)\|_{L^2}+\|w\cdot\nabla {h}\|_{L^2}+\|w\cdot\nabla b\|_{L^2}\nonumber\\&
+\|v\cdot\nabla b\|_{L^2}+\|{h}\cdot\nabla w\|_{L^2}+\|b\cdot\nabla w\|_{L^2}+\|b\cdot\nabla v\|_{L^2}+\|b_t\|_{L^2}\nonumber\\
\lesssim&\|\sqrt{\rho}w_t\|_{L^2}+\|b_t\|_{L^2}+\|w\|_{L^6}\|\nabla w\|_{L^2}^{\frac12}\|\nabla^2 w\|_{L^2}^{\frac12}+\|v\|_{L^\infty}\|\nabla w\|_{L^2}\nonumber\\
&+\|\nabla v\|_{L^\infty}\|w\|_{L^2}+\|{h}\|_{L^\infty}\|\nabla b\|_{L^2}+\|\nabla {h}\|_{L^\infty}\|b\|_{L^2}+\|b\|_{L^6}\|\nabla b\|_{L^2}^{\frac12}\|\nabla^2 b\|_{L^2}^{\frac12}\nonumber\\
&+
\|v\|_{\dot{B}_{p,1}^{2+\frac3p}}+\|v\|_{\dot{B}_{p,1}^{\frac3p}}\|v\|_{\dot{B}_{p,1}^{1+\frac{3}{p}}}
+\|{h}\|_{\dot{B}_{p,1}^{\frac3p}}\|{h}\|_{\dot{B}_{p,1}^{1+\frac{3}{p}}}+\|\nabla {h}\|_{L^\infty}\|w\|_{L^2}\nonumber\\
&+\|w\|_{L^6}\|\nabla b\|_{L^2}^{\frac12}\|\nabla^2 b\|_{L^2}^{\frac12}+\|v\|_{L^\infty}\|\nabla b\|_{L^2}+\|{h}\|_{L^\infty}\|\nabla w\|_{L^2}\nonumber\\
&+\|b\|_{L^6}\|\nabla w\|_{L^2}^{\frac12}\|\nabla^2 w\|_{L^2}^{\frac12}+\|\nabla v\|_{L^\infty}\|b\|_{L^2}\nonumber\\
\lesssim&\|\sqrt{\rho}w_t\|_{L^2}+\|b_t\|_{L^2} +\|\nabla w\|_{L^2}^3+\|\nabla w\|_{L^2}^2\|\nabla b\|_{L^2}
+\|\nabla w\|_{L^2}\|\nabla b\|_{L^2} ^2  \nonumber\\
&+\frac18(\|\nabla^2 w\|_{L^2}+\|\nabla^2 b\|_{L^2})
+(\|v\|_{L^\infty}+\|{h}\|_{L^\infty})(\|\nabla w\|_{L^2}+\|\nabla b\|_{L^2})\nonumber\\
&+(\|\nabla v\|_{L^\infty}+\|\nabla {h}\|_{L^\infty})(\|w\|_{L^2}+\|b\|_{L^2})
+
\|v\|_{\dot{B}_{p,1}^{2+\frac3p}}+\|v\|_{\dot{B}_{p,1}^{\frac3p}}\|v\|_{\dot{B}_{p,1}^{1+\frac{3}{p}}}+\|{h}\|_{\dot{B}_{p,1}^{\frac3p}}\|{h}\|_{\dot{B}_{p,1}^{1+\frac{3}{p}}},
\end{align}
where we have used the following fact:
\begin{align}\label{Q47}
\|\partial_tv\|_{\dot{B}_{p,1}^{\frac3p}}=&\|\Delta v-\mathbf{P}(v\cdot\nabla v)+\mathbf{P}({h}\cdot\nabla {h})\|_{\dot{B}_{p,1}^{\frac3p}}\nonumber\\
\le&\|v\|_{\dot{B}_{p,1}^{2+\frac3p}}+\|v\|_{\dot{B}_{p,1}^{\frac3p}}\|v\|_{\dot{B}_{p,1}^{1+\frac{3}{p}}}
+\|{h}\|_{\dot{B}_{p,1}^{\frac3p}}\|{h}\|_{\dot{B}_{p,1}^{1+\frac{3}{p}}},
\end{align}
which along with (\ref{Q4}), (\ref{Q5}), (\ref{Q8+1}), (\ref{Q19}) and (\ref{Q45}) ensures that
\begin{align}\label{Q48}
\sup_{t\in[t_1,T^\ast]} (\|\nabla^2 w\|_{L^2}+\|\nabla P_w\|_{L^2}+\|\nabla^2 b\|_{L^2})\le C.
\end{align}
On the other hand, let $(v, q)$ solve
$$-\Delta v+\nabla q=f,\quad\quad \mathrm{div} v=0.$$
Then one has $\nabla q=-\nabla(-\Delta)^{-1} \mathrm{div} f,$ and for any $r\in (1,\infty),$
$$\|\nabla q\|_{L^r}\le C\|f\|_{L^r},\quad \|\Delta v\|_{L^r}\le C\|f\|_{L^r}.$$
From this and the equation of (\ref{mhd2}) we infer
\begin{align*}%\label{Q49}
&\|\nabla^2 w\|_{L^6}+\|\nabla P_w\|_{L^6}+\|\nabla^2 b\|_{L^6}\nonumber\\
\lesssim&\|\nabla w_t\|_{L^2}+\|\nabla b_t\|_{L^2}+\|\nabla w\|_{L^2}^2\|\nabla^2 w\|_{L^2}
+\|\nabla^2 w\|_{L^2}\|\nabla b\|_{L^2} ^2  +\|\nabla w\|_{L^2}^2\|\nabla^2 b\|_{L^2} \nonumber\\
&+\frac18(\|\nabla^2 w\|_{L^6}+\|\nabla^2 b\|_{L^6})
+(\|v\|_{L^\infty}+\|{h}\|_{L^\infty})(\|\nabla ^2w\|_{L^2}+\|\nabla b^2\|_{L^2})\nonumber\\
&+(\|\nabla v\|_{L^\infty}+\|\nabla {h}\|_{L^\infty})(\|\nabla w\|_{L^2}+\|\nabla b\|_{L^2})
+
\|v\|_{\dot{B}_{p,1}^{2+\frac3p}}+\|v\|_{\dot{B}_{p,1}^{\frac3p}}\|v\|_{\dot{B}_{p,1}^{1+\frac{3}{p}}}+\|{h}\|_{\dot{B}_{p,1}^{\frac3p}}\|{h}\|_{\dot{B}_{p,1}^{1+\frac{3}{p}}},
\end{align*}
which implies
\begin{align}\label{Q50}
&\|\nabla^2 w\|_{L^6}+\|\nabla P_w\|_{L^6}+\|\nabla^2 b\|_{L^6}\nonumber\\
\lesssim&\|\nabla w_t\|_{L^2}+\|\nabla b_t\|_{L^2}+\|\nabla w\|_{L^2}^2\|\nabla^2 w\|_{L^2}
+\|\nabla^2 w\|_{L^2}\|\nabla b\|_{L^2} ^2  +\|\nabla w\|_{L^2}^2\|\nabla^2 b\|_{L^2} \nonumber\\
&
+(\|v\|_{L^\infty}+\|{h}\|_{L^\infty})(\|\nabla ^2w\|_{L^2}+\|\nabla b^2\|_{L^2})+(\|\nabla v\|_{L^\infty}+\|\nabla {h}\|_{L^\infty})(\|\nabla w\|_{L^2}+\|\nabla b\|_{L^2})
\nonumber\\
&+
\|v\|_{\dot{B}_{p,1}^{2+\frac3p}}+\|v\|_{\dot{B}_{p,1}^{\frac3p}}\|v\|_{\dot{B}_{p,1}^{1+\frac{3}{p}}}+\|{h}\|_{\dot{B}_{p,1}^{\frac3p}}\|{h}\|_{\dot{B}_{p,1}^{1+\frac{3}{p}}}.
\end{align}
Taking the $L^2$ norm for the time variables on $[t_1,t]$, we get by using (\ref{Q8+1}), (\ref{Q45}) and (\ref{Q48})  that
\begin{align}\label{Q51}
&\|\nabla^2 w\|^2_{L^2([t_1,t];L^6)}+\|\nabla P_w\|^2_{L^2([t_1,t];L^6)}+\|\nabla^2 b\|^2_{L^2([t_1,t];L^6)}\nonumber\\
\le& C(\|\nabla w_t\|^2_{L^2([t_1,t];L^2)}+\|\nabla b_t\|^2_{L^2([t_1,t];L^2)})+
C(\|\nabla ^2w\|^2_{L^2([t_1,t];L^2)}+\|\nabla ^2b\|^2_{L^2([t_1,t];L^2)})\nonumber\\
&+
C\|v\|^2_{L^2([t_1,t];{\dot{B}_{p,1}^{2+\frac3p}})}+C\|v\|^2_{L^2([t_1,t];{\dot{B}_{p,1}^{1+\frac{3}{p}}})}
+C\|{h}\|^2_{L^2([t_1,t];{\dot{B}_{p,1}^{1+\frac{3}{p}}})}\nonumber\\
\le& C.
\end{align}
This completes the proof of the proposition.
\end{proof}
\subsection{Proof of Theorem \ref{zhuyaodingli}.}\label{Q52}
 We then rewrite the equations for $u$ and $B$ in \eqref{mhdmoxing} as
\begin{eqnarray*}
\left\{\begin{aligned}
&\partial_t u-\Delta u+\nabla\Pi=B\cdot\nabla B-u\cdot\nabla u+\frac{a}{1+a}(\partial_t u+u\cdot\nabla u),\\
&\partial_t B-\Delta B=B\cdot\nabla u-u\cdot\nabla B.
\end{aligned}\right.
\end{eqnarray*}
Then it is easy to observe that for $t\in[t_1,T^*)$
\begin{align}\label{mingming4+1}
&\|(u,B)\|_{\widetilde{L}^\infty([t_1,t];\dot{B}_{p,1}^{-1+\frac{3}{p}})}+\|(\Delta u,\Delta B,\nabla\Pi)\|_{L^1([t_1,t];\dot{B}_{p,1}^{-1+\frac{3}{p}})}\nonumber\\
\lesssim&\|(u(t_1),B(t_1))\|_{\dot{B}_{p,1}^{-1+\frac{3}{p}}}
+\Big\|\frac{a}{1+a}(\partial_t u+u\cdot\nabla u)\Big\|_{L^1([t_1,t];\dot{B}_{p,1}^{-1+\frac{3}{p}})}\nonumber\\
&+\|(u\cdot\nabla u,B\cdot\nabla B,B\cdot\nabla u,u\cdot\nabla B)\|_{L^1([t_1,t];\dot{B}_{p,1}^{-1+\frac{3}{p}})}.
\end{align}

By the product law in Besov spaces gives
\begin{align*}%\label{Q55}
\Big\|\frac{a}{1+a}(\partial_tu+u\cdot\nabla u)\Big\|_{L^1([t_1,t];\dot{B}_{p,1}^{-1+\frac{3}{p}})}\lesssim \big(1+\|a\|_{L^\infty([t_1,t];\dot{B}_{q,1}^{\frac{3}{q}})}\big)
\|\partial_tu+u\cdot\nabla u\|_{L^1([t_1,t];\dot{B}_{p,1}^{-1+\frac{3}{p}})}.
\end{align*}
Yet thanks to Lemma \ref{bernstein} and (\ref{Q31}), one has
\begin{align*}%\label{Q56}
\|\partial_tu\|_{L^1([t_1,t];\dot{B}_{p,1}^{-1+\frac{3}{p}})}\le C(t^{\frac12}\|\partial_tw\|_{L^2([t_1,t];{H}^1)}+\|u(t_1)\|_{\dot{B}_{p,1}^{-1+\frac{3}{p}}}  )\le C(1+t^{\frac12})
\end{align*}
and
\begin{align}\label{Q57}
&\|u\cdot\nabla u\|_{L^1([t_1,t];\dot{B}_{p,1}^{-1+\frac{3}{p}})}+\|{B}\cdot\nabla {B}\|_{L^1([t_1,t];\dot{B}_{p,1}^{-1+\frac{3}{p}})}
+\|u\cdot\nabla B\|_{L^1([t_1,t];\dot{B}_{p,1}^{-1+\frac{3}{p}})}+\|{B}\cdot\nabla u\|_{L^1([t_1,t];\dot{B}_{p,1}^{-1+\frac{3}{p}})}\nonumber\\
\le& C \int_{t_1}^t (\|\nabla w\|_{L^2}\|\Delta w\|_{L^2} +\|v\|^2_{\dot{B}_{p,1}^{\frac3p}}) dt'+C \int_{t_1}^t (\|\nabla b\|_{L^2}\|\Delta b\|_{L^2} +\|{h}\|^2_{\dot{B}_{p,1}^{\frac3p}}) dt'
\le C.
\end{align}
Thanks to Theorem 2.87 in \cite{bcd} and Proposition \ref{shuyun}, we have
\begin{align}\label{Q58}
\Big\|\frac{a}{1+a}\Big\|_{\widetilde{L}^\infty([t_1,t];\dot{B}_{q,1}^{\frac3q})}\le C\|a\|_{\widetilde{L}^\infty([t_1,t];\dot{B}_{q,1}^{{\frac3q}})}\le C \|a(t_1)\|_{\dot{B}_{q,1}^{\frac3q}}\exp\{ C  \int_{t_1}^t \|u(t')\|_{\dot{B}_{ 6,1}^{\frac32}} dt'\}.
\end{align}
By Lemma \ref{bernstein}, we have
\begin{align*}%\label{Q59}
\|u\|_{\dot{B}_{ 6,1}^{\frac32}}\le C \|\nabla w\|_{L^6}^{\frac12} \|\nabla^2 w\|_{L^6}^{\frac12}+C\|v\|_{\dot{B}_{ p,1}^{1+\frac3p}},
\end{align*}
which along with (\ref{Q19})  and (\ref{Q31})  implies
\begin{align*}%\label{Q60}
&\|\nabla u\|_{L^1([t_1,t];L^\infty)}+\|u\|_{L^1([t_1,t];\dot{B}_{6,1}^{\frac32})}\le C\|u\|_{L^1([t_1,t];\dot{B}_{6,1}^{\frac32})}\nonumber\\
\le& C\|v\|_{L^1([t_1,t];\dot{B}_{p,1}^{1+\frac{3}{p}})}+t^{\frac12} \|\Delta w\|_{L^2([t_1,t];L^2)}^{\frac12}\|\Delta w\|_{L^2([t_1,t];L^6)}^{\frac12}\nonumber\\
\le& C (1+t^{\frac12}).
\end{align*}
Therefore, we obtain
\begin{align}\label{Q61}
\Big\|\frac{a}{1+a}(\partial_tu+u\cdot\nabla u)\Big\|_{L^1([t_1,t];\dot{B}_{p,1}^{-1+\frac{3}{p}})}\le  C \|a(t_1)\|_{\dot{B}_{q,1}^{\frac3q}}\exp\{Ct^{\frac12}\}.
\end{align}
 Taking estimates (\ref{Q57}), (\ref{Q61}) into (\ref{mingming4+1}) and   applying Gronwall's inequality, one can finally get
\begin{align}\label{Q62}
&\|(u,B)\|_{\widetilde{L}^\infty([t_1,t];\dot{B}_{p,1}^{-1+\frac{3}{p}})}
+\|(u,B)\|_{L^1([t_1,t];\dot{B}_{p,1}^{1+\frac{3}{p}})}
+\|\nabla \Pi\|_{L^1([t_1,t];\dot{B}_{p,1}^{-1+\frac{3}{p}})}\nonumber\\
\le& C( \|a_0\|_{{B}_{q,1}^{\frac{3}{q}}}+\|(u_0,B_0)\|_{\dot{B}_{p,1}^{-1+\frac{3}{p}}})\exp\{Ct^{\frac12}\},
\end{align}
from which and \eqref{Q58}, we  can complete the proof of  Theorem \ref{zhuyaodingli} by a standard argument.

\vspace*{2em} \noindent\textbf{Acknowledgements} This work was
partially supported by NNSFC (No.11271382), RFDP (No.
20120171110014), MSTDF (No. 098/2013/A3), and Guangdong Special Support Program (No. 8-2015).


\bigskip
\bigskip
\bigskip

\noindent{\large\bf  References}


\begin{thebibliography}{99}

\bibitem{abidi2007+1}
H.~Abidi.
\newblock {\'E}quation de {N}avier-{S}tokes avec densit{\'e} et viscosit{\'e}
  variables dans l'espace critique.
\newblock {\it Rev. Mat. Iberoam}, {\bf 23}(2):537--586, 2007.
\newblock $\,$

\bibitem{abidi2012}
H.~Abidi, G.~Gui and P.~Zhang.
\newblock On the well-posedness of 3-{D} inhomogeneous {N}avier-{S}tokes
  equations in the critical spaces.
\newblock {\it Arch. Ration. Mech. Anal.}, {\bf 204}(1):189--230, 2012.
\newblock $\,$
\bibitem{abidi2013}
H.~Abidi, G.~Gui and P.~Zhang.
\newblock Well-posedeness of 3D inhomogeneous Navier-Stokes equations with highly oscillatory velocity field.
\newblock {\it J. Math. Pures Appl.}, {\bf 100}(1):166--203, 2013.
\newblock $\,$
\bibitem{abidi2007}
H.~Abidi and M.~Paicu.
\newblock Existence globale pour un fluide inhomog{\'e}ne.
\newblock {\it Ann. Inst. Fourier (Grenoble)}, {\bf 57}(3):883--917, 2007.
\newblock $\,$
\bibitem{abidi2008}
H.~Abidi and M.~Paicu.
\newblock Global existence for the magnetohydrodynamic system in critical
  space.
\newblock {\it Proc. Roy. Soc. Edinburgh Sect. A}, {\bf 138}(3):447--476, 2008.
\newblock $\,$






\bibitem{bcd}
H.~Bahouri, J.~Y. Chemin and R.~Danchin.
\newblock {\it {F}ourier {A}nalysis and {N}onlinear {P}artial {D}ifferential
  {E}quations}.
\newblock Grundlehren Math. Wiss. , vol. {\textbf{343}}, Springer-Verlag,
  Berlin, Heidelberg, 2011.
\newblock $\,$

\bibitem{bony}
J.M. Bony.
\newblock Calcul symbolique et propagation des singularit{\'{e}}s pour les
  {\'{e}}quations aux d{\'{e}}riv{\'{e}}es partielles non lin{\'{e}}aires.
\newblock {\it Ann. Sci. {\'{E}}cole Norm. Sup.}, {\bf 14}(4):209--246, 1981.
\newblock $\,$

\bibitem{cao2010}
\newblock   C. Cao and J. Wu.
\newblock
{Two regularity criteria for the 3D MHD equations}.
\newblock  {\it J. Differential Equations}, {\bf 248} (9): 2263-2274, 2010.

\bibitem{caochongsheng}
C.~Cao and J.~Wu.
\newblock Global regularity for the 2{D} {MHD} equations with mixed partial
  dissipation and magnetic diffusion.
\newblock {\it Adv. Math.}, {\bf 226}(2):1803--1822, 2011.
\newblock $\,$




\bibitem{chemin2014}
J.Y.~Chemin, M. Paicu and P.~Zhang.
\newblock Global large solutions to 3D inhomogeneous Navier-Stokes system with one slow variable.
\newblock {\it J. Differential Equations}, {\bf 256}(12):223--252, 2014.
\newblock $\,$







\bibitem{chenqinglei}
\newblock   Q. Chen, C. Miao and Z. Zhang.
\newblock
{The Beale-Kato-Majda criterion for the 3D magneto-hydrodynamics equations}.
\newblock  {\it Comm. Math. Phys.}, {\bf275} (3):861-872, 2007.

\bibitem{chenqing}
Q.~Chen, Z.~Tan and Y.J. Wang.
\newblock Strong solutions to the incompressible magnetohydrodynamic equations.
\newblock {\it Math. Methods Appl. Sci.}, {\bf 34}(1):94--107, 2011.
\newblock $\,$

\bibitem{danchincpde2001}
R.~Danchin.
\newblock Local theory in critical spaces for compressible viscous and
  heat-conducting gases.
\newblock {\it Comm. Partial Differential Equations}, {\bf 26}(7-8):1183--1233,
  2001.
\newblock $\,$
\bibitem{danchin2003}
R.~Danchin.
\newblock Density-dependent incompressible viscous fluids in critical spaces.
\newblock {\it Proc. Roy. Soc. Edinburgh Sect. A}, {\bf 133}:1311--1334,
  2003.
\newblock $\,$
\bibitem{danchin2004}
R.~Danchin.
\newblock Local and global well-posedness results for flows of inhomogeneous
  viscous fluids.
\newblock {\it Adv. Differential Equations}, {\bf 9}(3-4):353--386, 2004.
\newblock $\,$

\bibitem{danchincpde}
R.~Danchin.
\newblock Well-posedness in critical spaces for barotropic viscous fluids with
  truly not constant density.
\newblock {\it Comm. Partial Differential Equations}, {\bf 32}(9):1373--1397,
  2007.
\newblock $\,$

\bibitem{danchin2010}
R.~Danchin.
\newblock On the well-posedness of the incompressible density-dependent Euler equations in the $L^p$ framework.
\newblock {\it J. Differential Equations}, {\bf 248}(8):2130--2170,
  2010.
\newblock $\,$

\bibitem{danchin2012}
R.~Danchin and P.B. Mucha.
\newblock A lagrangian approach for the incompressible {N}avier-{S}tokes
  equations with variable density.
\newblock {\it Comm. Pure Appl. Math.}, {\bf 65}(10):1458--1480, 2012.
\newblock $\,$

\bibitem{danchin2013}
R.~Danchin and P.B. Mucha.
\newblock Incompressible flows with piecewise constant density.
\newblock {\it Arch. Ration. Mech. Anal.}, {\bf 207}(3):991--1023, 2013.
\newblock $\,$

\bibitem{danchin2014}
R.~Danchin .
\newblock A Lagrangian approach for the compressible Navier-Stokes equations.
\newblock {\it Ann. Inst. Fourier (Grenoble)}, {\bf 64}(2):753--791, 2014.
\newblock $\,$


\bibitem{davidson}
P.A. Davidson.
\newblock {\it An Introduction to Magnetohydrodynamics}.
\newblock Cambridge University Press, Cambridge, 2001.
\newblock $\,$

\bibitem{desjardins}
B.~Desjardins and C.~Le Bris.
\newblock Remarks on a nonhomogeneous model of magnetohydrodynamics.
\newblock {\it Differential Integral Equations.}, {\bf 11}(3):377--394, 1988.
\newblock $\,$

\bibitem{duvaut}
G.~Duvaut and J.~L. Lions.
\newblock In{\'{e}}quations en thermo{\'{e}}lasticit{\'{e}} et
  magn{\'{e}}tohydrodynamique.
\newblock {\it Arch. Ration. Mech. Anal.}, {\bf 46}(4):241--279, 1972.
\newblock $\,$

\bibitem{gerbeau}
J.F. Gerbeau and C.~Le Bris.
\newblock Existence of solution for a density-dependent magnetohydrodynamic
  equation.
\newblock {\it Adv. Differential Equations.}, {\bf 2}(3):427--452, 1997.
\newblock $\,$

\bibitem{guiguilong}
G.~Gui.
\newblock Global well-posedness of the two-dimensional incompressible
  magnetohydrodynamics system withvariable density and electrical conductivity.
\newblock {\it J. Funct. Anal.}, {\bf 267}(5):1488--1539, 2014.
\newblock $\,$



\bibitem{hecheng3}
C.~He and X.~Xin.
\newblock Partial regularity of suitable weak solutions to the incompressible magnetohydrodynamic equations.
\newblock {\it J. Funct. Anal.}, {\bf 227}(1):113--152, 2007.
\newblock $\,$









\bibitem{huangxiangdi}
X.~Huang and Y.~Wang.
\newblock Global strong solution to the {2-D} nonhomogeneous incompressible
  {MHD} system.
\newblock {\it J. Differential Equations}, {\bf 254}(2):511--527, 2013.
\newblock $\,$

 \bibitem{huangjingchi}
J.~Huang, M.~Paicu and P.~Zhang.
\newblock Global solutions to 2-D inhomogeneous Navier-Stokes system with general velocity.
\newblock {\it J. Math. Pures Appl.}, {\bf 100}(1):806--831, 2013.
\newblock $\,$



\bibitem{ladyzenskaja}
O.A. Lady${\check{z}}$enskaja and V.A. Solonnikov.
\newblock The unique solvability of an initial-boundary value problem for
  viscous incompressible inhomogeneous fluids.
\newblock {\it Zap. Nau$\check{c}$n. Sem. Leningrad. Otdel. Mat. Inst. Steklov.
  (LOMI)}, {\bf 52}:52--109,218--219, 1975.
\newblock $\,$

\bibitem{lixiaoli}
X.~Li and D.~Wang.
\newblock Global strong solution to the three-dimensional
density-dependent incompressible magnetohydrodynamic flows.
\newblock {\it J. Differential Equations}, {\bf 251}(6):1580--1615, 2011.
\newblock $\,$

\bibitem{linfanghua2013}
  \newblock F. Lin, L. Xu and P. Zhang.
  \newblock
  Global small solutions to 2-D incompressible MHD system.
   \newblock {\it  J. Differential Equations}, {\bf 259}(10):5440-5485, 2015.
\newblock $\,$

\bibitem{linfanghua}
\newblock   F. Lin and P. Zhang.
\newblock
 Global small solutions to an MHD-type system: the three-dimensional case.
\newblock   {\it Comm. Pure Appl. Math.}, {\bf 67}(4): 531--580, 2014.


\bibitem{lions}
P.L. Lions.
\newblock {\it Mathematical Topics in Fluid Mechanics, vol. I: Incompressible
  Models}.
\newblock Oxford Lecture Ser. Math. Appl., vol. 3, Oxford University Press, New
  York,, 1996.
\newblock $\,$

\bibitem{miaochangxing}
C.~Miao and B.~Yuan.
\newblock On the well-posedness of the {C}auchy problem for an {MHD} system in
  {B}esov spaces.
\newblock {\it Math. Meth. Appl. Sci.}, {\bf 32}(1):53--76, 2009.
\newblock $\,$

\bibitem{paicu2012}
M.~Paicu and P.~Zhang.
\newblock Global solutions to the 3-{D} incompressible inhomogeneous
  {N}avier-{S}tokes system.
\newblock {\it J. Funct. Anal.}, {\bf 262}(8):3556--3584, 2012.
\newblock $\,$


\bibitem{paicu2013}
   \newblock M. Paicu, P. Zhang and  Z. Zhang.
    \newblock
    {Global unique solvability of inhomogeneous Navier-Stokes equations with bounded density}.
     \newblock {\it Commun. Partial Differential Equations}, {\bf 38}:1208-1234, 2013.









\bibitem{PD}
\newblock   R.V. Polovin and V.P. Demutski\v{\i}.
\newblock
{Fundamentals of Magnetohydrodynamics}.
\newblock {\it Consultants Bureau: New York, 1990.}

\bibitem{renxiaoxia}
  \newblock X. Ren, J. Wu, Z. Xiang and Z. Zhang.
  \newblock
  Global existence and decay of smooth solution for the 2-D MHD equations without magentic diffusion.
  \newblock {\it J. Funct. Anal.}, {\bf 267}(2):503-541, 2014.
\newblock $\,$

\bibitem{sermange}
M.~Sermange and R.~Temam.
\newblock Some mathematical questions related to the {MHD} equations.
\newblock {\it Comm. Pure Appl. Math.}, {\bf 36}(5):635--664, 1983.



\bibitem{simon}
J.~Simon.
\newblock {N}onhomogeneous viscous incompressible fluids: {E}xistence of
  velocity, density, and pressure.
\newblock {\it SIAM J. Math. Anal.}, {\bf 21}(5):1093--1117, 1990.
\newblock $\,$
\bibitem{xuhuan}
   \newblock H. Xu, Y. Li and X. Zhai.
    \newblock
{On the well-posedness of 2-D incompressible Navier-Stokes equations with variable viscosity in critical spaces}.
    \newblock {\it J. Differential Equations}, {\bf 260}(8):6604--6637, 2016.
\newblock $\,$
\bibitem{zhaixiaoping}
\newblock   X. Zhai, Y. Li and W. Yan.
\newblock
{Global well-posedness for the 3D incompressible in\-homo\-gene\-ous MHD system in the ciritical Besov spaces}.
\newblock {\it J. Math. Anal. Appl.}, {\bf 432}(1):179--195, 2015.
\newblock $\,$












\end{thebibliography}
\end{document}